\documentclass{amsart}
\usepackage{amsmath}
\usepackage{amssymb}
\usepackage{amsfonts}
\usepackage{graphicx}
\usepackage{texdraw}
\usepackage{graphpap}
\usepackage{wasysym}
\usepackage{chngcntr}
\usepackage{enumitem}
\usepackage[OT2,T1]{fontenc}

\counterwithout{figure}{section}

\input txdtools

\newtheorem{mth}{Theorem}

\newtheorem{thm}{Theorem}[section]
\newtheorem{cor}[thm]{Corollary}
\newtheorem{lem}[thm]{Lemma}
\newtheorem{prop}[thm]{Proposition}

\theoremstyle{definition}

\theoremstyle{remark}

\newtheorem*{rem}{Remark}
\newtheorem*{ex}{Example}

\numberwithin{equation}{section}

\newcommand{\diag}{X}
\newcommand{\el}{\ell}
\newcommand{\barg}{\gamma}
\renewcommand{\d}{{\partial}}
\newcommand{\dbar}{{\bar{\partial}}}
\newcommand{\spann}{{\operatorname{span}}}

\newcommand{\1}{{\chi}}
\newcommand{\Prob}{{\mathbb P}}
\newcommand{\Expe}{{\mathbb E}}
\newcommand{\C}{{\mathbb C}}

\newcommand{\calZ}{\mathcal{Z}}
\newcommand{\tr}{\operatorname{tr}}
\newcommand{\cl}{\operatorname{cl}}
\newcommand{\Ordo}{O}

\newcommand{\R}{{\mathbb R}}

\newcommand{\Z}{{\mathbb Z}}

\newcommand{\const}{{\textrm{const.}}}
\newcommand{\ti}{\textit{t.i.\,}}
\renewcommand{\L}{{\mathbb L}}
\newcommand{\bfR}{{\mathbf R}}

\newcommand{\hfun}{H}
\newcommand{\nbh}{\Omega}
\newcommand{\bulk}{\operatorname{Int}}
\newcommand{\Pol}{\mathcal{W}}
\newcommand{\bfk}{\mathbf{k}}
\newcommand{\Lap}{\Delta}

\newcommand\coity{{C \sb 0 \sp \infty}}

\newcommand{\calH}{{\mathcal H}}

\newcommand{\calC}{{\mathcal C}}

\newcommand{\calB}{{\mathcal B}}
\newcommand{\calN}{{\mathcal N}}
\newcommand{\calF}{{\mathcal F}}

\newcommand{\calM}{{\mathcal M}}
\newcommand{\calP}{{\mathcal P}}

\newcommand{\calW}{{\mathcal W}}

\newcommand{\eqpot}{\check{Q}}

\newcommand{\Ai}{\operatorname{Ai}}
\newcommand{\fii}{{\varphi}}
\newcommand{\config}{\Theta}

\newcommand{\bfK}{{\mathbf{K}}}
\newcommand{\erker}{F}
\newcommand{\Ham}{H}
\newcommand{\drop}{S}
\newcommand{\vt}{\vartheta}
\newcommand{\re}{\operatorname{Re}}
\newcommand{\im}{\operatorname{Im}}

\newcommand{\erfc}{\operatorname{erfc}}

\newcommand{\dist}{\operatorname{dist}}
\newcommand{\supp}{\operatorname{supp}}

\newcommand{\eps}{\varepsilon}
\newcommand{\rest}{J}

\def\lpar{\left (}
\def\rpar{\right )}
\def\labs{\left |}
\def\rabs{\right |}
\def\babs#1{\labs {#1} \rabs}

\begin{document}
\title[Ward's equation]
{Rescaling Ward identities in the random normal matrix model}

\author{Yacin Ameur}
\author{Nam-Gyu Kang}
\author{Nikolai Makarov}
\address{Yacin Ameur\\ Department of Mathematics\\ Faculty of Science\\ Lund University\\ P.O. Box 118\\ 221 00 Lund\\ Sweden.}
\email{yacin.ameur@maths.lth.se}
\address{Nam-Gyu Kang\\ Department of Mathematical Sciences\\ Seoul National University\\ San56-1 Shinrim-dong Kwanak-gu Seoul 151-747\\ South Korea.}
\email{nkang@snu.ac.kr}
\address{Nikolai Makarov\\ Department of Mathematics\\
California Institute of Technology\\
Pasadena\\ CA 91125\\
USA.}
\email{makarov@caltech.edu}

\subjclass[2010]{Primary: 60B20. Secondary: 60G55; 81T40; 30C40; 30D15; 35R09}

\keywords{Random normal matrix;  Universality; Ward's equation.}

\thanks{Nam-Gyu Kang was supported by Samsung Science and Technology Foundation (SSTF-BA1401-01).}

\begin{abstract} We study existence and universality of scaling limits for the eigenvalues of a random normal
matrix, in particular at points on the boundary of the spectrum. Our
approach
uses Ward's equation -- an integro-differential identity satisfied by the rescaled one-point function.
\end{abstract}

\maketitle

In random normal matrix theory, one studies normal matrices $M$ (i.e. $MM^*=M^*M$)
of some large order $n$ picked randomly
with respect to a probability measure of the form
\begin{equation}\label{hm}d\mu_n(M)=\frac 1 {\calZ_n}\, e^{\,-n\tr Q(M)}\, dM.\end{equation}
Here $dM$ is the surface measure on normal $n\times n$ matrices
inherited from $4n^2$-dimensional Lebesgue measure via
the natural embedding into $\C^{\, n^2}$, $Q(\zeta)$ is a suitable real-valued function defined on $\C$ ("large'' as $\zeta\to\infty$), and $\calZ_n$ is a normalizing constant;  $\tr Q(M)=\sum_1^n Q(\zeta_j)$ is
the usual trace of the matrix $Q(M)$, where $\zeta_j$ denote the eigenvalues.

If $Q$ is just defined on $\R$ and $dM$ is surface measure on the Hermitian matrices (i.e. $M^*=M$), one obtains
random Hermitian matrices.
The study of such eigenvalue ensembles, e.g. using
the technique of Riemann-Hilbert problems, has been an active area of research.

The
eigenvalues $\{\zeta_j\}_1^n$ of a normal matrix $M$, picked randomly with respect to the measure \eqref{hm},
form a random point process in the complex plane $\C$. The same point processes  also appear as
a special case of one-component plasma (or "OCP'') ensembles, where $\{\zeta_j\}_1^n$ has the interpretation of a system of repelling point charges subjected to the external 
field $n Q$. Here the factor $n$ is needed to ensure that the system stays in a finite portion of the plane as $n\to\infty$.

In either interpretation, one defines the "energy'' of the system $\{\zeta_j\}_1^n$ by
\begin{equation}\label{hadd}\Ham_n\left(\zeta_1,\ldots,\zeta_n\right)=\sum_{j\ne k}\log\frac 1 {\babs{\,\zeta_j-\zeta_k\,}}+n\sum_{j=1}^nQ\left(\zeta_j\right).\end{equation}

The joint distribution of the particles/eigenvalues then follows the
 the \textit{Boltzmann-Gibbs law},
\begin{equation}\label{E1.1}d \Prob_n(\zeta)=\frac 1 {Z_n}\,e^{\,-\,\Ham_n(\zeta)}\,d V_n(\zeta),\quad
\zeta=(\zeta_1,\ldots,\zeta_n)\in\C^n,\end{equation}
where $dV_n$ is Lebesgue measure in $\C^n$ divided by $\pi^n$, and
$Z_n=\int e^{\,-\,\Ham_n}\,d V_n$.
We can henceforth treat the system $\{\zeta_j\}_1^n$ simply as a random sample from the distribution
\eqref{E1.1}, having the interpretations as random eigenvalues or as point charges. An important feature of this process is that it is \textit{determinantal}.

As $n\to\infty$, the system $\{\zeta_j\}_1^n$ tends to occupy a certain set $S$ called the droplet.
More precisely, the empirical distribution $\frac 1 n\sum_1^n\delta_{\zeta_j}$
converges in a suitable sense \cite{HM} to the equilibrium measure given by weighted potential theory.
In addition, the fluctuations of the system about the equilibrium converges to the Gaussian field on $S$
with free boundary conditions. See \cite{AHM3}.

 In this paper we will study microscopic properties of the system, close to a point $p$ in the droplet; in particular, a boundary point. The figure below shows a random sample
from the classical Ginibre ensemble, in which $Q=\babs{\,\zeta\,}^{\,2}$ and $S=\left\{\,\zeta\,;\,\babs{\,\zeta\,}\le 1\,\right\}$.
In this case, the process $\{\zeta_j\}_1^n$ can alternatively be interpreted as the eigenvalues of an
$n\times n$ matrix whose entries are independent complex, centered Gaussian random variables of variance $1/n$,
see \cite{G}.

We rescale about the point $p=1$ by letting $z_j=\sqrt{n}\,\left(\zeta_j-1\right)$ and refer to the rescaled system
$\config_n=\{z_j\}_1^n$ as the \textit{free boundary Ginibre process}.

\begin{figure}[ht]
\begin{center}
\includegraphics[width=.75\textwidth]{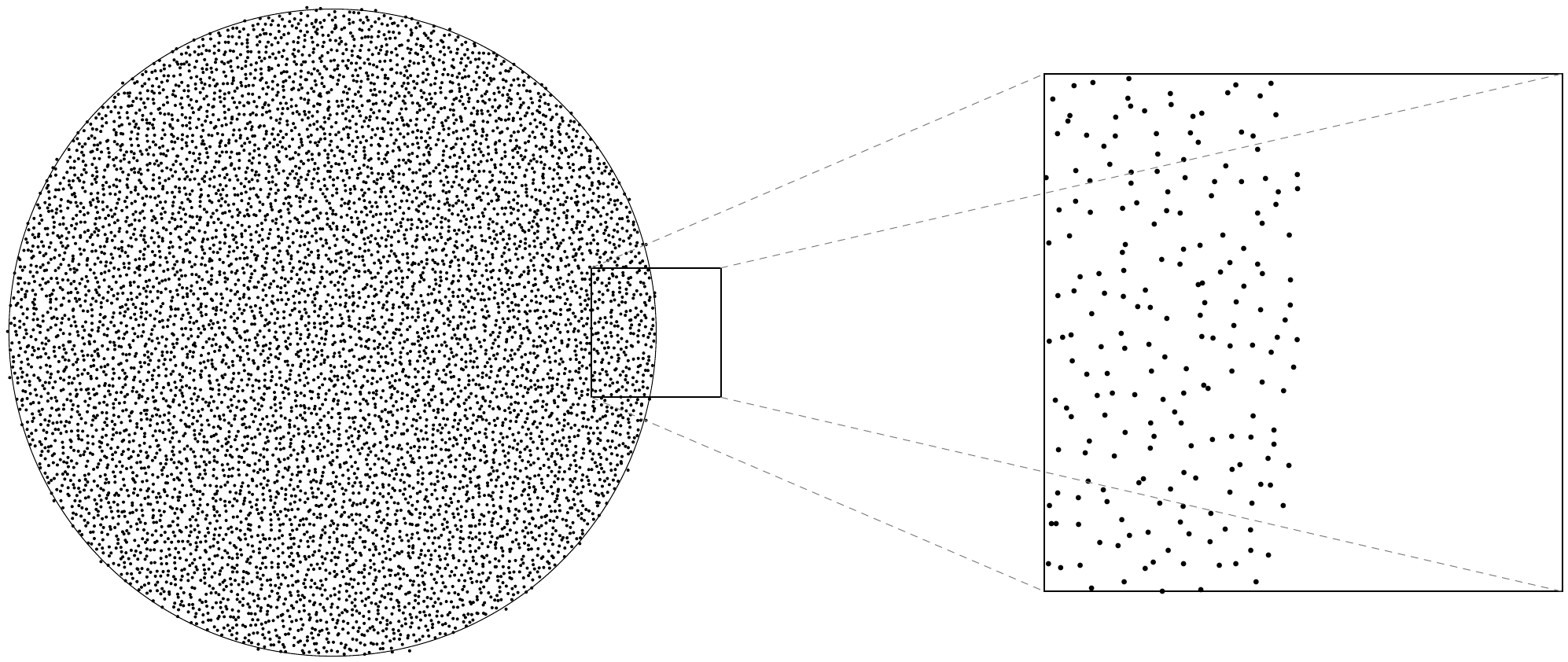}
\hspace{.05\textwidth}
\end{center}
\caption{A sample of the free boundary Ginibre process for a large value of $n$}
\end{figure}

The processes $\config_n$ converge as $n\to\infty$ to a
determinantal random
point field in $\C$ with correlation kernel
\begin{equation}\label{fhonn}K(z,w)=G(z,w)F\left(z+\bar{w}\right),\end{equation}
where we call $G(z,w):=e^{\,z\bar{w}\,-|\,z\,|^{\,2}/2\,-|\,w\,|^{\,2}/2}$ the \textit{Ginibre kernel} and $F(z):=\frac 1 2\erfc\left(\frac z{\sqrt{2}}\right)$ the \textit{free boundary plasma function}. To the best of our knowledge, this formula for the limiting point field first appeared in the paper \cite{FH} by Forrester and Honner;
a rigorous proof
was given in the paper \cite{BS}. (An alternative, simple argument depending on normal approximation of the Poisson distribution is given in Section \ref{pome}.)

We shall prove that the kernel \eqref{fhonn} appears "universally''
at regular points of the boundary, at least under an additional assumption of "translation invariance''.
This condition is satisfied e.g. when $Q$ is radially symmetric. In a certain sense,
\eqref{fhonn} is an analogue to the Airy kernel in Hermitian random matrix theory, i.e.,
\begin{equation}\label{airker}K(x,y)=\frac {\Ai(x)\Ai'(y)-\Ai'(x)\Ai(y)} {x-y},\end{equation}
which describes the eigenvalue spacing at the edge of the spectrum. See \cite{AGZ,BL,Fo,PS0} and the references there for further information.

We will
also establish existence and some basic properties of sequential limiting point fields pertaining to
a quite general $Q$ at an arbitrary point of the droplet. It is here useful to allow the point to vary with $n$, i.e., we can equally well zoom in on a "moving point''. This device is used in a separate paper \cite{AKMW} to study the distribution of eigenvalues near singular boundary points.


Our approach uses a relation between the $1$- and $2$-point functions, a particular case of Ward
identities for Boltzmann-Gibbs ensembles.  This relation is well-known in field theories \cite{KM} and has also been used in the papers \cite{AHM3,J} to study fluctuations of eigenvalues.

We here fix some point on the boundary and rescale Ward's identity about that point.
It turns out that, if this is done properly, then limiting  one- and two-point functions $R=R_1$ and $R_2$ can be defined in a way so that the
\textit{Berezin kernel}
$$B(z,w):=\frac {R(z)R(w)-R_2(z,w)}{R(z)}$$ satisfies the equation
\begin{equation}\label{thsh}\dbar C=R-1-\d\dbar \log R,\quad \text{where}\quad C(z):=\frac 1 \pi \int_\C \frac {B(z,w)}{z-w}\, d^2\,w.\end{equation}
We refer to the equation \eqref{thsh} as \textit{Ward's equation}. We stress that this equation is valid
at all (moving) points, provided that $\Lap Q$ does not vanish there and that $R$ does not vanish identically.

It is more or less immediate from the computations with the Ginibre ensemble that the
correlation kernel
\eqref{fhonn} gives rise to a solution to Ward's equation. However, in order to have \textit{uniqueness} of solution to Ward's equation, we
need to impose certain conditions, which depend on the nature of the point we are
rescaling about. Our verification of these conditions uses the formula for the
expectation of linear statistics from \cite{AHM3}, and some Bergman space techniques.

The method of rescaled Ward identities applies to some other situations as well. We shall for example consider \textit{hard edge} ensembles,
where we completely confine the system $\{\zeta_j\}_1^n$ to the droplet by setting $Q=+\infty$ outside of
$S$.
In this case, a new kernel arises at regular boundary points: the \textit{hard edge plasma kernel}
\begin{equation}\label{hfuncc}K(z,w)=G(z,w)H(z+\bar{w})\,\1_\L(z)\,\1_\L(w),\quad H(z)=\frac 1 {\sqrt{2\pi}}\int_{-\infty}^{\,0}\frac {e^{\,-\,\left(z-t\right)^{\,2}/2}}
{F(t)}\, dt,\end{equation}
where $\1_\L$ is the indicator function for the \textit{left} half plane $\L=\{\,z\,;\, \re z<0\,\}$,
and $F$ is the free boundary plasma function.
\begin{figure}[ht]
\begin{center}
\includegraphics[width=.45\textwidth]{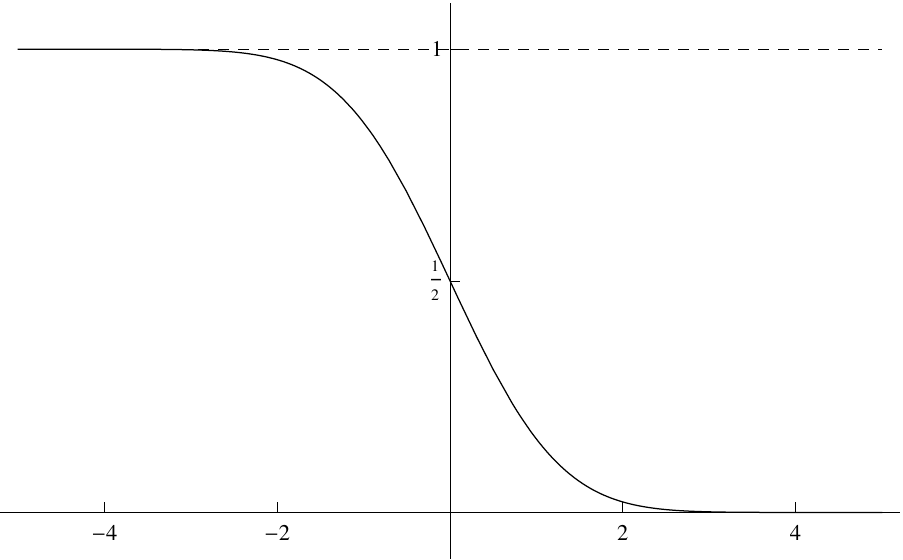}
\hspace{.05\textwidth}
\includegraphics[width=.45\textwidth]{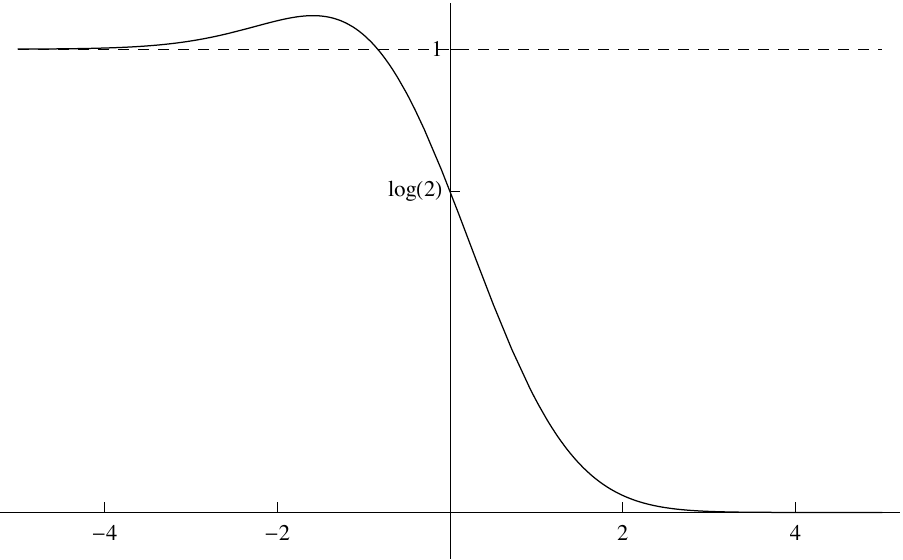}
\end{center}
\caption{The graphs of $F$ and $H$, restricted to the reals.}
\end{figure}

We will here give an elementary derivation of the formula for the case of the hard edge Ginibre ensemble. For reasons of length, it is convenient to postpone a complete treatment, including universality, to the companion
paper \cite{AKM2}. We remark that the hard edge kernel plays a role somewhat similar to the Bessel kernel in
Hermitian random matrix theory
\begin{equation*}K(x,y)=\frac {J_0(\sqrt{x}\,)\sqrt{y}\,J_0'(\sqrt{y}\,)-\sqrt{x}\,J_0'(\sqrt{x}\,)J_0(\sqrt{y}\,)}
{2(x-y)}.\end{equation*}
See e.g. \cite{FH,PS0,TW2} and the references there.

A detailed description of our results is given in the following section.

\subsection*{Notational conventions}
By $D(\zeta;r)$ we denote the open
disc with center $\zeta$ and radius $r$. We write $\d\omega$, $\bulk \omega$, $\cl\omega$, and $\omega^c$ for the
boundary, the interior, and the complement of a set $\omega\subset\C$.
The indicator
function of a set $E$ is denoted by $\1_E$.
We write $\d=\frac 1 2\left(\d/\d x-i\d/\d y\right)$ and $\dbar=\frac 1 2\left(\d/\d x+i\d/\d y\right)$ for the complex derivatives and
$\Lap=\d\dbar$ for the \textit{normalized} Laplacian. Thus $\Delta$ is $\frac 1 4$ times the standard Laplacian.
We write $dA(z)=d^{\,2} z/\pi$ for normalized Lebesgue measure. Thus the unit disc has measure $1$.
The volume measure
on $\C^k$ is defined by $dV_k(\zeta_1,\ldots,\zeta_k)=dA(\zeta_1)\cdots dA(\zeta_k)$.

A continuous function $f:\C^2\to\C$ is termed \textit{Hermitian} if $f(z,w)=\overline{f(w,z)}$. We shall
 say that $f$ is \textit{Hermitian-analytic}
 (or \textit{Hermitian-entire}) if $f$ is Hermitian and analytic (resp. entire) as a function of $z$ and $\bar{w}$. A Hermitian-entire function is uniquely determined by its diagonal values
 $f(z,z)$.

A Hermitian function $c$ is called a \textit{cocycle} if
\begin{equation*}\label{refco}c(z_1,z_2)c(z_2,z_3)\cdots c(z_{k-1},z_k)c(z_k,z_1)=1\end{equation*}
for all $k\ge 1$ and all sequences $(z_j)_1^k$. Alternatively, $c$ is a cocycle iff $c(z,w)=g(z)\overline{g(w)}$ for a continuous unimodular function $g$.

\section{Introduction and results}

\subsection{Potential theory and droplets}\label{subsec12} Fix a suitable function ("\textit{external potential}'') $Q:\C\to\R\cup\{+\infty\}$.
Let $\calP$ denote the class of positive, compactly supported Borel measures on $\C$.

Define
the \textit{weighted logarithmic energy} of $\mu\in\calP$ in external field $Q$ by
$$I_Q\left[\mu\right]=\iint_{\C^2}\log \frac 1 {\babs{\,\zeta-\eta\,}}\, d\mu(\zeta)\, d\mu(\eta)+\int_\C Q\, d\mu.$$
Thinking of $\mu$ as the distribution of an electric charge, $I_Q[\mu]$ is the sum of the Coulomb
interaction energy and the energy of interaction of $\mu$ with the external field $Q$.

We always assume that $Q$ is \textit{lower semi-continuous}, and that $Q$ is finite on
some set of positive logarithmic capacity. We will also assume that $Q$ is sufficiently
large at infinity, in the sense that
$$Q(\zeta)>>\log\babs{\,\zeta\,},\qquad (\zeta\to\infty).$$
To be precise, it will suffice to assume that $\liminf_{\zeta\to\infty}Q\left(\zeta\right)/\log\babs{\,\zeta\,}>2$.

A classical theorem of Frostman states that there then exists a unique \textit{equilibrium} measure
$\sigma$ which minimizes the weighted energy,
$$I_Q[\sigma]=\min_{\mu(\C)=1} I_Q[\mu],\qquad (\mu\in\calP).$$
See \cite{ST}.
We denote the compact support of the equilibrium measure by
$$S=S\left[Q\right]=\supp\sigma.$$
We refer to $S$ as the \textit{droplet} in the external field $Q$. It is known that if
$Q$ is smooth in some neighbourhood of $S$, then $\sigma$ is absolutely continuous and
takes the explicit form (see \cite{ST})
\begin{equation}\label{frostman}d\sigma=\1_S\cdot \Delta Q\,dA.\end{equation}
Since $\sigma$ is a probability measure, we have $\Lap Q\ge 0$ on $S$.

Our main assumptions \textit{throughout} are that there is some neighbourhood $\nbh$ of $S$ such that
\begin{enumerate}[label=(\alph*)]
\item \label{hubba} $Q$ is real analytic in $\nbh$,
\item \label{bubba} $\Lap Q>0$ in $\nbh$.
\end{enumerate}

With these assumptions, the complement $S^c$ has a local Schwarz function at each boundary point,
and we can rely on the fundamental theorem of Sakai \cite{Sa} concerning domains with local Schwarz functions.
In particular, we can apply Sakai's regularity theorem, which implies that all but
finitely many boundary points $p\in \d S$
are \textit{regular} in the sense that there is a disc $D=D(p;\epsilon)$ such that
$D\setminus S$ is a Jordan domain and $D\cap (\d S)$ is a simple real analytic arc.
A non-regular point $p\in\d S$ is called a \textit{singular} boundary point. Such points can be classified further as cusps or double points. We shall here study the regular case and refer to \cite{AKMW} for the singular case.





\subsection{Rescaling eigenvalue ensembles} \label{ree} Fix a potential $Q$ as above, and
let $(\zeta_j)_1^n$ denote a point in $\C^n$ picked randomly with respect to the measure $\Prob_n$
in \eqref{E1.1}.
We refer to the corresponding (unordered) configuration $\{\zeta_j\}_1^n$ as the \textit{$n$-point process} (or simply "system'')
associated to $Q$. We also speak of $\{\zeta_j\}_1^n$ as a "configuration picked randomly
with respect to $\Prob_n$''.

To each Borel set $B$ in $\C$ we associate a random variable $N_B$ by $N_B=\#\left\{\, j\,;\, \zeta_j\in B\,\right\}$.
The system $\{\zeta_j\}_1^n$ is then determined by the set of \textit{joint intensities} (a.k.a. "correlation functions'')
\begin{equation}\label{E1.2}\bfR_{n,k}\left(\eta_1,\ldots,\eta_k\right)=\lim_{\eps\downarrow 0}\frac {\Prob_n\lpar\, \bigcap_{j=1}^k \left\{\,N_{D( \eta_j;\eps)}\ge 1\,\right\}\,\rpar}
{\eps^{2k}},\qquad k=1,\ldots,n,\quad n\in\Z_+.\end{equation}

A few comments are in order.

The Hamiltonian $H(\zeta_1,\ldots,\zeta_n)$ is defined as $+\infty$ when $\zeta_j=\zeta_k$ for distinct $j$ and $k$.
It is then easy to see that the limit in \eqref{E1.2} exists, at least when the points $\eta_1,\ldots,\eta_k$ are away from eventual discontinuities of $Q$. Moreover, $\bfR_{n,k}(\eta_1,\ldots,\eta_k)=0$ when $\eta_i=\eta_j$ for some distinct indices $i,j\le k$.

The joint intensities are non-negative
symmetric functions, subject to the relations
\begin{equation}\label{bassrel}\int_\C \bfR_{n,1}\, dA=n ;\qquad
\int_\C
\bfR_{n,k+1}\left(\eta_1,\ldots,\eta_k,\eta_{k+1}\right)\, dA\left(\eta_{k+1}\right)=
\left(n-k\right)\,\bfR_{n,k}\left(\eta_1,\ldots,\eta_k\right).\end{equation}
We sometimes identify the intensity $\bfR_{n,k}$
with the measure $\bfR_{n,k}\,d V_k$.

According to Dyson's determinant formula
the joint intensities can be represented in the form
\begin{equation*}\label{E1.7}\bfR_{n,k}(\zeta_1,\ldots,\zeta_k)=\det\lpar \bfK_n(\zeta_i,\zeta_j)\rpar_{i,j=1}^k\end{equation*}
where $\bfK_n$ is a Hermitian function known as a \textit{correlation kernel} of the process. A proof of this formula, written for the case of Hermitian random matrices, is given e.g. in \cite{ST}, Section IV.7.2. A proof in the case of normal matrices
runs in the same way.

We are interested in microscopic properties of the system $\{\zeta_j\}_1^n$ near a point $p\in S$.
It is natural to magnify distances about $p$ by a factor $\sqrt{n\Lap Q(p)}$ and
fix an angle $\theta\in\R$. We shall
consider rescaled point processes of the form $\config_{n}=\config_{n}(p,\theta)=
\{z_j\}_1^n$, where
\begin{equation}\label{E1.3}z_j=e^{\,-i\theta}\sqrt{n\Delta Q(p)}\,\lpar \zeta_j-p\rpar,\quad j=1,\ldots,n.
\end{equation}
The \textit{law}
of $\config_n$ is defined as the
image of the Boltzmann-Gibbs measure \eqref{E1.1} under the map \eqref{E1.3}.
The rescaled system $\config_{n}$ then
has intensities denoted $R_{n,k}$, where
\begin{equation*}R_{n,k}\left(z_1,\ldots,z_k\right)=\frac 1 {\left(n\Lap Q(p)\right)^{k}}\,
\bfR_{n,k}\left(\zeta_1,\ldots,\zeta_k\right)
.\end{equation*}

We shall henceforth usually assume that $p$ is a regular boundary point.
We then throughout, by convention, take $e^{i\theta}$ as the \textit{outer normal} to $\d S$ at $p$.


The point process $\config_n$ is determinantal with kernel $K_n$ given by
\begin{equation*}\label{detkernel}K_n(z,w)=\frac 1 {n\Lap Q(p)}\bfK_n(\zeta,\eta),\quad \text{where}\quad
\begin{cases}\,z=e^{\,-i\theta}\sqrt{n\Lap Q(p)}\,\left(\zeta-p\right)\cr
w=e^{\,-i\theta}\sqrt{n\Lap Q(p)}\,\left(\eta-p\right)\cr
\end{cases},\end{equation*}

The fundamental problem is existence and uniqueness of a limiting determinantal point field
of the processes $\config_n$, as $n\to\infty$.
For our purposes, convergence will mean locally uniform convergence of all intensities $R_{n,k}$
to some limits $R_k$ as $n\to\infty$. Whenever this is the case, $R_k$ can be interpreted
in terms of Lenard's theory (see \cite{S}) as
a $k$-point function for a "point field'' in $\C$, meaning a probability
law on a suitable space of (perhaps) infinite configurations $\{z_i\}_1^\infty\subset\C$. A precise
definition, and a discussion of relevant convergence results, is given in the appendix.

It here suffices to note that the desired convergence of the processes $\config_n$ will hold if
the correlation kernels $K_n$ converge to a limit $K$ locally uniformly on $\C^2$. Moreover, the
limiting point field is uniquely determined by $K$ if the functions $K_n(z,z)$ are uniformly bounded, and it is then determinantal with
intensity functions $$R_k(z_1,\ldots,z_k)=\det\left(K(z_i,z_j)\right)_{i,j=1}^k.$$
More generally, if $K_n$ is a correlation kernel and $(c_n)_1^\infty$ is a sequence of cocycles,
then $c_n K_n$ is another kernel
giving rise to the same joint intensities $R_{n,k}$. The problem is thus to show that there exists a sequence
of cocycles such that $c_n K_n$ converges locally uniformly to a non-trivial limit $K$
with bounded convergence on the diagonal in $\C^2$.

The best understood case is when $p\in\bulk S$. Then, under the much weaker assumption that $Q$ is
$C^2$-smooth near $p$, the rescaled processes $\config_n$ converge to the \textit{Ginibre$(\infty)$ point
field}. The correlation kernel of this field is the Ginibre kernel, \begin{equation}\label{E1.9}G(z,w)=e^{\,z\bar{w}\,-\,\babs{\,z\,}^{\,2}/2\,-\,
\babs{\,w\,}^{\,2}/2}.\end{equation}
If $p\in\bulk S$, and if we rescale about $p$ according
to \eqref{E1.3} with arbitrary angle $\theta$, then we have
\begin{equation*}\label{thebull}R_{n,k}\left(z_1,\ldots,z_k\right)\to \det\lpar G\left(z_i,z_j\right)\rpar_{i,j=1}^k\qquad \text{as}\quad n\to\infty\end{equation*}
with locally uniform convergence on $\C^k$. (Cf. \cite{AHM2}; a new proof is given in Section \ref{thor}.)

In the following sections, we will state the main
results of the paper.

\subsection{Compactness, non-triviality, and Ward's equation} \label{ms1}
Our first theorem states the existence of sequential limits of the rescaled point processes $\config_n$
and specifies the form of limiting correlation kernels. With a view to later
applications \cite{AKMW} we shall adopt a somewhat broader point of view, rescaling about a \textit{moving point} $p=(p_n)_1^\infty$ where $p_n$ is a sequence of points in $S$. A constant sequence $p_n=p$ will be identified with the point $p$. In the following result we are also at liberty to choose a sequence of angles $\theta_n\in\R$ and rescale about $p$ according to
\begin{equation}\label{genresc}z_j=e^{-i\theta_n}\sqrt{n\Lap Q(p_n)}\left(\zeta_j-p_n\right).\end{equation}

\begin{mth} \label{TT1} Let $p\in\drop$ be a (possibly moving) point, and rescale about $p$ as above.
\begin{enumerate}[label=(\roman*)]
\item \label{TT1_1}
Compactness: There is a sequence $c_n$ of cocycles such that
every subsequence of $\left(c_nK_n\right)_1^\infty$ has a subsequence converging uniformly on compact subsets.
\item \label{TT1_2}
Analyticity: Each limit point in \ref{TT1_1} satisfies $K=G\Psi$ where $\Psi$ is a Hermitian entire function which satisfies the following "mass-one inequality'':
\begin{equation}\label{mocio}\int_\C e^{\,-\,|\,z-w\,|^{\,2}}\babs{\,\Psi(z,w)\,}^{\,2}\, dA(w)\le \Psi(z,z).\end{equation}
\end{enumerate}
\end{mth}

A limit point $K$ in Theorem \ref{TT1} will be called a \textit{limiting kernel}. It follows from \ref{TT1_1} that $K$ is a \textit{positive matrix} in Aronszajn's sense \cite{Ar}, i.e., for all finite sequences $(z_j)_1^N$ of points and all choices of scalars $(\alpha_j)_1^N$ we have
$$\sum_{j,k=1}^N\alpha_j\bar{\alpha}_k K(z_j,z_k)\ge 0.$$
Indeed, $K$ is a limit of $c_{n_k}K_{n_k}$ where $K_n$ are positive matrices and $c_n$ are cocycles.
Moreover, by the general theory mentioned in the previous section, a limiting kernel $K$ is the correlation kernel of a some point field in the plane,
which we call a \textit{limiting point field}.
The $1$-point function of this point field is denoted by $R(z)=K(z,z)$. If $R\ne 0$ on $\C$
we can define the Berezin kernel
\begin{equation*}\label{nott}B(z,w)=\frac{\left|\,K(z,w)\,\right|^{\,2}}{K(z,z)}
,\end{equation*}
and the "Cauchy transform''
\begin{equation}\label{nott2}C(z)=\int_\C\frac {B(z,w)}{z-w}\, dA(w).\end{equation}

\begin{mth} \label{TT1.5} Let $K=G\Psi$ be a limiting kernel in Theorem \ref{TT1} and write $R(z)=K(z,z)$. Let
$p\in S$ be a (possibly moving) point.
\begin{enumerate}[label=(\roman*)]
\item \label{TT1.5_1} Zero-one law: Either $R$ is trivial, in the sense that $R=0$ identically, or $R>0$ everywhere.
\item \label{TT1.5_2} Ward's equation: If $R$ is non-trivial, then
the integral $C(z)$ in \eqref{nott2} converges and defines a smooth function such that
\begin{equation*}\label{deal}\dbar C(z)=R(z)-1-\Lap\log R(z),\qquad z\in \C.\end{equation*}
\item \label{TT1.5_3} Complementarity: The "complementary kernel''
$$\tilde{K}(z,w):=G(z,w)(1-\Psi(z,w))$$
is a positive matrix. In particular, $R(z)=\Psi(z,z)\le 1$ for all $z$.
\end{enumerate}
\end{mth}

In connection with the zero-one law, we mention that triviality (i.e. $R=0$) occurs when one rescales about singular boundary points. See \cite{AKMW}.

As a word of warning, we mention that the complementary kernel does not in general solve Ward's equation.

By the preceding results, it is natural to try to find all (or at least some)
limiting kernels $K$ giving rise to a solution to Ward's equation. In order to fix a solution
uniquely, we need to know that certain additional conditions are satisfied, which depend
on the nature of the point we are rescaling about. For regular points, the results are as follows.

\begin{mth}\label{TT2} Fix a regular boundary point $p$ and let $K$ be a
corresponding limiting kernel in Theorem \ref{TT1}. Write $x=\re z$.
\begin{enumerate}[label=(\roman*)]
\item \label{TT2_1} Exterior estimate: There is a constant $C$ such that
\begin{equation*}R(z)\le C\,e^{\,-\,2\,x^{\,2}}\quad\text{for}\quad x\ge 0.
\end{equation*}
 \item Interior estimate: \label{TT2_2} If $\el$ is any number with $\el<1/2$ then there is a constant $C$ (perhaps depending on $\el$) such that
\begin{equation*}
\babs{\,R(z)-1\,}\le C\,e^{\,-\,\el\, x^{\,2}}\quad \text{for}\quad x\le 0.
\end{equation*}
\end{enumerate}
\end{mth}

We believe that the interior estimate (ii) should be true with $\el=2$. However, the exact value of $\el$ is not important in the sequel.

\begin{mth}\label{TT2.25} Suppose that the droplet $S$ is connected with everywhere
regular boundary $\d S$. Then there is a subset $N\subset \d S$ of measure zero for the arclength such that if $K$ is a limiting rescaled kernel about a point $p\in(\d S)\setminus N$, then $R(x)=K(x,x)$ satisfies
\begin{equation*}\int_{-\infty}^{+\infty}x\cdot(R(x)-\1_{(-\infty,0)}(x))\, dx=\frac 1 8.
\end{equation*}
\end{mth}

In particular, the foregoing results imply that any limiting kernel
at a regular boundary point is non-trivial, i.e. $R(z)>0$ for all $z$.
Moreover, as $R$ is bounded, we can conclude that all limiting
point fields are uniquely determined by the corresponding limiting kernel.

The topological assumptions on the droplet in Theorem \ref{TT2.25} are made in order to be able to use the fluctuation theorem from \cite{AHM3}; the theorem should be true without this assumption.


\subsection{Translation invariant solutions to Ward's equation} Let $K=G\Psi$ be a limiting kernel
in Theorem \ref{TT1}.
We say that $\Psi$ (or $K=G\Psi$) is (vertical) \textit{translation invariant}
(in short: \ti) if
\begin{equation*}\label{mack}\Psi\left(z+it,w+it\right)=\Psi\left(z,w\right),\qquad t\in\R.\end{equation*}
In this case, $\Psi$ can be represented in the form
\begin{equation*}\label{chack}\Psi(z,w)=\Phi(z+\bar{w})\end{equation*}
for some entire function $\Phi$.
To study kernels of this form, it is convenient to introduce the "Gaussian kernel''
\begin{equation*}\label{gau}\barg(z):=\frac 1 {\sqrt{2\pi}}\,e^{\,-\,z^{\,2}/2},\qquad z\in\C.\end{equation*}
Let $g$ be a bounded function on $\R$.
We shall use the symbol "$*$'' to denote the convolution operation
$$\barg*g(z):=\int_{-\infty}^{+\infty} \barg\left(z-t\right)g(t)\, dt,\qquad z\in\C.$$
Thus $\barg*g$ has the meaning of ordinary convolution in $\R$ followed
by analytic continuation. The function $z\mapsto (\barg * g)(z)\cdot e^{z^2/4}$ is a version of the \textit{Bargmann transform} of $g$ (see e.g. \cite{Fo}).
Note that the free boundary plasma function $\erker$ can be written
\begin{equation}\label{ffuncc}\erker(z)=\barg*\1_{(-\infty,0)}(z)=\int_{-\infty}^{\,0}\barg\left(z-t\right)\, dt.\end{equation}

The convolution operation is intimately connected with \ti limiting kernels. Indeed, if $K(z,w)=G(z,w)\Phi(z+\bar{w})$ is an arbitrary \ti
limiting kernel, then we shall prove that there exists a Borel function $f$ with $0\le f\le 1$ such that $\Phi=\gamma *f$.
Our proof of this fact relies on Bochner's theorem on positive definite functions; see Section \ref{transmetr}.

We will refer to the formula $\Phi=\gamma*f$ as the \textit{Gaussian representation} of the \ti limiting kernel $K$. Our next theorem says that
$f$ is necessarily a characteristic function.


\begin{mth} \label{TT3}
Suppose that $\Phi$ is an entire function of the form $\Phi=\gamma *f$ where
$f$ is a bounded function. Then
the kernel $K(z,w)=G(z,w)\Phi(z+\bar{w})$
gives rise to a solution to Ward's equation if and only if there
is a connected set $I\subset \R$ of positive measure such that
$\Phi=\barg *\1_I$.
\end{mth}

 By Theorem \ref{TT2}, a \ti limiting kernel $K(z,w)=G(z,w)\Phi(z+\bar{w})$ must be of the form $\Phi=\gamma*\1_{(-\infty,a)}$ for some $a\in\R$. It is easy to show that the only choice of $a$ consistent with the $1/8$-formula in Theorem \ref{TT2.25} is $a=0$. See Section \ref{tir} for details. Thus we will have $\Phi=F$ at almost every boundary point, provided e.g. that the droplet $S$ is connected and having everywhere regular boundary.

In general, we do not know whether all limiting kernels necessarily are \ti It is however easy to show that this is the case when the potential
is radially symmetric, i.e., when $Q(z)=Q(|z|)$.




\begin{mth} \label{TT4} Suppose that
$Q$ is radially symmetric and that the droplet $S$ is connected. For a boundary point $p$ we let $\config_n$ be the free boundary process
rescaled about $p$ in the outer normal direction.
Then $\config_n$ converges to the unique point field $BG$
in $\C$ with the
correlation kernel
\begin{equation}\label{E1.12}K(z,w)=G(z,w)\cdot \erker\left(z+\bar{w}\right).\end{equation}
\end{mth}

The topological assumption means that the droplet is either a disc or an annulus. This restriction should be regarded as a technical assumption made in order to apply Theorem \ref{TT2.25}.

We will call the kernel $K$ in \eqref{E1.12}, resp. the point field $BG$, the "boundary Ginibre''
kernel/point field. It is a natural conjecture that $BG$ is universal at regular boundary points
for all potentials. After this note was completed, the convergence was verified for the non-symmetric potentials
("ellipse ensembles'') $Q=|\,\zeta\,|^{\, 2}-t\re(\zeta^{\,2})$, $0<t<1$. See \cite{LR}. The verification depends on explicit computations with the orthogonal polynomials, which were obtained in the thesis \cite{Ri}.


\subsection{Berezin kernel and mass-one equation} \label{bkamo} Let $B_n$ be the Berezin kernels of the rescaled
processes $\config_n$, i.e.,
\begin{equation*}\label{E1.4}B_n(z,w)=\frac {R_{n,1}(z)R_{n,1}(w)-
R_{n,2}(z,w)}
{R_{n,1}(z)}=\frac {\babs{\,K_n(z,w)\,}^{\,2}}{K_n(z,z)}.\end{equation*}
We have (e.g. by \eqref{bassrel})
\begin{equation}\label{mocc1}\int_\C B_n(z,w)\, dA(w)=1.\end{equation}

\begin{figure}[ht]
\begin{center}
\includegraphics[width=.45\textwidth]{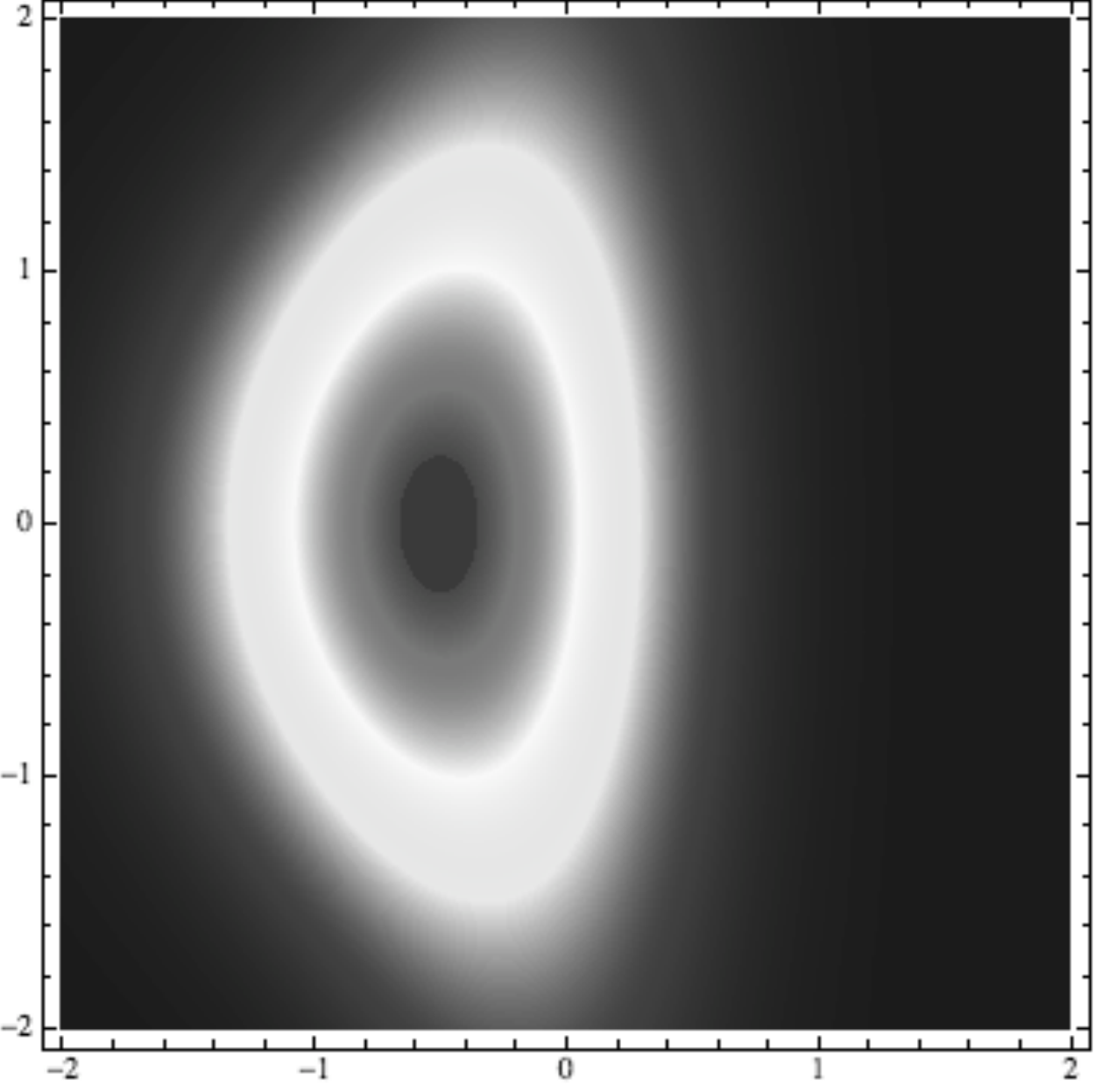}
\hspace{.05\textwidth}
\includegraphics[width=.45\textwidth]{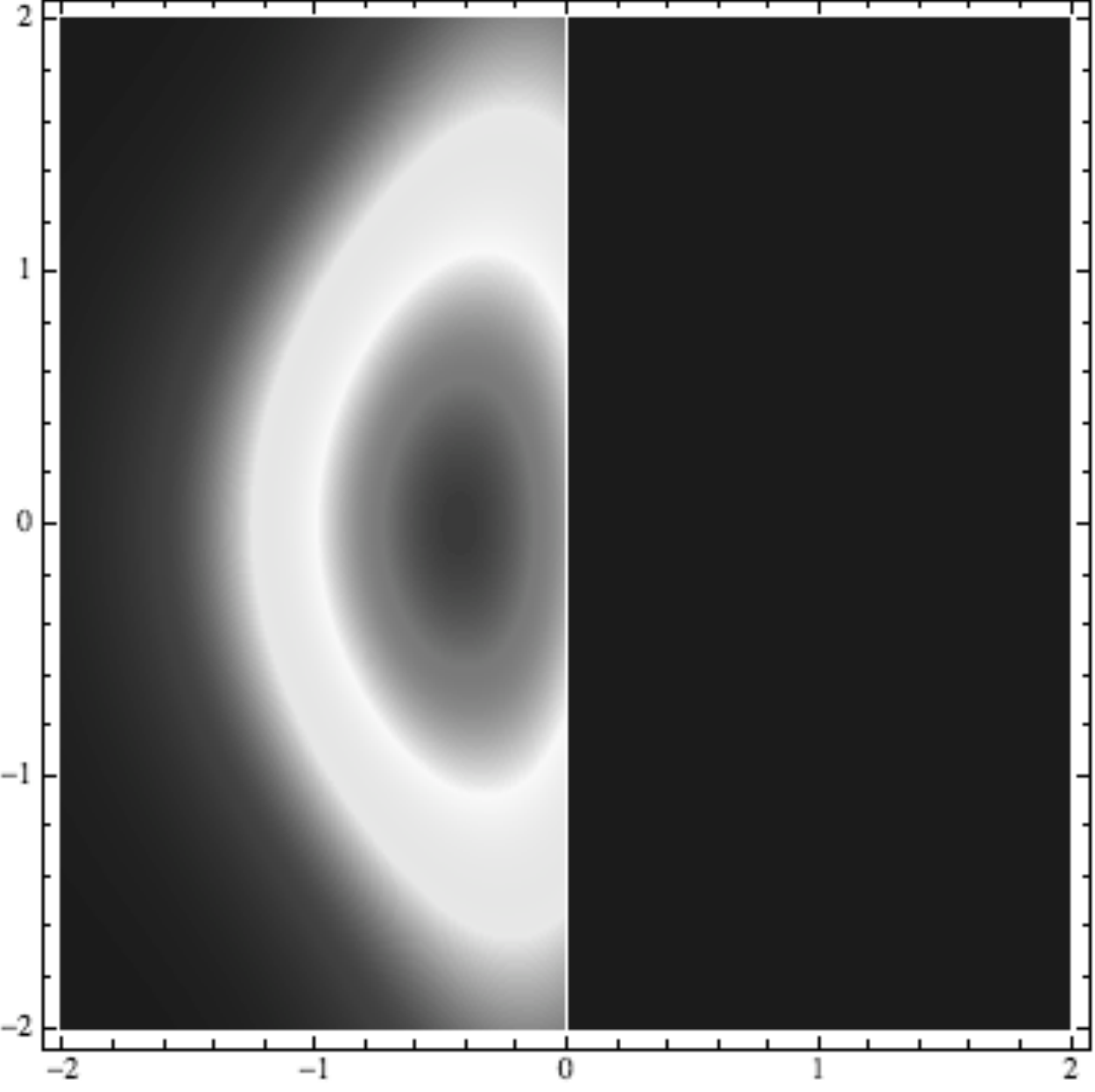}
\end{center}
\caption{The
Berezin kernel $B_n(0,w)$ for the free boundary Ginibre ensemble, and its hard edge counterpart, for a large
value of $n$.}
\label{fig1.5}
\end{figure}

Let us write $R_n=R_{n,1}$ and denote by $R_{n-1}^{\,(a)}$ the $1$-point function for the conditional $(n-1)$-point process $\config_n|\left\{a\in\config_n\right\}$.
For each fixed $a\in\C$, we then have \begin{equation}\label{wsoo}B_n(a,w)=R_{n}(w)-R_{n-1}^{\,(a)}(w).\end{equation}
(For details, see \cite{AHM2}, Section 7.6.)

It is natural to ask whether the relation \eqref{mocc1} also holds for the limiting kernel $B=\lim_n B_n$.

Figure 4 shows the Berezin kernel corresponding to the boundary Ginibre ensembles. In this case, direct computations show that $\int B(a,w)\, dA(w)=1$ and so we have
\begin{equation*}R^{\,(a)}(z)=R(z)-B(a,z)=F(2x)-e^{\,-\,\babs{\,z-a\,}^{\,2}}\babs{\,F(a+z)\,}^{\,2}/F(2a),\qquad z=x+iy,\end{equation*}
The corresponding conditional $1$-point function rescaled about a bulk point is $z\mapsto 1-e^{\,-\,\babs{\,z-a\,}^{\,2}}$.

Elementary calculations now show that if we rescale about
the point $p$ and insert a point charge at $a=0$, then the repulsion caused
if $p$ is in the bulk is stronger, by a factor
$\pi/(\pi-2)\approx 3$, compared to if the point $p$ is on the boundary.


Let us say that a limiting kernel $K=G\Psi$ satisfies the
\textit{mass-one equation} if the corresponding Berezin kernel $B$ satisfies
$\int_\C B(z,w)\, dA(w)=1$ for all $z$, i.e.,
\begin{equation*}\label{mass-one}\int_\C e^{\,-\,\babs{\,z-w\,}^{\,2}}\babs{\,\Psi(z,w)\,}^{\,2}
\, dA(w)=\Psi(z,z),\qquad z\in\C.\end{equation*}
This equation is technically similar to Ward's equation and it has a simple spectral interpretation
(see Section \ref{good}). Every kernel $K$ giving rise to a solution to the mass-one equation is furthermore
a correlation kernel of some point field.

In view of the Gaussian representability of \ti limiting kernels, it is natural to study solutions
to the mass-one equation of the form $\Phi=\gamma *f$ where $f$ is a bounded function
on $\R$. Of course, such solutions do not necessarily correspond to actual limiting kernels.

The following theorem describes all solutions of this form.

\begin{mth}\label{TT5} Let $\Phi=\gamma *f$ where $f$ is a bounded Borel function on $\R$,
and put $\Psi(z,w)=
\Phi(z+\bar{w})$.
Then $\Psi$ gives rise to a solution to the mass-one equation if and only if there is a Borel set
$e\subset\R$ of positive measure such that $f=\1_e$.
\end{mth}

A noteworthy consequence of our theory is that all non-trivial \ti limiting kernels satisfy the mass-one equation. This follows by combining theorems \ref{TT1.5}, \ref{TT3}, and \ref{TT5}.

\subsection{Organization of the paper} In Section \ref{pome} we consider the boundary Ginibre ensembles, both
for the free boundary and the hard edge.
We give a short proof of the convergence of rescaled ensembles to the boundary Ginibre point fields
with kernels \eqref{fhonn} and \eqref{hfuncc}, respectively.

In Section \ref{parti}, we prove Theorem \ref{TT1} (compactness and analyticity).

In Section \ref{WEFA} we derive Ward's equation and prove Theorem \ref{TT1.5}.

In Section \ref{RBP}, we establish a priori bounds for regular points
(Theorem \ref{TT2}). We also prove the $1/8$-formula in Theorem \ref{TT2.25}
at almost every
boundary point.

In Section \ref{tis} we specialize to \ti solutions and prove theorems
\ref{TT3}, \ref{TT4}, and \ref{TT5}.


The last section, Section \ref{concrem}, contains various concluding remarks. In particular, we comment
on the nature of the mass-one and Ward equations, we show that these equations take the form of twisted convolution equations, we discuss Hilbert spaces of entire functions associated to limiting kernels,
and write down versions of Ward's equation in some settings which are different from the free boundary
case studied in this paper (hard edge, bulk singularities, $\beta$-ensembles).

\section{Example: The Ginibre ensembles} \label{pome}

\subsection{Principles of notation} \label{wotan} Consider first a general potential $Q$.
By a \textit{weighted polynomial of order $n$} we mean
a function of the form $f=q\cdot e^{-nQ/2}$ where $q$ is an (analytic)
polynomial of degree at most $n-1$. Let $\Pol_n$ denote the space of all weighted polynomials
of order $n$, considered as a subspace of $L^2=L^2(\C,dA)$. It is well-known that the reproducing
kernel $\bfK_n(\zeta,\eta)$ for the space $\Pol_n$ is a correlation kernel for the process
$\{\zeta_j\}_1^n$ corresponding to $Q$.
This implies that one has the formula
\begin{equation}\label{E1.8}\mathbf{K}_n(\zeta,\eta)=\sum_{j=0}^{n-1}
q_j(\zeta)\,\bar{q}_j(\eta)\,e^{\,-nQ(\zeta)/2-nQ(\eta)/2},\end{equation}
where $q_j$ is the $j$:th orthonormal polynomial with respect to the measure $e^{-nQ(\zeta)}dA(\zeta)$.

 Recall the Ginibre potential $Q=\left|\,\zeta\,\right|^{\,2}$. The corresponding droplet is $S=\left\{\,\zeta\,;\,\left|\,\zeta\,\right|\le 1\,\right\}$.

We shall give an elementary proof for $BG$-convergence using Poisson approximation of the normal distribution.
Our proof is quite similar to
the argument in the paper \cite{R}, where the spectral radius of a Ginibre matrix is studied.
Let us mention also that there are several other proofs of BG convergence
for the free boundary Ginibre ensemble.

\subsection{Free boundary Ginibre ensemble} \label{pome1}
 Let $\{\zeta_j\}_1^n$ denote a random configuration for the free boundary Ginibre
process.
We rescale
about the boundary point $p=1$ in the outer normal direction, via $z_j=\sqrt{n}\left(\zeta_j-1\right)$, writing $\config_n=\{z_j\}_1^n$ for the rescaled
process. Let $G$ and $F$ be the Ginibre kernel and the free boundary plasma function, respectively. See
\eqref{E1.9} and \eqref{ffuncc}.
 We shall prove the following theorem, found in \cite{FH} (cf. \cite{BS}).

\begin{thm} \label{ginibref} There exists a unique point field $BG$ with correlation kernel $K(z,w)=G(z,w)F(z+\bar{w})$, and
the processes $\config_n$ converge to $BG$ as $n\to\infty$ with locally uniform convergence of intensity
functions.
\end{thm}

Since $K(z,z)<1$, it suffices to prove the statement about convergence of intensity
functions.

By \eqref{E1.8}, a correlation kernel for the process $\{\zeta_j\}_1^n$ is computed to
\begin{equation}\label{corro}\bfK_n(\zeta,\eta)=n\sum_{j=0}^{n-1}\frac {\lpar n\zeta\bar{\eta}\rpar^j} {j!}e^{\,-n\frac {\babs{\,\zeta\,}^{\,2}+
\babs{\,\eta\,}^{\,2}} 2}.\end{equation}
Now rescale according to
\begin{equation*}\zeta= 1+z/\sqrt{n}\quad ,\quad \eta= 1+w/\sqrt{n}.\end{equation*}
and note that the rescaled process $\config_n$ has correlation kernel
$K_{n}(z,w)=\frac 1 n \bfK_n\left(\zeta,\eta\right).$
Using \eqref{corro}, we write $K_n$ in the form
\begin{equation*}K_n(z,w)=\sum_{j=0}^{n-1}\lpar \frac {n\zeta\bar{\eta}} \lambda\rpar^j\frac {\lambda^j}{j!}e^{\,-\lambda}\end{equation*}
where
\begin{equation}\label{lambdan}\lambda=\lambda(n)=\frac n 2\lpar|\,\zeta\,|^{\,2}+|\,\eta\,|^{\,2}\rpar=n+\sqrt{n}\re(z+w)+\frac 1 2\left(|\,z\,|^{\,2}+|\,w\,|^{\,2}\right).\end{equation}
We next let $X_n$ be a Poisson distributed random variable with intensity $\lambda=\lambda(n)$
(in short: $X_n\sim \text{Po}(\lambda)$), \textit{i.e.},
$$\Prob\left\{ X_n=k\right\}=\frac {\lambda^k} {k!} e^{\,-\,\lambda},\quad k=0,1,\ldots.$$
We then have the identity
$$K_n(z,w)=\Expe\left[\lpar \frac {n\zeta\bar{\eta}} \lambda\rpar^{X_n}\cdot
\1_{\left\{X_n<n\right\}}\right].$$
Now introduce a new random variable $Y_n$ by
$$X_n=\lambda+\sqrt{\lambda}\, Y_n.$$
By the central limit theorem, $Y_n$ converges in distribution to the standard normal as $n\to\infty$; the convergence is moreover uniform. (This is the well-known "normal approximation of the Poisson distribution''; uniform convergence follows e.g. by the Berry-Esseen theorem.)

Now factorize $K_n(z,w)$ in the following way,
\begin{equation}\label{faan}K_n(z,w)=A_n\cdot B_n:=\lpar \frac {n\zeta\bar{\eta}}\lambda\rpar^\lambda\cdot
\Expe\left[\lpar \frac {n\zeta\bar{\eta}}\lambda\rpar^{\sqrt{\lambda}\cdot Y_n}
\cdot \1_{\left\{Y_n<\alpha_n\right\}}\right],\end{equation}
where
\begin{equation*}\alpha_n=\frac {n-\lambda}{\sqrt{\lambda}}.\end{equation*}
Note that $\alpha_n\to -\re(z+w)$ as $n\to\infty$.

\begin{lem} \label{abel} We have the convergence
$$B_n\to e^{\,-b^2/2}\erker(z,w)
\quad\text{as}\quad n\to\infty,$$
where $b=b(z,w)=\im(z+\bar{w})$ and $\erker$ is the free boundary kernel \eqref{ffuncc}. Moreover,
$$A_n=e^{\,ib\sqrt{n}}e^{\,b^2/2}G(z,w)(1+o(1))$$
where $G$ is the Ginibre kernel and $o(1)\to 0$ uniformly on compact sets as $n\to\infty$.
\end{lem}

\begin{proof}
By a straightforward calculation we have
\begin{equation}\label{facile}\frac {n\zeta\bar{\eta}}\lambda=1+\frac i {\sqrt{n}}\im(z+\bar{w})+\frac 1 n a+O\left(n^{-3/2}\right)\end{equation}
where
$$a=a(z,w)=z\bar{w}-(z+\bar{w})\re(z+w)-\frac 1 2 \left(|\,z\,|^{\,2}+|\,w\,|^{\,2}\right)+\left[\re(z+w)\right]^{\,2}.$$
Inserting these expressions into $B_n$ (see \eqref{faan}) using the fact that the $Y_n$ are
asymptotically normal, we now approximate as follows (the symbol "$\sim$'' stands for asymptotic equality
as $n\to\infty$)
\begin{align*}B_n&\sim \frac 1 {\sqrt{2\pi}}\int_{-\infty}^{\,\alpha_n}\lpar 1+\frac i {\sqrt{n}}\im(z+\bar{w})+O(1/n)\rpar^{\sqrt{n}\cdot t}\, e^{\,-\,t^{\,2}/2}\, dt\\
&\sim \frac 1 {\sqrt{2\pi}}\int_{-\infty}^{\,-\re(z+\bar{w})}e^{\,i\im(z+\bar{w})t}e^{\,-\,t^{\,2}/2}\, dt
= e^{\,-\,b^{\,2}/2}\cdot \erker(z+\bar{w}).
\end{align*}

We now turn to the factor $A_n$ in \eqref{faan}. To deal with it, we denote $c=\re(z+w)$.
By \eqref{facile}, we then have
\begin{align*}A_n=\lpar1+\frac {ib}{\sqrt{n}}+\frac a n+O\left(n^{-3/2}\right)\rpar^{n+c\sqrt{n}+O(1)}
=e^{\,ib\sqrt{n}}\,e^{\,b^{\,2}/2}\,e^{\,ibc+a}\,e^{\,o(1)}.\end{align*}
Noting that
$$ibc+a=z\bar{w}-\frac 1 2\left(|\,z\,|^{\,2}+|\,w\,|^{\,2}\right),$$
we finish the proof of the lemma.
\end{proof}

By the lemma and relation \eqref{faan}, we have
$$K_n(z,w)=e^{\,i\sqrt{n}\im(z-w)}K(z,w)(1+o(1))$$
where $K$ is the free boundary kernel defined in \eqref{E1.12}. Since the factor
$c_n(z,w)=e^{i\sqrt{n}\im(z-w)}$ is a cocycle, this factor can be dropped when computing intensity
functions $R_{n,k}(z)=\det(K_n(z_i,z_j))$. This proves the desired convergence of intensity functions, at the same time establishing existence and uniqueness of the field $BG$. The proof of Theorem \ref{ginibref}
is complete. $\qed$

\subsection{Hard edge Ginibre ensemble} \label{pome2}
Let $Q^S(z)=\babs{\,z\,}^{\,2}$ when $|\,z\,|\le 1$
and $Q^S=+\infty$ otherwise. Let $\{\zeta_i\}_1^n$ denote a random configuration
from the corresponding ensemble. Rescaling about $p=1$ via $z_j=\sqrt{n}\,(\zeta_j-1)$ we
obtain a process $\config_n$. Let $H$ be the hard edge plasma function \eqref{hfuncc}.

\begin{thm} \label{hardg} There exists a unique point field $BG_h$ with correlation kernel $$K(z,w)=G(z,w)H(z+\bar{w})
\,(\1_\L\otimes\1_\L)(z,w).$$ The processes $\config_n$ converge to $BG_h$ in the sense that
all intensity functions converge locally boundedly almost everywhere, and locally uniformly
in $\L^2$, to the intensity functions of $BG_h$.
\end{thm}

Note that $K(z,z)<2$ for all $z$.
We prove in the appendix that the convergence of intensity functions in the theorem
implies the existence and uniqueness of a field $BG_h$ with correlation kernel $K$.
It thus suffices to prove convergence.

By  \eqref{E1.8} and a calculation, a correlation kernel for the hard edge Ginibre process is given by
$$\mathbf{K}_n(\zeta,\eta)=n\sum_{j=0}^{n-1}
\frac {(n\zeta\bar{\eta})^j} {\gamma\left(j+1,n\right)}e^{\,-nQ^S\left(\zeta\right)/2-nQ^S\left(\eta\right)/2},$$
where
$\gamma\left(j+1,n\right)=\int_0^n s^je^{-s}\,d s$
is the lower incomplete Gamma function.
Now rescale by
$$\zeta=1+z/\sqrt{n}\quad ,\quad \eta=1+w/\sqrt{n},\qquad z,w\in\L$$
and write
$$K_n(z,w)=\frac 1 n \mathbf{K}_n(\zeta,\eta)=\sum_{j=0}^{n-1}\lpar \frac {n\zeta\bar{\eta}} \lambda\rpar^j\frac {\lambda^j}{\gamma\left(j+1,n\right)}e^{-\lambda},$$
where $\lambda=n(|\,\zeta\,|^{\,2}+|\,\eta\,|^{\,2})/2$ is as in \eqref{lambdan}.
We shall use a rough estimate for $\gamma(j+1,n)$. Observe that, by a well-known fact, we have
$$\gamma\left(j+1,n\right)=j!\lpar 1-e^{-n}\sum_{k=0}^j\frac {n^k}
{k!}\rpar=j!\Prob\left(U_n>j\right)$$
where $U_n\sim \text{Po}(n)$. By normal approximation of the Poisson distribution
$$\Prob\left(U_n>j\right)= \fii\left(\xi_{j,n}\right)(1+o(1)),$$
where $\fii(\xi)=\frac 1 {\sqrt{2\pi}}\int_\xi^\infty e^{-t^2/2}\, dt$, $\xi_{j,n}=(j-n)/\sqrt{n}$, and
and $o(1)\to 0$ as $n\to\infty$ uniformly in $j$. We have shown that
\begin{equation*}K_n(z,w)=(1+o(1))\sum_{j=0}^{n-1}\lpar \frac {n\zeta\bar{\eta}} \lambda\rpar^j\frac {\lambda^j}{j!}e^{-\lambda}\frac 1 {\fii(\xi_{j,n})}.\end{equation*}
Finally, if $X_n\sim \text{Po}(\lambda)$, we can write the last sum in the form
\begin{equation*}\sum_{j=0}^{n-1}
\lpar\frac {n\zeta\bar{\eta}}\lambda\rpar^j \frac 1{\fii(\xi_{j,n})} \Prob(X_n=j)
=\Expe\left[\lpar\frac {n\zeta\bar{\eta}}\lambda\rpar^{X_n}\cdot
\frac 1 {\fii\lpar \frac {X_n-n}{\sqrt{n}}\rpar}\cdot \1_{\left\{X_n<n\right\}}\right].\end{equation*}
Defining $Y_n$ by $X_n=\lambda+\sqrt{\lambda}Y_n$, we now get a relation of the form
\begin{align*}K_n(z,w)&=(1+o(1))\left(\frac {n\zeta\bar{\eta}}\lambda\right)^\lambda
\cdot \Expe\left[\left(\frac {n\zeta\bar{\eta}}\lambda\right)^{\sqrt{\lambda} Y_n}
\frac 1 {\fii\left(\sqrt{\lambda/n}Y_n\,-\alpha_n\right)}
\cdot \1_{\left\{Y_n<\alpha_n\right\}}\right]\\
&=(1+o(1))A_n\cdot \tilde{B}_n,\qquad \alpha_n=(n-\lambda)/\sqrt{\lambda}.\end{align*}
Using the asymptotic identities $\alpha_n\sim -\re(z+w)$ and
$\lambda/n\sim 1$, we approximate the factor $\tilde{B}_n$
as follows, using the central limit theorem,
\begin{align*}\tilde{B}_n&\sim\frac 1 {\sqrt{2\pi}}\int_{-\infty}^{\,-\re(z+w)}
e^{\,i\im(z+\bar{w})t}\frac 1 {\fii\lpar\sqrt{\lambda/n}\,t-\alpha_n\rpar}
e^{\,-\,t^{\,2}/2}\, dt\\
&\sim e^{\,-\,b^{\,2}/2}\frac 1 {\sqrt{2\pi}}\int_{-\infty}^{\,-\,\re(z+w)}
\frac 1 {\fii\left(\,\re(z+\bar{w})+t\,\right)}e^{\,-\,\left(t-i\im(z+\bar{w})\right)^{\,2}/2}\, dt\\
&=e^{\,-\,b^{\,2}/2}\frac 1 {\sqrt{2\pi}}\int_{-\infty}^{\,0}\frac 1 {\fii(u)}e^{\,-\,\left(u-(z+\bar{w})\right)^{\,2}/2}\, du=
e^{\,-\,b^{\,2}/2}H(z+\bar{w}).
\end{align*}
Using the asymptotics for $A_n$ in Lemma \ref{abel} we now conclude that
$$K_n(z,w)=e^{\,i\sqrt{n}\im(z-w)}G(z,w)\hfun(z+\bar{w})(1+o(1)),\quad z,w\in\L,$$
where $o(1)\to 0$ uniformly on compacts. The first factor is a cocycle. We
conclude that if $K(z,w)=G(z,w)\hfun(z+\bar{w})\1_\L(z)\1_\L(w)$ is the hard edge kernel, then
$K_n\to K$ almost everywhere with locally bounded convergence, finishing the proof of the theorem.
q.e.d.

\section{Analyticity and compactness} \label{parti}
In this section, we prove Theorem \ref{TT1}. We thus fix a moving point $p=(p_n)\in S$ and a sequence of angles $\theta_n$ and rescale about $p$ according to
$$z_n=e^{-i\theta_n}\sqrt{n\Lap Q(p_n)}(\zeta-p_n).$$
To simplify notation we shall in the following write $p$ for $p_n$ and $\theta$ for $\theta_n$. The careful reader will note that all our estimates are independent of
$n$.

\subsection{General notation}
For a measurable function $\phi:\C\to \R\cup\{+\infty\}$ we define $L^2_\phi$ to be the space of functions normed by
$$\left\|\,f\,\right\|_\phi^{\,2}=\int_\C\left|\,f\,\right|^{\,2}e^{\,-\,\phi}\,dA.$$
When $\phi=0$ we just write $L^2$ for $L^2_\phi$ and denote the norm by $\|\cdot\|$.

We denote by $C$ (large) and $c$ (small) various positive unspecified constants (independent of $n$) whose exact value
can change meaning from time to time.

\subsection{Potentials and reproducing kernels} \label{rkta}
Fix a neighbourhood $\nbh$ of $\drop$ and
a number $\delta_0>0$ such that $Q$ is real-analytic and strictly subharmonic in the
$2\delta_0$-neighbourhood of $\nbh$.


Let $\calP_n$ be the space of analytic polynomials of degree at most $n-1$; we equip $\calP_n$ with the norm of $L^2_{nQ}$.
The corresponding space $\Pol_n$ of weighted polynomials is defined to consist of all functions of the form $f=qe^{\,-\,nQ/2}$ where
$q\in\calP_n$; we regard $\Pol_n$ as a subspace of $L^2$ and denote the corresponding orthogonal projections by
$$\pi_n:L^2_{nQ}\to \calP_n\qquad \text{and}\qquad \Pi_n:L^2\to\Pol_n.$$
We write $\bfk_n$ and $\bfK_n$ for the reproducing kernels of $\calP_n$ and $\Pol_n$ respectively.
Then
$$\bfK_n(\zeta,\eta)=\bfk_n(\zeta,\eta)e^{\,-\,nQ(\zeta)/2}e^{\,-\,nQ(\eta)/2}.$$
It is easy to see that the assignment
$$U_n:L^2_{nQ}\to L^2\qquad :\qquad f\mapsto fe^{\,-\,nQ/2}$$
is unitary, maps $\calP_n$ onto $\Pol_n$, and satisfies $U_n\pi_n=\Pi_nU_n$.

\subsection{Analytic continuation and bulk approximations} \label{bula}
Let $A(\zeta,\eta)$ be a Hermitian-analytic function defined in a neighbourhood in $\C^2$ of the "diagonal'' $\diag:=\left\{\,(\zeta,\zeta)\,;\,\zeta\in \nbh\,\right\}$, such that
$$A(\zeta,\zeta)=Q(\zeta),\quad \zeta\in \nbh.$$
We can choose $\delta_0>0$ small enough that
$A(\zeta,\eta)$ is defined and Hermitian-analytic in the set of points $(\zeta,\eta)\in\C^2$ whose distance
to $\diag$ is $<2\delta_0$. Call this set $\Lambda$.


We now define "bulk approximations'' $\bfk_n^\#$ and $\bfK_n^\#$, defined in the domain of $A(\zeta,\eta)$ via
\begin{equation}\label{buckla}\bfk_n^\#(\zeta,\eta)=n\left(\d_1\bar{\d}_2A\right)(\zeta,\eta)\cdot e^{\,nA(\zeta,\eta)}\qquad \text{and}\qquad
\bfK_n^\#(\zeta,\eta)=\bfk_n^\#(\zeta,\eta)e^{\,-\,nQ(\zeta)/2}e^{\,-\,nQ(\eta)/2}.\end{equation}

\subsection{Elementary estimates for the one-point function} \label{behm}
 We write $\bfR_n(\zeta)=\bfK_n(\zeta,\zeta)$ for the one-point function. By a basic fact for reproducing kernels we have the identity
\begin{equation}\label{genrkth}\bfR_n(\zeta)=\sup\left\{\,\left|\,f(\zeta)\,\right|^{\,2}\,;\,f\in\Pol_n\,,\,
\left\|\,f\,\right\|\le1\,\right\}.
\end{equation}
We shall also use the following simple pointwise-$L^2$ estimate.

\begin{lem}\label{pl2} Suppose that $u$ is analytic in the disc $D:=D\left(p;c/\sqrt{n\Lap Q(p)}\right)$,
 where $Q$ is $C^2$-smooth at $p$. Let $f=ue^{-\,nQ/2}$. Then there is a number $C$ depending
 only on $c$ and $\Lap Q(p)$ such that
$$\babs{\,f(p)\,}^{\,2}\le Cn\int_{D}\babs{\,f\,}^{\,2}\, dA.$$
\end{lem}

\begin{proof} Suppose \textit{w.l.o.g.} that $\Lap Q(p)=1$ and pick a number $a>1$.
Consider the function $F_n(z)=f\left(p+z/\sqrt{n}\right)\cdot e^{a\babs{\,z\,}^{\,2}/2}$. We then have
$\Lap \log \babs{\, F(z)\,}^{\,2}\ge -\Lap Q\left(p+z/\sqrt{n}\right)+a\ge 0$ for
$\babs{\,z\,}\le c$ if $n$ is large enough. Then $F_n$ is subharmonic in $D$, which implies the desired estimate.
\end{proof}

Let $\gamma$ be the minimum of $Q+2U^\sigma$, where $U^\sigma(z)=\int_\C\log\frac 1 {\babs{\,z-w\,}}\, d\sigma(w)$
is the logarithmic potential of the equilibrium measure $\sigma=\Lap Q\cdot\1_S\, dA$.
By the \textit{obstacle function} corresponding to $Q$, we mean the
subharmonic function
$$\eqpot\left(\zeta\right)=-2U^\sigma\left(\zeta\right)+\gamma.$$
It is known (see \cite{ST}) that $\eqpot=Q$ on $S$ while $\eqpot$ is harmonic on $S^c$ and is of logarithmic increase
$$\eqpot\left(\zeta\right)= \log\babs{\,\zeta\,}^{\,2}+\Ordo(1),\qquad (\zeta\to\infty).$$
Furthermore $\eqpot$ has a Lipschitz continuous gradient on $\C$. We remind of the following basic result; the "maximum principle of weighted potential theory''.

\begin{lem} \label{maxlem} If $f\in\calW_n$ and $\babs{\,f\,}\le 1$ on $S$, then
$\babs{\, f\,}\le e^{\,-\,n\left(Q-\eqpot\right)/2}$ on $\C$.
\end{lem}

\begin{proof} If $f=p\cdot e^{-nQ/2}$, then $\frac 1 n \log\babs{\, p\,}^{\,2}$ is a subharmonic minorant of $Q$ which grows no faster than $\log\babs{\,\zeta\,}^{\,2}+\const$
as $\zeta\to \infty$. It is well-known that $\eqpot(\zeta)$ is the supremum of $f(\zeta)$ where $f$ ranges over the functions having these properties (see e.g. \cite{ST}).
\end{proof}

Combining the preceding lemmas
with the identity \eqref{genrkth}
gives the following bound for the $1$-point function.

\begin{lem}\label{snor}
There is a constant $C$ independent of $n$ and $\zeta$ such that
$$\bfR_n(\zeta)\le Cne^{\,-\,n\left(Q-\eqpot\right)(\zeta)},\qquad \zeta\in\C.$$
\end{lem}

\subsection{Rescaled kernels} Let $p\in\C$ and fix a real parameter $\theta$.
Let $\{\zeta_j\}_1^n$ denote the point process corresponding to $Q$. Recall that rescaled point process at $p$ in the direction $e^{\,i\theta}$ is the process
$\config_n=\{z_j\}_1^n$ where $z_j=e^{\,-\,i\theta}\sqrt{n\Delta Q(p)}\,(\zeta_j-p)$. A kernel $K_n$ for the rescaled process
is given by
$$K_n(z,w)=\frac 1 {n\Lap Q(p)}\, \bfK_n(\zeta,\eta)$$
where
\begin{equation}\label{reo}z=e^{\,-\,i\theta}\sqrt{n\Delta Q(p)}\,(\zeta-p)\quad ,\quad w =e^{\,-\,i\theta}\sqrt{n\Delta Q(p)}\,(\eta-p).\end{equation}
We define the rescaled bulk approximation $K_n^\#$ by
\begin{equation*}\label{bulapp}K_n^\#(z,w)=\frac 1 {n\Lap Q(p)} \bfK_n^\#(\zeta,\eta).\end{equation*}
Here $\bfK_n^\#$ is the bulk approximation to $\bfK_n$ defined in \eqref{buckla}.

\subsection{Convergence of the approximate kernels} \label{cak} The kernel $G(z,w)=e^{\,z\bar{w}\,-\,|\,z\,|^{\,2}/2\,-\,|\,w\,|^{\,2}/2}$ satisfies
$$\left|\,G(z,w)\,\right|^{\,2}=e^{\,-\,|\,z-w\,|^{\,2}}.$$
Let $V_n$ denote the set of points $(z,w)$ such that $(\zeta,\eta)\in \Lambda$ and \eqref{reo} holds.
Here $\Lambda$ is the $2\delta_0$-neighbourhood of the diagonal $X$, see Section \ref{bula}.

It is clear from \eqref{reo} that the sets $V_n$ eventually contains each compact subset of $\C^2$. Indeed, there is a constant $\rho>0$ depending only on $\Lap Q(p)$ such that
\begin{equation}\label{obsvn}\left\{\,(z,w)\,;\, |\,z\,|\le \rho\sqrt{n}\,,\, |\,w\,|\le \rho\sqrt{n}\,\right\}\subset V_n.\end{equation}
We have the following lemma.

\begin{lem} \label{cocy} We have $K_n^\#(z,w)=c_n(z,w)G(z,w)(1+o(1))$ as $n\to\infty$ where
$c_n$ are cocycles on $V_n$ and $o(1)\to 0$ as $n\to\infty$, uniformly on compact subsets of $\C^2$.
\end{lem}

\begin{proof} We can assume that $p=0$ and $\theta=0$. Put $\Lap Q(0)=\delta$.
Recall that
$$K_n^\#(z,w)=\frac 1 \delta \left(\d_1\bar{\d}_2 A\right)(\zeta,\eta) e^{\,n\,\left[A(\zeta,\eta)-A(\zeta,\zeta)/2-A(\eta,\eta)/2\right]}.$$
As $n\to\infty$, our rescaling means that $\d_1\bar{\d}_2 A(\zeta,\eta)\to \delta$.
Moreover, by Taylor's formula, the expression in the exponent is
\begin{align*}
n&\left[\d_1 A(0,0)\cdot \zeta+\dbar_2 A(0,0)\cdot\bar{\eta}+\frac 1 2\left(\d_1^{\,2} A(0,0)\cdot \zeta^{\,2}+
2\d_1\dbar_2 A(0,0)\cdot \zeta\bar{\eta}\,\,\,\,+\dbar_2^{\,2}A(0,0)\cdot\bar{\eta}^{\,2}\right)\right]\\
-\frac n 2 &\left[\d_1 A(0,0)\cdot \zeta+\dbar_2 A(0,0)\cdot \bar{\zeta}+\frac 1 2
\left(\d_1^{\,2} A(0,0)\cdot \zeta^{\,2}+2\d_1\dbar_2 A(0,0)\cdot |\,\zeta\,|^{\,2}+\dbar_2^{\,2} A(0,0)\cdot
\bar{\zeta}^{\,2}\right)\right]\\
-\frac n 2 &\left[\d_1 A(0,0)\cdot \eta+\dbar_2 A(0,0)\cdot \bar{\eta}+
\frac 1 2
\left(\d_1^{\,2} A(0,0)\cdot \eta^2+2\d_1\dbar_2 A(0,0)\cdot |\,\eta\,|^{\,2}+\dbar_2^{\,2} A(0,0)\cdot
\bar{\eta}^{\,2}\right)\right]\\
&+n\cdot \Ordo\left(\left\|(\zeta,\eta)\right\|^{\,3}\right),
\end{align*}
and this equals
\begin{align*}i\sqrt{\frac n \delta}&\im\left\{\d_1 A(0,0)(z-w)\right\}+\frac i {2\delta}\im\left\{\d_1^{\,2} A(0,0)\left(z^{\,2}-w^{\,2}\right)\right\}\\
&+z \bar{w}-|\,z\,|^{\,2}/2-|\,w\,|^{\,2}/2+
\frac 1 {\sqrt{n}}\Ordo\left(\left\|(z,w)\right\|^{\,3}\right).
\end{align*}
The first two terms correspond to cocycles.
\end{proof}

\begin{rem}
The proof of the lemma shows that if $|\,\zeta\,|<\delta_n$ where $n\delta_n^{\,3}\to 0$ as $n\to\infty$, then
$$n\re A(\zeta,0)-nQ(\zeta)/2-nQ(0)/2=-n\Lap Q(0)\left|\,\zeta\,\right|^{\,2}/2+o(1).$$
Hence (replacing "$0$'' by "$p$'') we see that there is a number $C$ such that
\begin{equation*}\label{rtayl}e^{\,n\re A(\zeta,p)-nQ(\zeta)/2}\le Ce^{\,nQ(p)/2-n\Lap Q(p) \,|\,\zeta-p\,|^{\,2}/2}.
\end{equation*}
\end{rem}

\subsection{Compactness} \label{freecomp} Recall that $\nbh$ denotes some sufficiently small neighbourhood of $S$.
Fix a point $p$ in $\nbh$ and rescale
about $p$ as in \eqref{E1.3}.

\begin{thm} \label{cpthm} Let $p$ be an arbitrary point of $\nbh$.
There is a sequence of cocycles $c_n$ such that every subsequence of $c_nK_n$ has a subsequence converging
uniformly on compact subsets of $\C^2$. Furthermore, every limit point has the form
$K=G\Psi$ where $\Psi(z,w)$ is an Hermitian entire function.
\end{thm}

The theorem implies Theorem \ref{TT1}; this is but the special case when $p$ is a boundary point of $\drop$.

In the proof, we will use the
functions $\Psi_n$ defined on $V_n$ by the equation
$$K_n=\Psi_n K_n^\#.$$
We will need two lemmas. Recall the definition of the set $V_n$ from the beginning of Section \ref{cak}.

\begin{lem} \label{AL1} The function $\Psi_n$ is Hermitian-analytic
in the set $V_n$.
\end{lem}

\begin{proof} For $(z,w)\in V_n$ we have $(\zeta,\eta)\in \Lambda$ and
$$\Psi_n(z,w)=\frac {K_n(z,w)}{K_n^\#(z,w)}=\frac{\bfK_n(\zeta,\eta)}
{\bfK_n^\#(\zeta,\eta)},$$
whence
$$\Psi_n(z,w)=\frac {\bfk_n(\zeta,\eta)e^{\,-nQ(\zeta)/2}e^{\,-nQ(\eta)/2}}
{\left(\d_1\dbar_2 A\right)(\zeta,\eta)\cdot ne^{\,n\left(A(\zeta,\eta)-Q(\zeta)/2-Q(\eta)/2\right)}}
=\frac {\bfk_n(\zeta,\eta)}
{\left(\d_1\dbar_2 A\right)(\zeta,\eta)\cdot ne^{\,nA(\zeta,\eta)}}.$$
The statement follows since $\zeta$ and $\eta$ depend analytically on $z$ and $w$.
\end{proof}

\begin{lem} \label{AL2} For each compact set $K\subset\C^2$ there is a constant $C=C_K$ such that
$\left|\,\Psi_n(z,w)\,\right|^{\,2}\le Ce^{\,|z-w|^2}$ when $(z,w)\in K$ and $n$ is large enough.
\end{lem}

\begin{proof} Choose $n_0$ large enough that $K\subset V_{n_0}$.
Since $K_n$ is a positive kernel,
and since (by Lemma \ref{cocy}) $\left|\,K_n^\#(z,w)\,\right|\to \left|\,G(z,w)\,\right|$ uniformly on compact subsets, we have uniformly on $K$,
$$\left|\,\Psi_n(z,w)\,\right|^{\,2}=\babs{\,\frac {K_n(z,w)}{K_n^\#(z,w)}\,}^{\,2}\le C \frac {R_{n}(z)R_{n}(w)}
{\left|\,G(z,w)\,\right|^{\,2}},\quad (R_n(z)=K_n(z,z)),$$
for all $n\ge n_0$.
By Lemma \ref{snor}, we have a uniform bound $R_{n}\le C$, which finishes the proof
of the lemma.
\end{proof}

\begin{proof}[Proof of Theorem \ref{cpthm}] Lemma \ref{AL2} shows that the family $\{\Psi_n\}$ is
locally bounded on $\C^2$, viz. is a normal family. Pick a locally uniformly convergent subsequence
$\{\Psi_{n_k}\}$
converging to a limit $\Psi$.
Also fix $z$ and recall that $\int\babs{\,K_{n_k}(z,w)\,}^{\,2}\, dA(w)=K_{n_k}(z,z)$.
In terms of the functions $\Psi_n$,
$$\int_{D\left(0;\rho\sqrt{n_k}\right)} \babs{\,K_{n_k}^\#(z,w)\cdot\Psi_{n_k}(z,w)\,}^{\,2}\, dA(w)=\Psi_{n_k}(z,z)(1+o(1)),\quad (k\to\infty),$$
where $\rho>0$ is the constant in \eqref{obsvn}.
Letting $k\to\infty$ we get, by Fatou's lemma, that the mass-one inequality
\eqref{mocio} holds.

Finally we use Lemma \ref{cocy} to select cocycles $c_n$ such that $c_nK_n^\#\to G$ uniformly on compact subsets of $\C^2$. Then $c_{n_k}K_{n_k}=\Psi_{n_k}\cdot c_{n_k}K_{n_k}^\#\to \Psi G$, finishing the proof of the theorem.
\end{proof}

Let $K=\Psi G$ be a limiting kernel and write $R(z)=K(z,z)=\Psi(z,z)$; we have shown above that the mass-one inequality holds, i.e.
$\int e^{-\babs{z-w}^2}\babs{\Psi(z,w)}^2\, dA(w)\le R(z)$. It will be useful to reformulate this inequality.

\begin{lem} \label{moi} The mass-one inequality holds if and only if
\begin{equation}\label{ipde}\sum_{n=0}^\infty\frac {\babs{\,\d^{\,n} R\,}^{\,2}} {n!}\le R.\end{equation}
\end{lem}

\begin{proof} Since $R(z)=\Psi(z,z)$
we have
$$\Psi(z+w,z)=\sum_{n=0}^\infty \frac {\d_1^{\,n}\Psi(z,z)}{n!}w^{\,n}=\sum_{n=0}^\infty
\frac {\d^{\,n} R(z)}{n!} w^{\,n}.$$
Likewise,
$$\Psi(z,z+w)=\sum_{n=0}^\infty \frac {\dbar^{\,n} R(z)}{n!}\bar{w}^{\,n}.$$
It follows that
\begin{align*}
\int  e^{\,-\,|\,w\,|^{\,2}}&|\,\Psi(z,z+w)\,|^{\,2}\, dA(w)= \lim_{M\to\infty}\int_{\,|\,w\,|\,<\,M} e^{\,-\,|\,w\,|^{\,2}}|\,\Psi(z,z+w)\,|^{\,2}\, dA(w)\\
&=\lim_{M\to\infty}\sum_{n=0}^\infty \frac {\babs{\,\d^{\,n} R(z)\,}^{\,2}} {(n!)^2}
\int_{\,|\,w\,|\,<\,M}|\,w\,|^{\,2n}e^{\,-\,|\,w\,|^{\,2}}\, dA(w)
=\sum_{n=0}^\infty \frac {\babs{\,\d^{\,n} R(z)\,}^{\,2}} {n!}.
\end{align*}
The proof of the lemma is finished.
\end{proof}

\begin{rem} The proof of Lemma \ref{moi} shows that the
mass-one equation for a kernel $\Psi(z,w)$ is equivalent to that the function $R(z)=\Psi(z,z)$ satisfy
\begin{equation}\label{MOEQ}R=\sum_{n=0}^\infty\frac {\babs{\,\d^{\,n} R\,}^{\,2}} {n!}.\end{equation}
One can
regard this as a differential equation of infinite order.
\end{rem}

\section{Ward's equation and the mass-one inequality 
} \label{WEFA}
In this section, we prove Ward's equation and the triviality alternative, i.e., part \ref{TT1.5_2} of Theorem \ref{TT1.5}. In the next section, we start by deriving
a slightly modified (or "localized'') form of the Ward identity used in \cite{AHM3}.
This modification is necessary when dealing with hard edge processes, and is in general
quite convenient.

\subsection{Ward's identity} \label{theidentity}
To set things up, fix a test-function $\psi\in C^\infty_0(\C)$. Define a function $W_n^+[\psi]$ of
$n$ variables by
\begin{equation*}W_n^+[\psi]=I_n[\psi]-II_n[\psi]+III_n[\psi]\end{equation*}
where
\begin{equation*}I_n[\psi]\lpar \zeta\rpar=\frac 1 2\sum_{j\ne k}^n\frac
{\psi\lpar \zeta_j\rpar-\psi\lpar \zeta_k\rpar} {\zeta_j-\zeta_k}
\quad ;\quad II_n[\psi]\lpar \zeta\rpar=n\sum_{j=1}^n \left[\d Q\cdot\psi\right] \lpar\zeta_j\rpar\quad
;\quad III_n[\psi](\zeta)=\sum_{j=1}^n \d\psi\lpar \zeta_j\rpar.\end{equation*}
Here we think of $\lpar \zeta_j\rpar_1^n$ as being randomized with respect to the Boltzmann-Gibbs law $\Prob_n$,
see \eqref{E1.1}.

We assume only that $Q$ be smooth in a neighbourhood of the support of $\psi$. We can then make sense of $W_n^+[\psi]$ even though $\d Q$ may be undefined in portions of the plane.
Indeed, we \textit{define}
$$\left[\d Q\cdot \psi\right](\zeta)=\begin{cases} \d Q(\zeta)\cdot \psi(\zeta)& \text{if}\quad \zeta\in\supp\psi,\cr
\quad 0& \qquad \text{otherwise.}\cr
\end{cases}$$
We then have the following form of \textit{Ward's identity}.

\begin{thm} \label{war1}  Suppose that $Q$ is $C^2$-smooth in a neighbourhood of
$\supp \psi$. Then
$\Expe_n\, W_n^+[\psi]=0.$
\end{thm}

\begin{proof} We modify the argument in \cite{AHM3}.
Given $\zeta\in \C$ and $\eps>0$ we let $D_{\eps\psi}(\zeta)$ be the
closed disc centered at $\zeta$ of radius $\eps\babs{\psi(\zeta)}$. Choosing $\eps=\eps(\psi)>0$ sufficiently small, there are two alternatives for each point $\zeta\in\C$, (i) $Q$ is $C^2$-smooth in a neighbourhood
of $D_{\eps\psi}(\zeta)$, or (ii) $\psi(\zeta)=0$.

Now fix an arbitrary sequence $\zeta=(\zeta_j)_1^n$ and
put $\eta_j=\phi\lpar\zeta_j\rpar=\zeta_j+\eps\psi\lpar \zeta_j\rpar/2$, $1\le j\le n$. The Jacobian for $\phi$ is
$$J(\zeta)=\babs{\,\d\phi\lpar \zeta\rpar\,}^{\,2}-\babs{\,\dbar\phi\lpar\zeta\rpar\,}^{\,2}=
1+\eps\re\d\psi\lpar\zeta\rpar+\Ordo\lpar \eps^2\rpar,\quad (\eps\to 0),$$
whence
(with $III_n=III_n[\psi]$)
\begin{equation*}d V_n(\eta)=\prod_{j=1}^n \babs{J(\zeta_j)}dA(\zeta_j)=\left[\,1+\eps\,\re \,III_n\lpar \zeta\rpar+
\Ordo\lpar\eps^2\rpar\,\right]\,d V_n(\zeta).
\end{equation*}
Moreover,
\begin{equation*}\begin{split}\log\babs{\,\eta_i-\eta_j\,}^{\,2}&=\log\babs{\,\zeta_i-\zeta_j\,}^{\,2}+
\log\babs{\,1+\frac \eps 2 \frac {\psi(\zeta_i)-\psi(\zeta_j)}
{\zeta_i-\zeta_j}\,}^{\,2}\\
&=\log\babs{\,\zeta_i-\zeta_j\,}^{\,2}+
\eps\re\frac {\psi(\zeta_i)-\psi(\zeta_j)}
{\zeta_i-\zeta_j}+\Ordo\lpar\eps^2\rpar,\\
\end{split}
\end{equation*}
so that
\begin{equation}\label{fs1}\sum_{j\ne k}^n\log\babs{\,\eta_j-\eta_k\,}^{\,-\,1}=
\sum_{j\ne k}^n\log\babs{\,\zeta_j-\zeta_k\,}^{\,-\,1}-
\eps  \re\, I_n(\zeta)+\Ordo\lpar\eps^2\rpar,\quad (\eps\to 0).
\end{equation}
If $D_{\eps\psi}(\zeta_j)$ is contained in a domain where $Q$ is $C^2$-smooth, then, by Taylor's formula,
\begin{equation}\label{fisnik}Q\lpar \eta_j\rpar=Q\lpar \zeta_j+\frac \eps 2 \psi\lpar\zeta_j\rpar\rpar
=Q\lpar\zeta_j\rpar+\eps\, \re\,\left[ \d Q\cdot\psi\right] (\zeta_j)+\Ordo\left(\eps^2\right).
\end{equation}
For other $j$'s we have $\psi(\zeta_j)=0$ and $\eta_j=\zeta_j$, whence \eqref{fisnik} holds, since
$\left[\d Q\cdot\psi\right](\zeta_j)=0$ by definition. Hence \eqref{fisnik} holds in all cases, so
\begin{equation}\label{fs2}n\sum_{j=1}^n Q\lpar \eta_j\rpar=
n\sum_{j=1}^n Q\lpar\zeta_j\rpar+\eps\, \re \,II_n(\zeta)+\Ordo\lpar\eps^2\rpar.
\end{equation}
Now \eqref{fs1} and \eqref{fs2} imply that the Hamiltonian
$\Ham_n\left(\eta\right)=\sum_{j\ne k}\log\babs{\,\eta_j-\eta_k\,}^{\,-\,1}+n\sum_{j=1}^n Q\left(\eta_j\right)$ satisfies
\begin{equation*}\label{fs3}\Ham_n \left(\eta\right)=\Ham_n\left(\zeta\right)
+\eps\cdot  \re\, \lpar- I_n(\zeta)+II_n(\zeta)\rpar+
\Ordo\lpar\eps^2\rpar.\end{equation*}
It follows that the partition function $Z_n:=\int_{\C^n}e^{- \Ham_n\left(\eta\right)}\,
d V_n\left(\eta\right)$ satisfies
\begin{equation*}\begin{split}Z_n&=\int_{\C^n}e^{-\Ham_n(\zeta)-\eps\re\,\lpar- I_n(\zeta)+II_n(\zeta)\rpar+
\Ordo\lpar\eps^2\rpar}\,\left[1+\eps \re III_n(\zeta)+\Ordo\lpar\eps^2\rpar\right]\,
d V_n(\zeta).\\
\end{split}\end{equation*}
Since the integral is independent of $\eps$, the coefficient of $\eps$
in the right hand side must vanish, which means that
\begin{equation*}\re \int_{\C^n}
\lpar III_n(\zeta)+ I_n(\zeta)-II_n(\zeta)\rpar
\,e^{- \Ham_n(\zeta)}\,d V_n(\zeta)=0,\end{equation*}
or $\re \Expe_n \,W_n^+[\psi]=0$. Replacing $\psi$ by
$i\psi$ in the preceding argument gives $\im \Expe_n
\,W_n^+[\psi]=0$ and the theorem follows.
\end{proof}

\subsection{Rescaled version} \label{funo} We now fix a point $p$ in a small neighbourhood
$\nbh$ of $S$, where $Q$ is $C^2$-smooth.
We rescale the system $\left\{\zeta_j\right\}_1^n$
about $p$ in the usual way, obtaining the rescaled system $\config_n=\{z_j\}_1^n$ where
$$z_j=e^{\,-i\theta}\sqrt{n\Lap Q(p)}\,(\zeta_j-p).$$
The value of $\theta$ is here irrelevant.
Recall that the rescaled intensity functions are defined via
$$R_{n,k}(z_1,\ldots,z_k)=\frac 1 {(n\Lap Q(p))^k}\bfR_{n,k}(\zeta_1,\ldots,\zeta_k),$$
while the
Berezin kernel rooted at $p$ is defined by
\begin{equation*}B_n(z,w)
=\frac {R_{n,1}(z)R_{n,1}(w)-
R_{n,2}(z,w)}
{R_{n,1}(z)}.\end{equation*}

The following result is a rescaled
form of Ward's identity.

\begin{thm} \label{THL} 
In the above circumstances we have the equation
\begin{equation*}\label{E1.5}\dbar C_n(z)
=R_n(z)-1-\Lap \log R_n(z)+o(1),\qquad z\in\C,\end{equation*}
where
$$C_n(z):=
\int_\C \frac { B_n(z,w)} {z-w}\, dA(w),\quad \text{and}\quad R_n(z):=R_{n,1}(z)=B_n(z,z),$$
and $o(1)\to 0$ uniformly on compact subsets of $\C$ as $n\to\infty$.
\end{thm}

\begin{proof}
Fix a point $p$ such that $Q$ is $C^2$-smooth and strictly subharmonic in a neighbourhood $U$ of $p$.
We can without loss of generality assume that $p=0$ and $\theta=0$.
Fix a test-function $\psi$ supported in the dilated set $\sqrt{n\delta}\cdot U$, where $\delta=\Lap Q(p)$.
Write
$$z=\sqrt{n\delta}\cdot\zeta\qquad ,\qquad w=\sqrt{n\delta}\cdot\eta,$$
and let $\psi_n\left(\zeta\right)=\psi\lpar z\rpar$. Thus $\supp\psi_n\subset U$.
The change of variables $(\zeta,\eta)\mapsto \lpar z,w\rpar$ gives
\begin{equation}\label{rep0}\begin{split}\Expe_n\, I_n[\psi_n]&=\frac 1 2\,
\iint_{\C^2}\frac {\psi\lpar \zeta\sqrt{n\delta}\rpar-\psi\lpar \eta\sqrt{n\delta}\rpar} {\zeta-\eta}\, \bfR_{n,2}(\zeta,\eta)\,d V_2(\zeta,\eta)\\
&=\sqrt{n\delta}\,\frac 1 2 \,\iint_{\C^2} \frac {\psi(z)-\psi(w)} {z-w}\,
R_{n,2}(z,w)\,d V_2(z,w)\\
&=\sqrt{n\delta}\,\int_{\sqrt{n\delta}\cdot U} \psi(z)\,dA(z)
\int_{\C} \frac {R_{n,2}(z,w)}{z-w}\,dA(w)
.\\
\end{split}
\end{equation}

Similarly, since $\supp\psi_n\subset U$,
\begin{equation}\label{rep1}\begin{split}\Expe_n \,II_n[\psi_n]&=n\,\int_U
\left[\d Q\cdot\psi_n\right]\lpar \zeta\rpar\,\bfR_{n,1}(\zeta)\,dA(\zeta)\\
&=n \int_{\sqrt{n\delta}\cdot U} \d Q\lpar \frac z {\sqrt{n\delta}}\rpar \cdot \psi(z)
\cdot R_{n,1}(z)\,dA(z).\\
\end{split}\end{equation}

Finally, in the sense of distributions,
\begin{equation}\label{rep2}\begin{split}\Expe_n \,III_n[\psi_n]&=\int_U \d\psi_n\cdot \bfR_{n,1}\,dA
=\sqrt{n\delta}\,\int_U \d\psi\left(\zeta\sqrt{n\delta}\right)\,\bfR_{n,1}\left(\zeta\right)\,dA\left(\zeta\right)\\
&=\sqrt{n\delta}\,\int_{\sqrt{n\delta}\cdot U} \d\psi\cdot R_{n,1}\,dA
=-\sqrt{n\delta}\,\int_{\sqrt{n\delta}\cdot U} \psi\cdot \d R_{n,1}\,dA.\\
\end{split}\end{equation}

After dividing by $\sqrt{n\delta}$ in Ward's identity (Theorem \ref{war1}), we deduce from eq.'s \eqref{rep0}--\eqref{rep2} that
\begin{align*}\int_{\sqrt{n\delta}\cdot U} \psi(z)\,
\left[\,\int_\C \frac {R_{n,2}(z,w)} {z-w}\,dA(w)\,\right]\,
dA(z)&=\int_{\sqrt{n\delta}\cdot U} \frac {\sqrt{n}} {\sqrt{\delta}}\,\d Q\lpar \frac z {\sqrt{n\delta}}\rpar\cdot
\psi(z)\cdot R_{n,1}(z)\,dA(z)\\
&+
\int_{\sqrt{n\delta}\cdot U} \psi\cdot \d R_{n,1}\,dA.\end{align*}
Since $\psi$ is arbitrary, we get the following identity, in the sense of
distributions,
\begin{equation}\label{18}\int_\C \frac {R_{n,2}(z,w)} {z-w}\,dA(w)=
\frac {\sqrt{n}} {\sqrt{\delta}}\,\d Q\lpar \frac z {\sqrt{n\delta}}\rpar
\cdot R_{n,1}(z)+\d R_{n,1}(z),\qquad z\in\sqrt{n\delta}\cdot U.
\end{equation}
But
\begin{equation*}\frac {R_{n,2}(z,w)} {R_{n,1}(z)}=
R_{n,1}(w)-B_n(z,w)\end{equation*}
so we can write \eqref{18} as
\begin{equation*}\int_\C \frac {R_{n,1}(w)} {z-w}dA(w)-C_n(z)=
\frac {\sqrt{n}} {\sqrt{\delta}}\,\d Q\lpar \frac z {\sqrt{n\delta}}\rpar+\d \log R_{n,1}(z).\end{equation*}
Here $C_n(z)=\int \frac {B_n(z,w)} {z-w}dA(w)$.
Taking a $\dbar$-derivative now gives
\begin{equation*}R_{n,1}(z)-\dbar\, C_n(z)=\frac 1 \delta \Delta Q\lpar \frac z {\sqrt{n\delta}}\rpar+\Delta\log R_{n,1}(z),\qquad z\in\sqrt{n\delta}\cdot U.\end{equation*}
As $n\to \infty$ we have that $\Delta Q(\zeta/\sqrt{n\delta})/\delta\to 1$ uniformly on compact subsets of $\C$.
We have shown
\begin{equation*}\label{fhaep}R_{n,1}(z)-\dbar C_n(z)=1+\Delta\log R_{n,1}(z)+o(1),\quad z\in\C.\end{equation*}
Recalling that $R_{n,1}(z)=B_n(z,z)$ we  conclude the proof of Theorem
\ref{THL}.
\end{proof}

\subsection{Ward's equation} \label{vixen}
Let $K=\Psi G$ denote any limiting kernel in
Theorem \ref{TT1}. Referring to a suitable subsequence, we write $$R(z)=\lim_{k\to\infty}R_{n_k}(z)=K(z,z)=\Psi(z,z)$$
for the one-point function. In the following, we shall assume that
$R$ does not vanish identically.

We shall also
use the corresponding \textit{holomorphic kernel}
$$L(z,w):=e^{z\bar{w}}\Psi(z,w)$$
and write
$$L_z(w):=L(w,z).$$

\begin{lem} \label{ma1}  $L(z,w)$ is a positive matrix and $z\mapsto L(z,z)$ is logarithmically
subharmonic, i.e.,  the function $|z|^2+\log R(z)$ is subharmonic.
\end{lem}

\begin{proof} We know that $K(z,w)=L(z,w)e^{-|z|^2/2-|w|^2/2}$ is a positive matrix, i.e. $\sum \alpha_j\bar{\alpha}_k K(z_j,z_k)\ge 0$ for all
choices of points $z_j$ and scalars $\alpha_j$. This means that
$\sum \beta_j\bar{\beta}_k L(z_j,z_k)\ge 0$ where $\beta_j=\alpha_j e^{-|z_j|^2/2}$,
i.e., $L$ is a positive matrix.
Following Aronszajn \cite{Ar} we can then define a semi-definite inner product
on the span of the $L_z$'s by $\left\langle L_z,L_w\right\rangle_*=L(w,z)$. The completion of the span the functions
$L_z$ ($z\in\C$) is a semi-normed Hilbert space $\calH_*$ of entire functions, and the reproducing kernel in this space is $L$.

Since the kernel $L$ is Hermitian-entire, the function $F(z,w)=\left\langle L_w,L_z\right\rangle_*$ is
also Hermitian-entire, and
we have identities such as $\d_z F(z,w)=\left\langle L_w,\dbar_z L_z\right\rangle_*$,
$\dbar_w\d_z F(z,w)=\left\langle \dbar_w L_w,\dbar_z L_z\right\rangle_*$, etc. It follows that
at points where $L(z,z)>0$ we have by the Cauchy-Schwarz inequality that
$$\Lap \log L(z,z)=\frac {\left\|\,\dbar_z L_z\,\right\|_*^{\,2}\cdot\left\|\,L_z\,\right\|_*^{\,2}-\babs{\langle \dbar_z L_z,L_z\rangle_*}^{\,2}}
{L(z,z)^{\,2}}\ge 0.$$
On the other hand, when $L(z,z)=0$ we have $\log L(z,z)=-\infty$. It follows that $\log L(z,z)$ has the
sub-mean value property, and hence is subharmonic.
\end{proof}


\begin{lem} \label{isol} If $R(z_0)=0$ then there is a real-analytic function
$\tilde{R}$ such that
$$R(z)=\babs{\,z-z_0\,}^{\,2}\tilde{R}(z).$$
Moreover, if $R$ does not vanish identically, then all zeros of $R$ are isolated.
\end{lem}

\begin{proof} By Lemma \ref{moi}, we have the mass-one inequality
$$\sum_{n=0}^\infty \frac {\babs{\,\d^j R(z)\,}^{\,2}}{j!}\le R(z),\qquad (z\in\C),$$
so since $R(z_0)=0$ we have $\d^j R(z_0)=0$ for all $j$. However,
$\d^j R(z_0)=\d_1^j \Psi(z_0,z_0)$ and $\dbar^j R(z_0)=\dbar_2^j\Psi(z_0,z_0)$
so the Hermitian function $\Psi(z,w)$ vanishes whenever $(z-z_0)(w-z_0)=0$. Hence we can write $\Psi(z,w)=(z-z_0)(\bar{w}-\bar{z}_0)\Psi_1(z,w)$ where $\Psi_1$ is another Hermitian-entire function. If we define $\tilde{R}(z)=\Psi_1(z,z)$ we now have
$R(z)=\babs{\,z-z_0\,}^{\,2}\tilde{R}(z)$.

To prove the second statement, assume that the zeros of $R$ have an accumulation point, i.e. that there exists a convergent sequence $(z_j)_1^\infty$ of distinct zeros of $R$. Fix a point $w$ and put $\psi_w(z)=\Psi(z,w)$. By the argument above, we have that $\Psi(z,w)=0$ if $z=z_j$, so the holomorphic function $\psi_w$ vanishes
at all points $z_j$, whence $\psi_w$ vanishes identically. Since $w$ was arbitrary,
$\Psi=0$ and hence $R(z)=\Psi(z,z)=0$.
\end{proof}

Note that Lemma \ref{ma1} says that the distribution $1+\Lap\log R$ is a positive measure.

\begin{lem} In the situation of lemma \ref{isol}, the measures
$1+\Lap\log\tilde{R}$ is positive in a neighbourhood of $z_0$.
\end{lem}

\begin{proof} Let $z_0$ be a zero of $R$ and
let $\chi=\1_{D}$ be the characteristic function of some small disc
$D=D(z_0;\eps)$ about $z_0$. Put
$\mu=\chi\cdot (1+\Lap\log R)$ so $\mu$ is a positive measure by the previous lemma. By lemma \ref{isol}, we can write
$\mu=\delta_{z_0}+\chi\cdot (1+\Lap\log\tilde{R})$, so the function
$$S(z):=\log\left(e^{|z|^2}\tilde{R}(z)\right)$$
must satisfy that $\Lap S\ge 0$ in the sense
of distributions on the punctured disc $D'=D\setminus\{z_0\}$.
If $\tilde{R}(z_0)>0$ then $S$ extends analytically to $z_0$ and is hence
subharmonic in $D$.
 Otherwise $S(z_0)=-\infty$. Then $S$ has the sub-mean value property in $D$. Since $S$ is also upper semicontinuous, $S$ is subharmonic in the entire disc $D$ as desired.
\end{proof}

Now let $K=G\Psi$ be a non-trivial limiting kernel. Referring to a suitable subsequence we write $K=\lim c_{n_k}K_{n_k}$, etc., and
$$R(z)=\lim_{k\to\infty}R_{n_k}(z)=K(z,z)=\Psi(z,z).$$
Write $Z$ for the set of isolated zeros of $R$. When $z\not\in Z$ we can write
$$B(z,w)=\frac {\babs{\,K(z,w)\,}^{\,2}}{K(z,z)}$$
and define
\begin{equation}\label{cont}C(z)=\int_\C \frac {B(z,w)}{z-w}\, dA(w).\end{equation}
This clearly defines a smooth function in the complement of $Z$.

\begin{lem} \label{belem} $C_{n_k}$ converges boundedly and locally uniformly on $Z^c$ to $C$ as $k\to\infty$. In particular, $C$ is uniformly bounded on $Z^c$.
\end{lem}

\begin{proof} 

Fix a number $\eps$ with $0<\eps<1$. By 
the locally uniform convergence $c_{n_k}K_{n_k}\to K$ we can pick $N$ such that if $k>N$,
$|\,z\,|<1/\eps$, $\dist(z,Z)\ge \eps$, and $|\,w\,|<2/\eps$, then
$$\babs{\,B_{n_k}(z,w)-B(z,w)\,}<\eps^{\,2}.$$
For $k>N$ and $z$ with $\dist(z,Z)\ge \eps$, $|\,z\,|<1/\eps$ it follows that
\begin{align*}\babs{\,C_{n_k}(z)-C(z)\,}&\le \int\babs{\,\frac {B_{n_k}(z,w)-B(z,w)}
{z-w}\,}\, dA(w)\\
&\le\left(\int_{\babs{\,z-w\,}<1/\eps}+\int_{\babs{\,z-w\,}>1/\eps}\right)
\babs{\,\frac {B_{n_k}(z,w)-B(z,w)}{z-w}\,}\, dA(w)\\
&\le
\eps^{\,2} \int_{\babs{\,z-w\,}<1/\eps}\frac 1 {\babs{\,z-w\,}}\, dA(w) +
\eps\int\babs{\,B_{n_k}(z,w)-B(z,w)\,}\, dA(w)
\le 4 \eps.
\end{align*}

We have shown that the convergence $C_{n_k}\to C$ is uniform on compact subsets of $Z^c$.
We next recall the inequalities
$$B_{n_k}(z,w)\le R_{n_k}(w)\le 1+o(n_k)$$
where $o(n_k)\to 0$ uniformly on compacts as $k\to\infty$. (Note that $R\le 1$, since the mass-one inequality implies $R-R^2\ge 0$).
It follows that
$$\babs{\,C_{n_k}(z)\,}\le (1+o(n_k))\int_{\babs{w-z}\le 1} \frac 1 {\babs{\,w-z\,}}\, dA(w)+\int_{\babs{\,w-z\,}>1}
B_{n_k}(z,w)\, dA(w)\le 3+o(n_k).$$
This proves the uniform bound $\babs{\,C(z)\,}\le 3$ on the complement of $Z$.
\end{proof}

\begin{lem} Suppose that $K$ is non-trivial. Then
Ward's equation
\begin{equation}\label{wwaarr}\dbar C=R-1-\Lap \log R\end{equation} holds in the
sense of distributions.
\end{lem}

\begin{proof}
 By Theorem 4.2, we know that
 $\dbar C_{n}=R_n-1-\Lap\log R_n+o(1)$ where "$o(1)$'' is some function
 which converges to $0$ uniformly on compacts as $n\to \infty$. By Lemma \ref{belem}, the functions $C_{n_k}$ converge to $C$ boundedly and locally uniformly on $Z^c$.
 Since $Z$ is discrete, this implies that $\int C_{n_k}f\, dA\to \int Cf\, dA$ for each test-function $f$, viz. $C_{n_k}\to C$ in the sense of distributions, and also
 $\dbar C_{n_k}\to \dbar C$ in that sense. It follows that the functions $\Lap\log R_{n_k}$ converge in the sense of distributions.
 Since $R_{n_k}\to R$ locally uniformly, they must converge to $\Lap \log R$, which finishes the proof.
\end{proof}

\begin{thm} \label{nt} If $R$ does not vanish identically, then $R>0$ everywhere. Moreover, Ward's equation \eqref{wwaarr} holds pointwise on $\C$.
\end{thm}

\begin{proof} Suppose that $R(z_0)=0$. Let $D=D(z_0,\eps)$ be a small disc about $z_0$ and consider the measures
$\mu=\chi\cdot\left(1+\Lap\log R\right)$ and $\nu=\chi\cdot\left(1+\Lap\log\tilde{R}\right)$ where $\chi=\1_{D}$. By the previous lemmata
we know that the measures $\mu$ and $\nu$ are both positive, and clearly $\mu=\delta_{z_0}+\nu$.
Now consider the Cauchy transform
$$C^\mu(z)=\int_\C\frac 1 {z-w}\, d\mu(w).$$
Evidently
$$C^\mu(z)=\frac 1 {z-z_0}+C^\nu(z),\quad |z-z_0|<\eps,$$
and $\dbar C^\nu=\nu\ge 0$.
Now when $\babs{\,z-z_0\,}<\eps$,
the right hand side in Ward's equation equals
$$R(z)-(1+\Lap\log R)(z)=R(z)-\dbar C^\mu(z),$$
If
$$C(z):=\int \frac {B(z,w)}{z-w}\, dA(w)$$
is the left hand side in Ward's equation, we then have
$$\dbar\left(C(z)+C^\mu(z)\right)=R(z),$$
and hence (by Weyl's lemma)
$$C(z)=-\frac 1 {z-z_0}-C^\nu(z)+v(z)$$
where $v$ is smooth in some neighbourhood of $z_0$. If $C^\mu(z)$ remains bounded as
$z\to z_0$, then $\mu=\nu+\delta_{z_0}$ cannot
contain any point mass at $z_0$, so $\nu$ can be written $-\delta_{z_0}+\rho$ where $\rho(\{z_0\})=0$. This contradicts the fact that $\nu\ge 0$.
Hence
\begin{equation}\label{contr}\babs{\,C(z)\,}\to\infty\quad \text{as}\quad  z\to z_0.\end{equation}
This contradicts the boundedness of $C$ in Lemma \ref{belem}.
Hence $R(z_0)= 0$
is impossible.

We have shown that $R>0$ everywhere. Since Ward's equation $\dbar C=R-1-\Lap \log R$ holds in the sense of distributions and the right hand side is smooth, and application of Weyl's lemma now shows that $C(z)$ is smooth and that Ward's equation holds pointwise.
\end{proof}

\subsection{Holomorphic kernels and complementarity}\label{logs}
We now prove part \ref{TT1.5_3} of Theorem \ref{TT1.5}, i.e., we prove that
if $K=G\Psi$ is a limiting kernel, then the
complementary kernel $\tilde{K}=G(1-\Psi)$ is a positive matrix.

It is convenient to first prove the corresponding properties
for the holomorphic kernel
$$L(z,w):=e^{z\bar{w}}\Psi(z,w).$$
We shall find that $L$ is
the reproducing kernel for a certain Hilbert space $\calH_*$ of entire functions which is contractively embedded in the Fock space. By "Fock space'', we mean the Bergman space $L^2_a(\mu)$
of entire functions square-integrable with respect to the measure
$$d\mu(z)=e^{-|z|^2}\, dA(z).$$
We shall also prove that the \textit{complementary} holomorphic kernel
$\tilde{L}(z,w):=e^{z\bar{w}}(1-\Psi(z,w))$ is a positive matrix. Our proof
of the latter fact depends on a theorem of Aronszajn on differences of reproducing kernels.



To set things up, suppose that we rescale about the boundary point $p=0$ in the positive real direction. Suppose also that $\Lap Q(0)=1$. The rescaling is then simply
$$z=\sqrt{n}\zeta,\qquad w=\sqrt{n}\eta,\quad \text{etc.}$$
Let $A(\zeta,\eta)$ be the Hermitian-analytic function defined in a neighbourhood of the diagonal such that $A(\zeta,\zeta)=Q(\zeta)$. Recall (cf. Section 3) that, along some subsequence, we have
$\Psi=\lim\Psi_{n}$ where
\begin{equation}\label{pn}\Psi_n(z,w):=\frac {K_n(z,w)}{K_n^\#(z,w)}=\frac {\bfk_n(\zeta,\eta)}
{n(\d_1\dbar_2 A)(\zeta,\eta)e^{nA(\zeta,\eta)}}.\end{equation}

By Taylor's formula there is $\delta>0$ such that
\begin{equation}\label{mann}A(\zeta,\eta)=\zeta\bar{\eta}+f(\zeta,\eta)+H(\zeta)+\bar{H}(\eta),
\qquad \babs{\,\zeta\,}<\delta,\,\babs{\,\eta\,}<\delta
\end{equation}
where
$$H(\zeta)=\sum_{j=0}^2\frac {\d_1^j A(0,0)}{j!}\zeta^j,\qquad f(\zeta,\eta)=\sum_{j+k\ge 3}
(\d_1^j\dbar_2^k A)(0,0)\frac {\zeta^j}{j!}\frac{\bar{\eta}^k}{k!}.$$
Let us extend $f$ to a continuous function on $\C^2$ in some way. For example we can require that $f(z,w)=0$
when $\babs{\,\zeta\,}+\babs{\,\eta\,}>3\delta$. The function $H$ is of course well-defined everywhere, being a second-degree
polynomial.
It is important to observe that
$$f(\zeta,\eta)=O\left(\left\|\,(\zeta,\eta)\,\right\|^{\,3}\right)\quad \text{as}\quad (\zeta,\eta)\to (0,0).$$

Put
$$A_n(z,w):=nA(\zeta,\eta)=nA\left(z/\sqrt{n},w/\sqrt{n}\right),$$
and similarly
$$f_n(z,w):=nf\left(z/\sqrt{n},w/\sqrt{n}\right),\quad H_n(z):=nH\left(z/\sqrt{n}\right).$$
The equation \eqref{mann} then gives
\begin{equation}\label{ann}A_n(z,w)=z\bar{w}+f_n(z,w)+H_n(z)+\bar{H}_n(w).\end{equation}

Now define
$$E_n(z,w):=e^{z\bar{w}+f_n(z,w)}$$
and
$$L_n(z,w):=E_n(z,w)\cdot \Psi_n(z,w).$$
Then by \eqref{pn} and \eqref{mann}
$$L_n(z,w)=\frac 1 n\bfk_n\left(\frac z{\sqrt{n}},\frac w{\sqrt{n}}\right)e^{-A_n(z,w)+z\bar{w}+f_n(z,w)}=
k_n(z,w)e^{-H_n(z)-\bar{H}_n(w)}.$$
Next define an $n$-dimensional Hilbert space
$$\calH_n=\left\{f=q\cdot e^{-H_n};\, q\in\calP_n\right\}$$
equipped with the norm of $L^2(\mu_n)$ where
$$d\mu_n(z):=e^{-|z|^2-f_n(z,z)}\, dA(z).$$
Note that $\calH_n$ consists of entire functions.
Finally recall that
$$L(z,w)=e^{z\bar{w}}\Psi(z,w).$$

\begin{lem} \label{ln} $L_n$ is the reproducing kernel for $\calH_n$ and $L_n\to L$ locally uniformly on $\C^2$ as $n\to\infty$. Moreover, for all $z\in\C$ we have
$\int_\C\babs{L(z,w)}^2e^{-|w|^2}\, dA(w)<\infty.$
\end{lem}

\begin{proof} To show that $L_n\to L$ locally uniformly, it suffices to note that
$$E_n(z,w)=e^{z\bar{w}+nf(z/\sqrt{n},w/\sqrt{n})}=e^{z\bar{w}+O(n^{-1/2})},\quad
(n\to\infty),$$
where the $O$-constant is uniform on compact subsets of $\C^2$.
Moreover the mass-one inequality shows that
\begin{align*}\int_\C\babs{\,L(z,w)\,}^{\,2}e^{-\,|\,w\,|^{\,2}}\, dA(w)=e^{\,|\,z\,|^{\,2}}
\int_\C e^{-\,\babs{\,z-w\,}^{\,2}}\babs{\,\Psi(z,w)\,}^{\,2}\, dA(w)\le e^{\,|\,z\,|^{\,2}}\Psi(z,z).\end{align*}

It remains to show that $L_n$ has the reproducing property stated above.
Write $L_{n,w}(z)=L_n(z,w)$.
For an element $f=q\cdot e^{-H_n}$ of $\calH_n$ we then have
\begin{align*}\left\langle f,L_{n,w}\right\rangle_{\calH_n}&=
\int_\C q(z)\,e^{-\,H_n(z)}\,\bar{L}_n(z,w)\, e^{-\,|\,z\,|^{\,2}\,-\,f_n(z,z)}\, dA(z)\\
&=e^{-\,H_n(w)}\,\int_\C q(z)\,\bar{k}_n(z,w)\,e^{-\,H_n(z)\,-\bar{H}_n(z)\,-\,|\,z\,|^{\,2}\,-\,f_n(z,z)}\, dA(z).
\end{align*}
The expression in the exponent equals $-Q_n(z):=-nQ(z/\sqrt{n})$.
Writing $\tilde{q}(\zeta)=q(z)$ and recalling that $\bfk_n$ is the reproducing kernel
for the subspace $\calP_n$ of polynomials of the space $L^2_{nQ}$ normed by
$\left\|\,p\,\right\|_{nQ}^2:=\int \babs{\,p\,}^{\,2}\,e^{-nQ}\, dA$, we now find
\begin{align*}\left\langle f,L_{n,w}\right\rangle_{\calH_n}
&=e^{-\,H_n(w)}\,\left\langle \tilde{q},\bfk_{n,\eta}\right\rangle_{nQ}=e^{-\,H_n(w)}\,\tilde{q}(\eta)=
e^{-\,H_n(w)}\,q(w)=f(w).
\end{align*}
Also $L_{n,w}$ belongs to the space $\calH_n$ since $L_{n,w}(z)=C_w k_{n,w}(z)e^{-H_n(z)}$ where $C_w=e^{-\bar{H}_n(w)}$. The proof of the lemma is complete.
\end{proof}

We can now finish the proof of part \ref{TT1.5_3} of Theorem \ref{TT1.5}.

Let $\calM$ be the algebraic linear span of the kernels $L_z$, $z\in\C$, with semi-definite inner product $\left\langle L_z,L_w\right\rangle_*:=L(w,z)$.
By the zero-one law we can assume that $L(z,z)>0$ for all $z$, so the inner product is
actually a true (positive definite) inner product. By Fatou's lemma and the convergence $L_n\to L$
we now derive a basic inequality (where $d\mu(z)=e^{-\,|\,z\,|^{\,2}}\,dA(z)$)
\begin{align*}\left\|\,\sum_{j=1}^N\alpha_j L_{z_j}\,\right\|_{L^2(\mu)}^{\,2}&\le
\liminf_{n\to\infty}\sum_{j,k=1}^N\alpha_j\bar{\alpha}_k \int_\C L_n(w,z_j)\bar{L}_n(w,z_k)\, d\mu_n(w)\\
&=\liminf_{n\to\infty} \sum_{j,k=1}^N\alpha_j\bar{\alpha}_k L_n(z_j,z_k)
=\sum_{j,k=1}^N \alpha_j\bar{\alpha}_k L(z_j,z_k)\\
&=\left\|\,\sum_{j=1}^N\alpha_jL_{z_j}\,\right\|_*^{\,2},
\end{align*}
so $\calM$ is contained in $L^2(\mu)$ and the inclusion $I:\calM\to L^2(\mu)$ is a contraction. (Moreover, the mass-one equation is equivalent to the statement that $I$
be isometric; see Section \ref{good} for related comments.)

In any case, it follows that the completion $\calH_*$ of $\calM$ can be regarded as a
(possibly non-closed) subspace of $L^2_a(\mu)$. We will write $\calH_\Psi$ for $\calH_*$ and speak of the space (of entire functions) associated to the kernel
$L(z,w)=e^{\,z\bar{w}\,}\,\Psi(z,w)$.

Note that the Fock space $L^2_a(\mu)$ has reproducing kernel $L_0(z,w)=e^{\,z\bar{w}}$.
Let us define a Hermitian entire function $\tilde{L}$ by $\tilde{L}=L_0-L$, i.e.
$$\tilde{L}(z,w):=e^{\,z\bar{w}}\,\left(1-\Psi(z,w)\right).$$

Since the inclusion $I:\calH_*\to L^2_a(\mu)$ is contractive, we can apply a theorem of Aronszajn (\cite{Ar}, Theorem II, p. 355), which implies that the corresponding reproducing kernels then satisfy that the difference $L_0-L$ is a positive matrix.

This implies that the kernel $\tilde{K}=G(1-\Psi)$ is a positive matrix as well, since
$\tilde{K}(z,w)=\tilde{L}(z,w)\,e^{-\,|\,z\,|^{\,2}/2\,-\,|\,w\,|^{\,2}/2}.$
The proof of Theorem \ref{TT1.5} part \ref{TT1.5_3} is complete. q.e.d.



\subsection{Reformulation of Ward's equation}
It is convenient to somewhat reformulate Ward's equation. Given a Hermitian-entire function
$\Psi$ (positive on the diagonal in $\C^2$) we define the functions
$$R(z)=R^\Psi(z)=\Psi(z,z)$$
and
\begin{equation}\label{dex}D(z)=D^\Psi(z)=\int\frac {e^{\,-\,|\,z-w\,|^{\,2}}}{z-w}\babs{\,\Psi(z,w)\,}^{\,2}\, dA(w).\end{equation}
Thus $D=RC$.

\begin{lem} \label{psiprop} Ward's equation \eqref{wwaarr} is satisfied if and only if there exists a smooth function $P(z)$ such that
\begin{equation}\label{I}\dbar P=R-1\end{equation}
and
\begin{equation}\label{II}D=PR-\d R.\end{equation}
\end{lem}

\begin{proof} The equation \eqref{wwaarr} means that
\begin{equation}\label{WA2}\dbar(D/R)=R-1-\dbar(\d R/R).\end{equation}
Let $P_0$ be an arbitrary solution to the equation $\dbar P_0=R-1$. Then \eqref{WA2} becomes
$$\dbar\left[\frac D R-P_0+\frac {\d R} R\right]=0.$$
This last identity is fulfilled if and only if there is an entire function $E$ such that
$$D-P_0R+\d R=ER.$$
Setting $P=P_0+E$, we see that \eqref{I} and \eqref{II} are satisfied. Conversely, if \eqref{I} and \eqref{II}
hold, then
$$\dbar (D/R)=\dbar (P-\d R/R)=R-1-\dbar(\d R/R),$$
i.e. \eqref{WA2} holds.
\end{proof}

\subsection{Relations for the free boundary plasma kernel}

We finish this section by noting the following theorem.

\begin{thm} \label{exx} The kernel $K(z,w)=G(z,w)\erker(z+\bar{w})$
satisfies Ward's equation and the mass-one equation.
\end{thm}

\begin{proof}[Proof 1]
The proof of Ward's equation in Section \ref{vixen} and the example of the Ginibre ensemble in
Section \ref{pome1} shows that Ward's equation is satisfied. The mass-one equation can be deduced
in a similar way; in fact we shall prove in Section \ref{tis} that the mass-one equation is a consequence of Ward's
equation in the translation invariant case. (Another verification of the mass-one equation is found
in \cite{AOC}, Lemma 8.6.)
\end{proof}

\begin{proof}[Proof 2]
 We here give an alternative, direct verification that the function
$R(z)=\erker(2\re z)$ satisfies the mass-one equation \eqref{MOEQ}. To this end, note that
$$\lim_{x\to-\infty}\erker^{(n)}(x)=\delta_{n0},\qquad n=0,1,2,\ldots.$$
Using this, we obtain by differentiating in \eqref{MOEQ}
the \textit{equivalent} equation
\begin{equation}\label{churn}\sum_{n=0}^\infty\frac {\erker^{(n)}(z)\erker^{(n+1)}(z)}
{n!}=\frac 1 2 \erker'(z).\end{equation}
Dividing by $\erker'$, and using the Rodrigues formula for the
Hermite polynomial $h_n$,
$$h_n(z)=(-1)^ne^{z^2/2}\frac {d^n} {dz^n}\left( e^{-z^2/2}\right)=(-1)^n\frac {\erker^{(n+1)}(z)} {\erker'(z)},$$
one can rewrite \eqref{churn} in the form
\begin{equation}\label{churn2}
\erker(z)-\frac 1 2=\erker'(z)\sum_{n=1}^\infty\frac {h_{n-1}(z)h_n(z)}
{n!}\end{equation}
But both sides of \eqref{churn2} have a zero at the origin, so we need
only verify that the derivatives are equal. Using the recursion
$h_n'(z)=nh_{n-1}(z)$, one realizes that our assertion is equivalent to that
\begin{equation}\label{churn3}
\erker'(z)=\erker'(z)\sum_{n=1}^\infty\frac 1 {n!}\left(nh_{n-1}^2(z)+
(n-1)h_{n-2}(z)h_n(z)-zh_{n-1}(z)h_n(z)\right).\end{equation}
However, since $h_0=1$ we have
$$\sum_{n=1}^\infty \frac 1 {n!}\left(nh_{n-1}^2-h_n^2\right)=
\sum_{n=1}^\infty\left\{\frac {h_{n-1}^2}
{(n-1)!}-\frac {h_n^2} {n!}\right\}=1,$$
so the sum in the right hand side of \eqref{churn3} equals
$$1+\sum_{n=1}^\infty\frac 1 {n!}h_n\left(h_n+(n-1)h_{n-2}-zh_{n-1}\right),$$
and this equals $1$ by the recursive definition of Hermite polynomials: $h_0=1$, $h_1=z$, and $h_{n}=zh_{n-1}-(n-1)h_{n-2}$  for $n\ge 2$.
\end{proof}

\section{A priori estimates at regular boundary points} \label{RBP}

In this section we prove asymptotic estimates for the rescaled $1$-point function at a regular boundary point
$p\in\d\drop$.
We rescale about $p$ in the outer normal direction and let $K=G\Psi$
denote an arbitrary limiting kernel in Theorem \ref{TT1}. As usual we write $R(z)=K(z,z)$.


\subsection{Heat kernel estimate} \label{fhke}
Fix a number $\vt<1$ (close to $1$) and a
smooth function
$\psi$ with $\psi=1$ in $D(0;\vt)$ and $\psi=0$ outside $D(0;1)$.
For given $\zeta\in\C$ and $\delta>0$ we define
$\chi$ by $\chi(\omega)=\psi\left((\omega-\zeta)/\delta\right)$. Then $\chi=1$ in $D\left(\zeta;\vt\delta\right)$,
$\chi=0$ outside $D\left(\zeta;\delta\right)$, and the Dirichlet norm
$\left\|\bar{\d}\chi\right\|$
depends only on $\vt$. We sometimes write $\chi_\zeta$
for $\chi$.

We next fix a sequence $(\delta_n)_1^\infty$ of positive numbers in the interval $(2\gamma/\sqrt{n},\delta_0/2)$,
where $\gamma$ is a sufficiently small positive number independent of $n$, and
$$n\delta_n^{\,3}\to0\quad \text{as}\quad n\to\infty.$$

Below
we will fix a point $\zeta$ in $S$.
We shall also use the Hermitian analytic extension
$A(\zeta,\eta)$ satisfying $A(\zeta,\zeta)=Q(\zeta)$ for $\zeta\in \nbh$ (a neighbourhood of $S$).
We assume that $\delta_0$ is small enough that $A(\zeta,\eta)$ is defined whenever $\babs{\zeta-\eta}<2\delta_0$, $\zeta\in S$.

We will use the kernels $\bfK_n$ and $\bfK_n^\#$, where we recall that (cf. Section \ref{bula})
$$\bfK_n^\#(\zeta,\eta)=n\left(\d_1\bar{\d}_2A\right)(\zeta,\eta)\cdot e^{\,nA(\zeta,\eta)}e^{\,-nQ(\zeta)/2}e^{\,-nQ(\eta)/2}.$$
For a fixed $\eta$ we will use abbreviations such as $\bfK_\eta(\zeta)=\bfK_{n,\eta}(\zeta)=\bfK_n(\zeta,\eta)$ etc. Moreover, if
$f$ is a function supported in the domain of the function $A(\cdot,\eta)$, we write
$$\Pi_n^\#f(\eta)=\left\langle f\, ,\, \bfK_\eta^\#\right\rangle:=\int_{\supp f} f(\zeta)\bar{\bfK}_\eta^\#(\zeta)\, dA(\zeta).$$

\begin{thm} \label{T1} For $\zeta\in \bulk S$ define $\delta=\delta(\zeta)=\dist(\zeta,\d S)$.
There is a constant $C$ such that $\delta<\delta_n$
implies
\begin{equation*}\left|\,\bfK_\zeta(\zeta)-\Pi_n\left(\chi_{\zeta}\bfK_\zeta^\#\right)(\zeta)\,\right|\le CM_n(\delta)\sqrt{\,\bfK_\zeta(\zeta)}\end{equation*}
where
\begin{equation}\label{E2}M_n(\delta)=\frac 1 {\sqrt{n\Lap Q(\zeta)}}+\frac 1 \delta e^{\,-n\Lap Q(\zeta)\left(\vt\delta\right)^2/2}.\end{equation}
In particular,
\begin{equation*}\left|\,\bfK_\zeta(\zeta)-\Pi_n\left(\chi_{\zeta}\bfK_\zeta^\#\right)(\zeta)\,\right|\le C\sqrt{n}\,M_n(\delta).\end{equation*}
\end{thm}

The proof relies on the following lemma.

\begin{lem} \label{L1}
If $f=ue^{-nQ/2}$ where $u$ is analytic in $D(\zeta;\delta)$, then
$$\left|\,f(\zeta)-\Pi_n^\#\left[\chi_\zeta f\right](\zeta)\,\right|\le CM_n(\delta)\left\|\,f\,\right\|.$$
\end{lem}

\begin{proof} Assume that $\zeta=0$ and write $\chi=\chi_\zeta$. Then $\Pi_n^\#\left(\chi f\right)(\zeta)$ equals to the integral
$$I=e^{\,-nQ(0)/2}\int\chi(\omega)u(\omega)\left(\d_1\bar{\d}_2 A\right)(0,\omega)\cdot ne^{\,-n\left(A(\omega,\omega)-
A(0,\omega)\right)}\, dA(\omega)$$
which means that
\begin{equation}\label{E3}I=-e^{\,-nQ(0)/2}\int\frac 1 \omega u(\omega)\chi(\omega)F(\omega)\bar{\d}_\omega\left[e^{\,-n\left(A(\omega,\omega)-
nA(0,\omega)\right)}\right]\, dA(\omega),\end{equation}
where
$$F(\omega)=\frac {\omega\left(\d_1\bar{\d}_2A\right)(0,\omega)}
{\bar{\d}_2A(\omega,\omega)-\bar{\d}_2A(0,\omega)}.$$
Since the holomorphic function $H(\xi)=\bar{\d}_2A(\xi,\omega)-\bar{\d}_2A(0,\omega)$ satisfies
$H(\omega)=a\cdot \omega+\Ordo(\omega^2)$ as $\omega\to 0$, where $a=\Lap Q(0)>0$, we have that
$F(\omega)=1+\Ordo(\omega)$ and $\bar{\d}F(\omega)=\Ordo(\omega)$ as $\omega\to 0$. Our assumptions
imply that the $\Ordo$-constant can be chosen independent of $\zeta$.

Integrating by parts in \eqref{E3} one obtains
$$I=f(0)+e^{\,-\,nQ(0)/2}\left(\epsilon_1+\epsilon_2\right)$$
where
$$\epsilon_1=\int\frac {u\cdot\bar{\d}\chi\cdot F} {\omega}e^{\,-\,n\left[A(\omega,\omega)-A(0,\omega)\right]}
\quad ,\quad \epsilon_2=\int \frac {u\cdot\chi\cdot \bar{\d}F} {\omega}e^{\,-\,n\left[A(\omega,\omega)-A(0,\omega)\right]}.$$
It follows that there is a constant $C$ (independent of $\zeta$, $n$ and $\delta$) such that
$$\left|\,\epsilon_1\,\right|\le C\frac 1 {\vt\delta}\int\left|\,u\,\right|\left|\,\bar{\d}\chi\,\right|\, e^{\,-n\left[Q(\omega)-\re A(0,\omega)\right]}
\quad ,\quad |\,\epsilon_2\,|\le C\int\chi\,|\,u\,|\,e^{\,-n\left[Q(\omega)-\re A(0,\omega)\right]}.$$
By the remark after Lemma \ref{cocy} and the assumption $n\delta_n^{\,3}\to 0$, we have the estimate
\begin{equation}\label{E3.5}e^{\,-\,n\left[Q(\omega)/2-\re A(0,\omega)\right]}\le Ce^{\,nQ(0)/2-n\Lap Q(0)\,|\,\omega\,|^{\,2}/2},\quad |\,\omega\,|<2\delta_n.\end{equation}
Using this and the Cauchy-Schwarz inequality, we now find
(since $|\,\omega\,|\ge \vt\delta$ when $\dbar\chi(\omega)\ne 0$)
$$|\,\epsilon_1\,|\,e^{\,-\,nQ(0)/2}\le C\frac 1 \delta e^{\,-\,n\Lap Q(0)\,(\vt\delta)^{\,2}/2}\left\|\,f\,\right\|\left\|\,\bar{\d}\chi\,\right\|\le C\frac 1 \delta e^{\,-\,n\Lap Q(0)\,(\vt\delta)^{\,2}/2}\left\|\,f\,\right\|,$$
and
$$|\,\epsilon_2\,|\,e^{\,-nQ(0)/2}\le C\left\|\,f\,\right\|\left(\int_\C e^{\,-\,n\Lap Q(0)\,|\,\zeta\,|^{\,2}}\right)^{1/2}\le C\frac 1 {\sqrt{n\Lap Q(0)}}\left\|\,f\,\right\|.$$
The proof of the lemma is complete.
\end{proof}

\begin{proof}[Proof of Theorem \ref{T1}]
We have that
$$\overline{\Pi_n^\#\left[\chi_\zeta\bfK_\zeta\right](\zeta)}=
\overline{\left\langle\,\chi_\zeta\bfK_\zeta\,,\,\bfK_\zeta^\#\,\right\rangle}=
\left\langle\,\chi_\zeta\bfK_\zeta^\#\,,\,\bfK_\zeta\,\right\rangle
=\Pi_n\left[\chi_\zeta\bfK_\zeta^\#\right](\zeta),$$
whence
\begin{equation}\label{E4}\left|\,\bfK_\zeta(\zeta)-\Pi_n\left[\chi_\zeta\bfK_\zeta^\#\right](\zeta)\,\right|=
\left|\,\bfK_\zeta(\zeta)-\Pi_n^\#\left[\chi_\zeta\bfK_\zeta\right](\zeta)\,\right|.\end{equation}
It now suffices to take $f=\bfK_\zeta$ in Lemma \ref{L1}, since
$\left\|\,\bfK_\zeta\,\right\|^{\,2}=\bfK_n(\zeta,\zeta)\le Cn$ (cf. Lemma \ref{snor}).
\end{proof}

\subsection{Bergman projection estimate} \label{fbpe}
Recall that $L^2_\phi$ denotes the space of functions $f$ normed by
$\left\|\,f\,\right\|_\phi^{\,2}=\int\left|\,f\,\right|^{\,2}e^{-\phi}$.
We shall let $A^2_\phi$ denote the subspace of $L^2_\phi$ consisting of entire functions.
We write $\pi_\phi$ for the orthogonal (Bergman) projection $L^2_\phi\to A^2_\phi$.

When $\pi$ is the orthogonal projection of a Hilbert space onto a closed subspace, we denote
by $\pi^\bot=I-\pi$ the complementary projection.

Our starting point is a simple
"Hörmander estimate'' (cf. \cite{H}, p. 250): if $\phi$ is smooth and strictly subharmonic in $\C$, and if $u\in C^\infty_0(\C)$, then
\begin{equation}\label{E5.5}\left\|\,\pi_\phi^\bot u\,\right\|_\phi^{\,2}\le
\int_\C\left|\,\bar{\d}u\,\right|^{\,2}\frac {e^{\,-\phi}}{\Lap\phi}.
\end{equation}

\begin{lem} \label{L3}
Fix $\zeta\in \bulk S$. Put
$\delta=\dist(\zeta,\d S)$ and assume that $\delta>2\gamma/\sqrt{n}$.
Then there is a constant $C$  such that
\begin{equation}\label{E5}\left|\,\bfK_\zeta^\#(\zeta)-\Pi_n\left[\chi_\zeta\bfK_\zeta^\#\right](\zeta)\,\right|\le Cne^{\,-\,n\Lap Q(\zeta)\,(\vt\delta)^{\,2}/2}.\end{equation}
\end{lem}

\begin{proof}
Let $g_\zeta(\omega)=\bfK_\zeta^\#(\omega)e^{nQ(\omega)/2}$. Observe that $g_\zeta$ is holomorphic near $\zeta$. Put
$$u(\omega)=\chi_\zeta(\omega)g_\zeta(\omega)-\pi_n\left[\chi_\zeta g_\zeta\right](\omega),$$
and observe that $u$ is a norm-minimal solution in $L^2_{nQ}$ to the problem $\dbar u=\dbar f$ where
$f=\chi_\zeta\cdot g_\zeta$. We shall prove that
\begin{equation}\label{E6}\left\|\,u\,\right\|_{nQ}\le Cn^{-1/2}\left\|\,\bar{\d}\left(\chi_\zeta\cdot g_\zeta\right)\,\right\|_{nQ}.\end{equation}
To do this, we introduce the strictly subharmonic function
$$\phi(\omega)=\eqpot(\omega)+n^{-1}\log\left(1+|\,\omega\,|^{\,2}\right),$$
and consider $v_0=\pi_{n\phi}^\bot\left(\chi_\zeta g_\zeta\right)$. Here $\eqpot$ is the equilibrium
potential, cf. Section \ref{behm}.

By the estimate \eqref{E5.5} we have
$$\left\|\,v_0\,\right\|_{n\phi}^{\,2}\le \int\left|\,\bar{\d}\left(\chi_\zeta\cdot g_\zeta\right)\,\right|^{\,2}\frac {e^{\,-\,n\phi}}
{n\Lap \phi}\, dA.$$
Since $\chi_\zeta$ is supported in $\drop$, and since $\Lap\phi> \Lap Q\ge \text{const.}>0$ there, we obtain
$$\left\|\,v_0\,\right\|_{n\phi}\le Cn^{-1/2}\left\|\,\bar{\d}\left(\chi_\zeta\cdot g_\zeta\right)\,\right\|_{nQ}.$$
Next note the estimate $n\phi\le nQ+\text{const.}$ on $\C$, which is obvious in view of the growth assumption on
$Q$ near infinity. This gives $\left\|\,v_0\,\right\|_{nQ}\le C\left\|\,v_0\,\right\|_{n\phi}$, and we have shown \eqref{E6} with $u=v_0$.

Since $n\phi(\omega)=(n+1)\log|\,\omega\,|^{\,2}+O(1)$ as $\omega\to\infty$ (see Section \ref{behm}), we have the equality
$A^2_{n\phi}=\calP_{n}$ in the sense of sets. Hence $u=v_0$ solves, in addition to
\eqref{E6}, the problem
$$\bar{\d} u=\bar{\d}\left(\chi_\zeta g_\zeta\right)\qquad \text{and}\qquad u-\chi_\zeta g_\zeta\in\calP_{n}.$$

 Since $g_\zeta(\omega)=\bfk_n^\#(\omega,\zeta)e^{\,-\,nQ(\zeta)/2}$ is analytic, we have $\bar{\d}u=\bar{\d}\chi_\zeta\cdot g_\zeta$.
Recalling that $\bfk_n^\#=n(\d_1\bar{\d}_2A)e^{\,nA}$, the remark after Lemma \ref{cocy} gives
$\left|\,\bar{\d} u(\omega)\,\right|^{\,2}e^{\,-\,nQ(\omega)}\le Cn^{\,2}\left|\,\bar{\d}\chi_\zeta(\omega)\,\right|^{\,2}e^{\,-n\Lap Q(\zeta) |\,\omega-\zeta\,|^{\,2}}.$
Since $|\,\omega-\zeta\,|\ge \vt\delta$ when $\bar{\d}\chi_\zeta(\omega)\ne 0$, we find
$$\left|\,\bar{\d}\left(\chi_\zeta g_\zeta\right)(\omega)\,\right|^{\,2}e^{\,-n\,Q(\omega)}\le Cn^{\,2}\left|\,\bar{\d}\chi_\zeta(\omega)\,\right|^{\,2}e^{\,-n\Lap Q(\zeta)\,(\vt\delta)^{\,2}}.$$
We have shown that
\begin{equation*}\label{E7}\left\|\,\bar{\d}\left(\chi_\zeta g_\zeta\right)\,\right\|_{nQ}\le Cne^{\,-n\Lap Q(\zeta)\,(\vt\delta)^{\,2}/2}.\end{equation*}
Applying the estimate \eqref{E6}, one obtains
\begin{equation}\label{E8}\left\|\,u\,\right\|_{nQ}\le C\sqrt{n}\,e^{\,-n\Lap Q(\zeta)\,(\vt\delta)^{\,2}/2}.\end{equation}
We shall now finally use the assumption that $\delta\ge \gamma/\sqrt{n}$. This gives that the function
$u$ is analytic in the disc $D\left(\zeta;\gamma/\sqrt{n}\right)$, so that Lemma \ref{pl2} applies. We obtain that
$$\left|\,u(\zeta)\,\right|^{\,2}e^{\,-\,nQ(\zeta)}\le Cn\left\|\,u\,\right\|_{nQ}^{\,2},$$
where $C$ depends on $\gamma$ and $\Lap Q(\zeta)$.
Combining with \eqref{E8}, we have shown (with a new $C$)
\begin{equation*}\label{EE8}\left|\,g_\zeta(\zeta)-\pi_n\left[\chi_\zeta g_\zeta\right](\zeta)\,\right|e^{\,-\,nQ(\zeta)/2}\le Cne^{\,-\,n\Lap Q(\zeta)\,(\vt\delta)^{\,2}/2}.\end{equation*}
The proof of the lemma is complete.
\end{proof}

The following result is just a restatement of part \ref{TT2_1} of Theorem \ref{TT2}.

\begin{thm} \label{T2} Fix a constant $\el<1/2$. There is then a constant $C$ such that
if $\zeta\in \drop$ and $\delta=\dist(\zeta,\d \drop)$, then
\begin{equation}\label{E10}\left|\,\bfK_{n}(\zeta,\zeta)-\bfK_{n}^\#(\zeta,\zeta)\,\right|\le C\left(1+n e^{\,-\,n\Lap Q(\zeta)\el\, \delta^{\,2}}\right).\end{equation}
\end{thm}

\begin{proof} By Lemma \ref{snor}
we have $\left\|\bfK_\zeta\right\|^{\,2} \le Cn$.
This gives that
\eqref{E10} holds trivially when
$\delta<\gamma/\sqrt{n}$ (because $\bfK_n(\zeta,\zeta)\le Cn$ and $\bfK_n^\#(\zeta,\zeta)\le Cn$ for
sufficiently large $C$). We can thus assume that $\delta>\gamma/\sqrt{n}$.
To this end, we put $\el=\vt^{\,2}/2$.
Then
\begin{align*}\left|\,\bfK_\zeta(\zeta)-\bfK_\zeta^\#(\zeta)\,\right|&\le
\left|\,\bfK_\zeta(\zeta)-\Pi_n\left(\chi_{\zeta}\bfK_\zeta^\#\right)(\zeta)\,\right|+
\left|\,\Pi_n^\bot\left(\chi_{\zeta}\bfK_\zeta^\#\right)(\zeta)\,\right|\\
&\le C\left(1+ne^{\,-\,n\Lap Q(\zeta)\el \delta^{\,2}}\right)+Cne^{\,-n\Lap Q(\zeta)\el \delta^{\,2}}.\end{align*}
where we have used Theorem \ref{T1} to estimate the first term and
the estimate \eqref{E5} to estimate the second one.
\end{proof}

\subsection{An exterior estimate}
Recall from Lemma \ref{snor} that there is a constant $C$ such that
\begin{equation}\label{E18}\bfR_n(\zeta)\le Cne^{\,-\,n\left(Q-\eqpot\right)(\zeta)},\qquad \zeta\in\C.\end{equation}
Now fix a regular boundary point and rescale in the usual way
$$z=e^{-i\theta}\sqrt{n\Lap Q(p)}(\zeta-p),\qquad R_n(z)=\frac 1 {n\Lap Q(p)}
\bfR_n(\zeta).$$

\begin{lem} \label{L4} Suppose that $p$ is a regular boundary point
at distance at least $\delta$ from all singular boundary points, where $\delta>0$ is independent of $n$. There is then a constant $C=C(\delta)$ such that, whenever
$\zeta\in S^c$ and $|z|\le \log n$ we have
$$R_n(z)\le Ce^{-2x^2},\qquad z=x+iy.$$
\end{lem}

\begin{proof} Let $N=e^{i\theta}$ be the outer normal direction at $p$ and let $V$ be the harmonic
continuation of $\check{Q}|_{S^c}$ to a neighbourhood of $p$.
We write $\delta_n=C\log n/\sqrt{n}$.

By Taylor's formula we have, for $x>0$, with $M=Q-V$,
$$M\left(p+xN/\sqrt{n\Lap Q(p)}\right)=\frac 1 {2n\Lap Q(p)} \frac {\d^2 M}{\d n^2}(p)x^2+\frac 1 {6(n\Lap Q(p))^{3/2}}\frac {\d^3 M}{\d n^3}(p+\theta)x^3,$$
where "$\d/\d n$'' is exterior normal derivative (having nothing to do with the integer $n$) and
$\theta=\theta(n,p)$ is some number between $0$ and $\delta_n$. However, since
$p$ is a regular point and $Q=V$ on $\d S$ we have
$\frac {\d^2 M}{\d s^2}(p)=0$ where $\d/\d s$ denotes differentiation in the
tangential direction.
 Adding this to the above Taylor expansion, using that
$(\d_s^2+\d_n^2)M=4\Lap M=4\Lap Q$, we obtain, when $|z|\le C\log n$,
\begin{equation*}nM\left(p+zN/\sqrt{n\Lap Q(p)}\right)=2x^2+O\left(\log^3 n/\sqrt{n}\right),\qquad 0\le x\le \delta_n.\end{equation*}
\end{proof}


It follows from the lemma that each limiting $1$-point function $R(z)=K(z,z)$ at a regular boundary point must satisfy
$R(z)\le Ce^{-2x^2}$ where $x=\re z$. This proves Theorem \ref{TT2}, part \ref{TT2_1}. Theorem \ref{TT2} is therefore
completely proved. q.e.d.

\subsection{The $1/8$-formula} We now prove Theorem \ref{TT2.25}.

Suppose that the droplet $S$ is connected and that the boundary $\d S$ is everywhere regular, so that the theory from \cite{AHM3} applies. Consider
the class $\calC_0$ of test-functions $f\in\coity(\C)$ with $f=0$ on $\d S$. For $f\in\calC_0$ we define functionals
$$\rho_n(f)=\int_\C f\cdot \left(\bfR_n-\left(n\Lap Q+\frac 1 2\Lap\log\Lap Q\right)\cdot\1_S \right)\, dA.$$
 We shall use the main result of the paper \cite{AHM3}, which implies that the limit $\rho(f)=\lim \rho_n(f)$ exists and equals
\begin{equation}\label{com}\rho(f)=\frac 1 {8\pi}\int_{\d S} \frac {\d f} {\d n}\, ds\end{equation}
where $\d f/\d n$ is the exterior normal derivative and $ds$ is arclength measure on $\d S$.

Next fix a parameter $M>0$ and consider the tubular $M/\sqrt{n\Lap Q}$-neighbourhood of $\d S$, defined by
$$\calN_M:=\left\{\,p+\eta\,;\, p\in\d S,\, \babs{\,\eta\,}<M/\sqrt{n\Lap Q(p)}\,\right\}.$$
To simplify, we now assume that $\d S$ is connected; a simple modification will prove the general case.

Now consider an arclength parametrization $p=p(s)$ ($0\le s\le s_0$) of $\d S$.
Also denote by $N_p$ the exterior unit normal at $p\in\d S$.

Define a coordinate system $(s,t)$ where $s$ and
$t$ are real parameters, related to the corresponding point $\zeta=\zeta_n(s,t)\in\calN_M$ by
$$\zeta=p(s)+\frac t{\sqrt{n\Lap Q(p(s))}}
N_{p(s)}.$$
In the $(s,t)$-system, the set $\calN_M$
corresponds to a strip
$$\{(s,t);\, 0\le s\le s_0,\, -M< t<M\}.$$

A simple geometric consideration shows that the area element satisfies the relation
\begin{equation}\label{jac}dA(\zeta)=(1+o(1))\frac 1 \pi\frac 1 {\sqrt{n\Lap Q(p(s))}}\, dsdt,\quad (\zeta=\zeta_n(s,t)\in\calN_M).\end{equation}
Here $o(1)\to 0$ as $n\to\infty$, and the $o(1)$-constant depends on $M$.

The rescaled $1$-point function about $p(s)$ will be denoted
$$R_{n,p(s)}(t):=\frac 1 {n\Lap Q(p(s))}\bfR_n\left(p(s)+\frac t{\sqrt{n\Lap Q(p(s))}}
N_{p(s)}\right),\quad (t\in\R)$$

We now define the functionals
\begin{align*}\rho_{n}'(f)&=\int_{\calN_M}f\cdot \left(\bfR_n-\left(n\Lap Q+\frac 1 2\Lap\log\Lap Q\right)\cdot \1_S\right)\, dA,\\
\rho_n''(f)&=\int_{\C\setminus\calN_M}
f\cdot \left(\bfR_n-\left(n\Lap Q+\frac 1 2\Lap\log\Lap Q\right)\cdot\1_S\right)\, dA.
\end{align*}
Clearly, $\rho_n(f)=\rho_n'(f)+\rho_n''(f)$.


To study $\rho_n'(f)$ we use Taylor's formula in the tubular neighbourhood $\calN_M$,
$$f(\zeta_n(s,t))=\frac {\d f} {\d n}(p(s))\cdot \frac t {\sqrt{n\Lap Q(p(s))}}+O(1/n).$$
This gives
\begin{align*}\rho_n'(f)&=(1+o(1))\frac 1\pi\int_{\d S}\frac {\d f}{\d n}(p(s))\, ds
\int_{-M}^{M}t\cdot(R_{n,p(s)}(t)-\1(t))\, dt,\quad (\chi=\chi_{(-\infty,0)}).\end{align*}
We have here made a change of variables in the integral defining $\rho_n'(f)$ and used relation \eqref{jac}.

Next, the apriori estimates in Theorem \ref{TT2} imply that there are constants $C$, $c>0$
such that
\begin{equation}\label{decay}\babs{\,\bfR_n(\zeta)-n\Lap Q(\zeta)\1_S(\zeta)\,}\le Cne^{-\,c\,n\,\delta(\zeta)^{\,2}},\quad \delta(\zeta):=\dist\left(\zeta,\d S\right).\end{equation}

Now for $f\in\calC_0$ let us write $\tilde{f}(\zeta)=f(\zeta)/\delta(\zeta)$, so
$\tilde{f}\in\coity(\C)$.

The estimates in \eqref{decay} give that there are constants $C$, $C'$ such that
\begin{align*}\babs{\,\rho_n''(f)\,}\le C\left\|\,\tilde{f}\,\right\|_\infty \int_{M}^\infty
t e^{-\,c\,t^{\,2}}\, dt+C'\int_S\babs{\,f\,}\, dA.\end{align*}
By \eqref{com} we now get
\begin{equation}\label{spit}\begin{split}(1+o(1))\frac 1 \pi &\babs{\int_{\d S}\frac {\d f}{\d n}(p(s))\, ds\left[\int_{-M}^M
t\cdot \left(R_{n,p(s)}(t)-\1(t)\right)\, dt-\frac 1 8\right]}=\babs{\,\rho_n''(f)\,}\\
&\le C\left\|\,\tilde{f}\,\right\|_\infty e^{-\,c\,M^{\,2}}+C'\left\|\,f\,\right\|_1.\\
\end{split}
\end{equation}
It is convenient to introduce a notation for the expression appearing in the inner integral,
$$h_{n,M}(s):=\int_{-M}^M
t\cdot \left(R_{n,p(s)}(t)-\1(t)\right)\, dt\,-\,\frac 1 8.$$

\begin{lem} \label{frus} For almost all $s\in [0,s_0]$ we have
$$\lim_{M\to\infty}\lim_{n\to\infty}h_{n,M}(s)=0.$$
\end{lem}

\begin{proof} We define auxiliary functions
$$h^*(s)=\limsup_{M\to\infty}\limsup_{n\to\infty}h_{n,M}(s),\quad
h_*-(s)=\liminf_{M\to\infty}\liminf_{n\to\infty}h_{n,M}(s).$$
We must prove that $h^*(s)\le 0$ and $h_*(s)\ge 0$ for almost every $s$. The two cases are similar, so we just prove the
inequality for $h^*$.

Suppose that the set $E_\alpha:=\left\{h^*>\alpha\right\}$ has positive measure for some $\alpha>0$. Take $\eps\in (0,1)$ to be fixed later.

Since the set $E_\alpha$ has positive measure, it contains a point of density
(see e.g. \cite{Co}, p. 183)
i.e. there is $x\in E_\alpha$ and $\delta>0$ such that
\begin{equation}\label{dens}\left|\,E_\alpha\cap (x-u,x+u)\,\right|\ge (1-\eps)\cdot 2u,\qquad (0<u<\delta).\end{equation}
(Here $|\,E\,|$ denotes Lebesgue measure of a set $E\subset\R$.)

Put $\omega=(x-\delta,x+\delta)$ and $\tilde{\omega}=(x-\delta/2,x+\delta/2)$.
Also pick a large number $P>0$.

 Now fix a test-function $f\in\calC_0$ such that $0\le \d f/\d n\le P$ on $\d S$, $\d f/\d n=0$
outside $\omega$, $\d f/\d n=P$ on $\tilde{\omega}$, and
$\left\|\,f\,\right\|_1\le 1$.

Taking $\limsup_{n\to\infty}$ and then $\limsup_{M\to\infty}$ in \eqref{spit}, we find that
$$\int_{\omega\cap E_\alpha}\frac {\d f}{\d n} (p(s))\,h^*(s)\, ds-
\int_{\omega\setminus E_\alpha}\frac {\d f}{\d n} (p(s))\,h^*(s)\, ds
\le C''.$$
It is easy to see that $h^*$ is bounded from below, say $h^*\ge -m$ so by \eqref{dens},
$$\int_{\omega\setminus E_\alpha}\frac {\d f}{\d n} (p(s))\,h^*(s)\, ds\le
mP\eps\cdot 2\delta,$$
and hence
$$\int_{\omega\cap E_\alpha}\frac {\d f}{\d n} (p(s))\,h^*(s)\, ds\le
2m\eps\delta P+C''.$$
Since $\d f/\d n=P$ on $\tilde{\omega}$ while
$h^*\ge \alpha$ on $E_\alpha$, the integral in the left hand side is at least
$\left|\,E_\alpha\cap\tilde{\omega}\,\right|\cdot P\alpha \ge (1-\eps)\delta P\alpha$. We have shown that
 $$(1-\eps)\delta\alpha P
\le P\cdot m\eps\cdot 2\delta+C''.$$
Choosing $\eps>0$ small enough and $P$ large enough, we can get a contradiction regardless of the values of $\alpha>0$, $m$, and $C''$. The contradiction shows that $\left|\,E_\alpha\,\right|=0$.
\end{proof}

The conclusion of Lemma \ref{frus} means that there exists a set $N\subset\d S$ of arclength zero such that if $p\in(\d S)\setminus N$, then
$$\lim_{n\to\infty}\int_{-\infty}^{+\infty}t\cdot \left(R_{n,p}(t)-\1(t)\right)\, dt=\frac 1 8.$$
Now fix any $p\in (\d S)\setminus N$.
Use compactness to choose a subsequence $n_k$ such that $R_{n_k,p}$ converges locally uniformly to some limit $R_{p}$. We then have
\begin{align*}\lim_{n\to\infty}&\int_{-\infty}^{+\infty}t\cdot \left(R_{n,p}(t)-\1(t)\right)\, dt=\lim_{k\to\infty}\int_{-\infty}^{+\infty}t\cdot \left(R_{n_k,p}(t)-\1(t)\right)\, dt\\
&=\int_{-\infty}^{+\infty}t\cdot \lim_{k\to\infty}\left(R_{n_k,p}(t)-\1(t)\right)\, dt=
\int_{-\infty}^{+\infty}t\cdot \left(R_{p}(t)-\1(t)\right)\, dt.\end{align*}
We have by this completely proved Theorem \ref{TT2.25}.

\section{Translation invariant solutions} \label{tis}

In this section, we study \ti limiting kernels and prove theorems \ref{TT3},
\ref{TT4}, and \ref{TT5}.

\subsection{The Gaussian representation of a \ti limiting kernel} \label{transmetr}
In this section we prove the following result.

\begin{lem} \label{TT2.5} Let $K(z,w)=G(z,w)\Phi(z+\bar{w})$ is an arbitrary \ti
limiting kernel. Then there exists a Borel function $f$ with $0\le f\le 1$ such that $\Phi=\gamma *f$.
\end{lem}

Recall that a limiting kernel $K=G\Psi$ is called \textit{translation invariant} (or \ti) if
$$\Psi(z+it,w+it)=\Psi(z,w),\qquad t\in\R.$$
Let us start the discussion of this important case by proving a simple lemma.

\begin{lem} $\Psi$ is translation invariant if and only if $\Psi(z,w)=\Phi(z+\bar{w})$ for some
entire function $\Phi$.
\end{lem}

\begin{proof} If $\Psi$ is translation invariant, we define $\Phi(z)=\Psi(z,0)$. We must prove that
$$\Psi(z,w)=\Psi(z+\bar{w},0).$$
However, for fixed $z$ both functions are analytic in $\bar{w}$ and they coincide
on the imaginary axis.
\end{proof}

We must prove the representation formula $\Phi=\gamma *f$ where
$\gamma(z)=\frac 1 {\sqrt{2\pi}}\, e^{-\,z^{\,2}/2}$ is the Gaussian kernel and
$f$ is a bounded Borel function with $0\le f\le 1$.

In the following, we fix any limiting holomorphic kernel
$$L(z,w)=e^{z\bar{w}}\Psi(z,w).$$
We shall apply Theorem \ref{TT1.5} part \ref{TT1.5_3}, which states that both $L$ and
$\tilde{L}(z,w)=e^{z\bar{w}}(1-\Psi(z,w))$ are positive matrices.


In the following we will use the convolution operation
$$\gamma *\mu(z)=\int_\R \gamma(z-t)\, d\mu(t),$$
where $\mu$ is a positive measure on $\R$. If $d\mu(t)=f(t)\, dt$ is absolutely continuous, we write, as before, $\gamma*\mu=\gamma*f$.

We now use the positivity of $L$ only on the imaginary axis, to conclude that for all finite subsets $\{x_i\}_1^N\subset \R$ and all complex scalars $\alpha_i$ we have
\begin{equation}\label{mjuk}\sum \alpha_j\bar{\alpha}_k e^{x_jx_k}\Phi(ix_j-ix_k)\ge 0.\end{equation}
Make the substitution $\Phi(iz)=V(z)e^{z^2/2}$. The positivity condition above then becomes
$$\sum \alpha_j\bar{\alpha}_k e^{x_j^2/2+x_k^2/2}V(x_j-x_k)\ge 0,\qquad (x_j)_1^N\subset\R.$$
This can be written
$$\sum \beta_j\bar{\beta}_kV(x_j-x_k)\ge 0,\qquad (\beta_j=\alpha_j e^{x_j^2/2}).$$
The function $V(x)$ is hence positive definite, and by Bochner's theorem
(e.g. \cite{K}),
it is the inverse Fourier transform of a positive measure $\mu$. We have shown that
$$\Phi(ix)e^{(ix)^2/2}=\int_\R e^{itx}\, d\mu(t),\quad x\in\R.$$
Since the entire function $\Phi$ is determined by its values on the imaginary axis, this gives
$$\Phi(z)=\int_\R e^{tz-z^2/2}\, d\mu(t)=\int_\R e^{-(t-z)^2/2}\, e^{t^2/2}d\mu(t),
\qquad z\in\C.$$
Writing $d\nu(t)=\sqrt{2\pi} e^{t^2/2}d\mu(t)$ and $\gamma(t)=\frac 1 {\sqrt{2\pi}}e^{-t^2/2}$
we now have the representation
$$\Phi=\gamma *\nu,$$
where $\nu$ is some positive measure.

Similarly, since the kernel $\tilde{L}(z,w)=e^{z\bar{w}}(1-\Phi(z+\bar{w}))$ is a positive matrix (see Theorem \ref{TT1.5}, cf. Section \ref{logs}), there is a positive measure $\nu_1$ so that
$$1-\Phi=\gamma *\nu_1.$$

The positive measures $\nu$ and $\nu_1$ have the property that
$$\gamma *(\nu+\nu_1)=\Phi+(1-\Phi)=1=\gamma *\1_\R.$$
The map $\rho\mapsto \gamma *\rho$ is 1-1 (as can be seen by taking Fourier transforms), so we obtain
$$\nu+\nu_1=\1_\R.$$
As both $\nu$ and $\nu_1$ are positive, this forces both measures to be absolutely
continuous, $d\nu(x)=f(x)\, dx$ and $d\nu_1(x)=f_1(x)\, dx$ where $f$ and $f_1$ are some non-negative Borel functions with $f(x)+f_1(x)=1$. In particular,
$0\le f\le 1$. By this, the proof of Lemma \ref{TT2.5} is complete. $\qed$

\subsection{T.i. solutions to Ward's equation} \label{setu} We shall now prove
Theorem \ref{TT3}.
Thus we shall find all solutions to Ward's equation
of the special form $K(z,w)=G(z,w)\Phi(z+\bar{w})$ where $\Phi=\gamma *f$ for some bounded Borel function $f$.

Thus assume that $\Phi=\gamma *f$ is an entire function.
It will be convenient to denote the restriction to $\R$ by
$\rest:=\Phi|_\R$, i.e.,
$$\rest(x)=R(x/2),\qquad x\in\R.$$
Observe that a function $V(z)$, which is translation invariant in the sense that $V(z+it)=V(z)$ for all $z\in\C$ and $t\in \R$ satisfies
$\d V=\frac 1 2 \d _x V$.
It is convenient to formulate the following reformulation of Ward's equation in terms of $\rest$:

\begin{lem}\label{wardnew} A \ti kernel $\Psi(z,w)=\Phi(z+\bar{w})$ satisfies Ward's equation if and only
if there exists a smooth function $G$ on $\R$ such that
\begin{equation}\label{dd1}G'=\rest-1,\end{equation}
and
\begin{equation*}\label{dd2}L=G\rest-\rest',\end{equation*}
where
\begin{equation}\label{dd3}L(x)=\int_\C\frac {e^{\,-\,|\,w\,|^{\,2}}}{w}\babs{\,\Phi(x-w)\,}^{\,2}\, dA(w),\qquad x\in\R.
\end{equation}
\end{lem}

\begin{proof} Set $G(x)=P(x/2)$ and $L(x)=D(x/2)$ in Lemma \ref{psiprop}, where we recall that $D(z)$
is defined by the integral \eqref{dex}.
\end{proof}

We will need two elementary lemmas.

\begin{lem} \label{alem1}For all $s,t\in\C$ we have
$$\int_\C e^{\,-\,|\,w\,|^{\,2}}e^{\,iwt}e^{\,i\bar{w}s}\, dA(w)=e^{\,-\,st}.$$
\end{lem}

\begin{proof} A Taylor expansion of $r\mapsto e^{\,ir(te^{\,i\theta}+se^{\,-i\theta})}$ around $r=0$ gives
$$\int_0^\infty e^{\,-\,r^{\,2}}e^{\,i\,r\,\left(te^{\,i\theta}+se^{\,-\,i\theta}\right)}\,r\, dr=\frac 1 2 \sum_{n=0}^\infty
\frac {i^{\,n}\Gamma\left(1+n/2\right)} {n!}\left(te^{\,i\theta}+se^{\,-\,i\theta}\right)^{\,n}.$$
If $n$ is odd, the zeroth Fourier coefficient of $(te^{i\theta}+se^{-i\theta})^n$ vanishes, while if $n$ is even, then
$$\frac 1 {2\pi}\int_0^{2\pi}\left(te^{\,i\theta}+se^{\,-i\theta}\right)^{\,n}\, d\theta=(st)^{\,n/2}\,{n\choose {n/2}}.$$
We have shown that
$$\frac 1 \pi \int_0^{2\pi}d\theta\int_0^\infty e^{\,-\,r^{\,2}}e^{\,i\,r\,\left(te^{\,i\theta}+se^{\,-\,i\theta}\right)}\,r\, dr=
\sum_{k=0}^\infty\frac {\left(-st\right)^{\,k}} {k!}=e^{\,-\,st},$$
finishing the proof of the lemma.
\end{proof}

\begin{lem}\label{alem2} For all $s,t\in\C$ we have
\begin{equation}\label{conte}\int_\C\frac {e^{\,-\,|\,w\,|^{\,2}}} {w}e^{\,itw}e^{\,is\bar{w}}\, dA(w)=i\frac {1-e^{\,-\,st}} s.
\end{equation}
\end{lem}

\begin{proof} Fix $s$ and write $I(t)$ for the left hand side in \eqref{conte}. Then $I(0)=0$
and Lemma \ref{alem1} shows that $I'(t)=ie^{-st}$. It follows that
$$I(t)=i\int_0^{\,t}e^{\,-\,s\tau}\, d\tau=i\frac {1-e^{\,-\,st}} s.$$
The proof of the lemma is complete.
\end{proof}

Since we are assuming that $\Phi=\gamma *f$ for some suitable function $f$, the restriction $\rest$ of $\Phi$ to $\R$ has the structure of the ordinary convolution
$$\rest=\gamma*f.$$
We will use the Fourier transform of the function $\rest$ in a suitable generalized sense:
$$\hat{\rest}:=\hat{\gamma}\cdot\hat{f}$$
where $f$ is regarded as a tempered distribution. This is well-defined, since $\hat{\gamma}=\sqrt{2\pi}\,\gamma$ is a Schwartz test function.

We will frequently use the following consequence of Fourier's inversion theorem
\begin{equation}\label{stes}\Phi(z)=\frac 1 {2\pi}\int_{-\infty}^{+\infty}e^{izt}\hat{\rest}(t)\, dt.\end{equation}
 Here the integral is interpreted as the value of the distribution $\hat{f}$ applied to the Schwarz test-function $t\mapsto \hat{\gamma}(t)e^{izt}$.

With these conventions, 
we conclude that
\begin{equation*}\Phi(x-w)=\frac 1 {2\pi}\int_\R e^{\,i(x-w)t}\hat{\rest}(t)\,dt,\qquad
\Phi(x-\bar{w})=\frac 1 {2\pi}\int_\R e^{\,i(x-\bar{w})s}\hat{\rest}(s)\,ds.\end{equation*}
Multiplying these identities together, we find that
\begin{equation}\label{bbab}\babs{\,\Phi(x-w)\,}^{\,2}=\frac 1 {(2\pi)^2}\iint_{\R^2}e^{\,ix(t+s)}e^{\,-\,iwt}e^{\,-\,i\bar{w}s}\hat{\rest}(t)\hat{\rest}(s)\, dsdt.\end{equation}

Now recall the expression for the function $L(x)$ in
\eqref{dd3}. Using \eqref{bbab} and Lemma \ref{alem2} we have
\begin{align*}L(x)&=\frac 1 {(2\pi)^2}\iint_{\R^2}e^{\,ix(s+t)}\hat{\rest}(t)\hat{\rest}(s)\, dsdt\int_\C \frac {e^{\,-\,|\,w\,|^{\,2}}}
we^{\,-iwt-i\bar{w}s}\, dA(w)\\
&=\frac i {(2\pi)^2}\iint_{\R^2}\frac {e^{\,-\,st}-1} s e^{\,ix(s+t)}\hat{\rest}(t)\hat{\rest}(s)\, dsdt.
\end{align*}
Next note that the relation $\rest=G'+1$ (see \eqref{dd1}) means that
$$\hat{\rest}(s)=is\hat{G}(s)+2\pi\,\delta(s),$$
where $\delta$ is the Dirac measure at $0$. Inserting this in the last expression for $L(x)$ we get
\begin{align*}L(x)&=\frac 1 {(2\pi)^2}\iint_{\R^2}\left(1-e^{\,-\,st}\right)e^{\,ix(t+s)}\hat{\rest}(t)\hat{G}(s)\, dsdt\\
&+\frac i {2\pi}\int_\R\lim_{s\to 0} \frac {e^{\,-st}-1} s \cdot e^{\,ixt}\hat{\rest}(t)\, dt\\
&=G(x)\rest(x)-\frac 1 {(2\pi)^2}\iint_{\R^2}e^{\,-\,st}e^{\,ix(s+t)}\hat{\rest}(t)\hat{G}(s)\, dsdt-\rest'(x),
\end{align*}
where we have used that
$$\frac1 {(2\pi)^2}\iint\sb {\R^2}e^{\,ix(s+t)}\hat{\rest}(t)\hat{G}(s)\, dsdt=\rest(x)G(x),$$
and also that
\begin{align*}-\frac i {2\pi}&\int_\R te^{\,ixt}\hat{\rest}(t)\, dt=-\frac d {dx}\frac 1 {2\pi}\int_\R e^{\,ixt}\hat{\rest}(t)\, dt=-\rest'(x).
\end{align*}

In view of Lemma \ref{wardnew}, Ward's equation is equivalent to that $L=G\rest-\rest'$.
Comparing with the last expression for $L(x)$ we have arrived at the following result.

\begin{lem}\label{warn2} Under the conditions above, Ward's equation is satisfied if and only
if we have $\rest=G'+1$ with a function $G$ such that
\begin{equation}\label{statt}\iint_{\R^2}e^{\,-\,st}e^{\,ix(s+t)}\hat{\rest}(t)\hat{G}(s)\, dsdt=0.\end{equation}
\end{lem}


We now prove Theorem \ref{TT3}.

Recall that $\rest=\Phi\bigm|_\R$ and $\Phi=\barg * f$.
Let $g$ be a continuous function on $\R$ such that
$g'=f-1$; this determines $g$ up to a constant. Let us define $G=g*\barg$. Then
$$G'=g'*\barg=\left(f-1\right)*\barg=\rest-1.$$
By Lemma \ref{warn2}, Ward's equation is equivalent to the identity \eqref{statt}
for a suitable choice of integration constant for $g$. We can rewrite \eqref{statt} in the form
\begin{align*}0&=\iint_{\R^2}e^{\,-\,\left(s+t\right)^{\,2}/2}e^{\,ix(s+t)}\hat{f}(t)\hat{g}(s)\, dsdt\\
&=\int_\R d\xi
\, e^{\,-\,\xi^{\,2}/2}e^{\,ix\xi}\int_\R
\hat{g}\left(\xi-t\right)\hat{f}(t)\, dt\\
&=\int_\R e^{\,ix\xi}e^{\,-\,\xi^{\,2}/2}\left(\,\hat{g}*\hat{f}\,\right)\left(\xi\right)\, d\xi\\
&= \calF^{\,-1}\left[\,\hat{\barg}\cdot\left(\,\hat{g}*\hat{f}\,\right)\,\right](x)\\
&=2\pi\,\left(gf\right)*\barg(x).
\end{align*}
This means that $gf=0$ in the sense of distributions and hence as measurable functions.
Let
$$E=\left\{\,x\in\R;\, g(x)=0\,\right\}.$$
Then $E$ is a closed set, and the complement
$E^c=\R\setminus E$ can be written as a countable
union of disjoint open intervals $I_j$. On each $I_j$ we have $f=0$ and $g'=-1$ almost everywhere.
Since $g=0$ at the endpoints, none of the intervals can be finite. Hence $E$ is connected.
Differentiating the relation $fg=0$ and using $g'=f-1$ we obtain that $f=f^{\,2}$ when $f\ne 0$.
Hence $f=\1_E$ almost everywhere. We have shown that $\Phi$ is representable in the form
$$\Phi(z)=\barg*\1_E(z)=\frac 1 {\sqrt{2\pi}}\int_E e^{\,-\,(z-t)^{\,2}/2}\, dt.$$
The proof of Theorem \ref{TT3} is finished. q.e.d.

\subsection{T.i. limiting kernels at regular boundary points} \label{tir} In this section we prove that the only \ti solution consistent with the apriori conditions listed in Theorem \ref{TT2}
is the BG-kernel.

\begin{thm} \label{morn} Suppose that $K(z,w)=G(z,w)\Phi(z+\bar{w})$ is a t.i. limiting kernel at a regular boundary point. Suppose also that $R(z)=\Phi(z,z)$ satisfies $\int_\R t\cdot (R(t)-\1(t))\, dt=\frac 1 8$.
Then $\Phi=F$ is the free boundary plasma kernel.
\end{thm}

\begin{proof} By Theorem \ref{TT2} we know that that $R(x)\to 1$ as $x\to-\infty$, $R(x)\to 0$ as $x\to +\infty$, and $\int_\R t\cdot (R(t)-\1(t))\, dt$.
 Moreover, by Theorem \ref{TT3} we can write $\Phi=\gamma *\1_{(-\infty,a)}$ for some $a\in\R$.
We must prove that $a=0$.

Let us first prove that
\begin{equation}\label{expoo}\int_{-\infty}^{\infty}tF(2t)\, dt=\frac 1 8\end{equation}
where $F=\gamma*\1$ is the usual plasma function.
To this end, recall that $F(t)=\Prob(X\ge t)$ where $X$ is a standard normal random variable. For suitable differentiable
functions $\fii$, we have
$$\Expe\left[\fii(X)\right]=-\int_\R \fii(t)\, dF(t)=\int_\R \fii'(t)F(t)\, dt.$$
In particular, taking $\fii(t)=t^2$ we find $2\int_\R tF(t)\, dt=\Expe(X^2)=1$, which implies \eqref{expoo}.

We now return to the \ti limiting kernel $R$. We know that $\int_\R t(R(t)-\1(t))\, dt=\frac 1 2$.
As we observed above, we can also write $R(x)=\gamma*\1_{(-\infty,a)}(2x)=F(2x-a)$ for some $a\in\R$ and we must prove that
$a=0$. However, by \eqref{expoo},
$$0=\int_\R t\cdot(R(t)-\1(t))\, dt-\frac 1 8=\int_\R t\cdot (F(2t-a)-F(2t))\, dt.$$
It is easy to see that the right hand side only vanishes when $a=0$, so we must have $\Phi=F$.
\end{proof}

\subsection{Radially symmetric potentials} \label{maincorproof}
We now prove Theorem \ref{TT4}. We start with a simple lemma.

\begin{lem} \label{radss} Assume that $Q$ is radially symmetric.
Fix a point $p\in\d S$
and rescale in the outwards normal direction (see \eqref{E1.3}).
Then every limiting kernel in Theorem \ref{cpthm} takes the form $K=G\Psi$ where
$\Psi(z,w)=\Phi(z+\bar{w})$
is translation invariant.
\end{lem}

\begin{proof} By assumption we have
$$\bfK_n\left(\zeta,\eta\right)=\bfK_n\left(\kappa\zeta,\kappa\eta\right)\qquad \text{where}\quad \kappa=e^{\,it/\sqrt{n}}.$$
We can suppose that $p=1\in \d S$ and we rescale about $p=1$. Set $\delta=\Lap Q(1)$ and
$z=\sqrt{n\delta}(\zeta-1)$, $w=\sqrt{n\delta}(\eta-1)$. Then
$$\kappa\zeta=\left(1+i\frac t{\sqrt{n}}+O\left(\frac 1 n\right)\right)\left(1+\frac z{\sqrt{n\delta}}\right)=1+\frac 1 {\sqrt{n\delta}}\left(z+it{\sqrt{\delta}}\right)+O\left(\frac 1 n\right),\quad (n\to\infty).$$
This means that $\Psi_n(z,w)=\Psi_n\left(z+it\sqrt{\delta}+o(1),w+it\sqrt{\delta}+o(1)\right)$
where $o(1)\to 0$ as $n\to\infty$.
\end{proof}

Now assume that $Q$ is radially symmetric and that the droplet $S$ is connected;
thus it is either a disc or an annulus. If $p\in \d S$ is a boundary point, then
the outer normal $N_p$ is simply a multiple of $p$, $N_p=\pm p/|p|$. We can assume
that $|p|=1$ and $N_p=p$.
Let us write
$$R_{n,p}(t)=\frac 1 {n\Lap Q(p)}\bfR_n\left(p\left(1+\frac t {\sqrt{n\Lap Q(p)}}\right)\right)$$
for the $1$-point function rescaled about $p$. The radial symmetry of $Q$ implies that
$R_{n,p}=R_{n,pe^{i\theta}}$ for all real $\theta$. From this we conclude that the almost-everywhere convergence in Theorem \ref{TT2.25} must hold pointwise, i.e.
if $R=\lim R_{n_k,p}$ is any limiting $1$-point function, then
$$\int_\R t\cdot (R(t)-\1(t))\, dt=\frac 1 8.$$
In view of Lemma \ref{radss} we know that $R$ corresponds to a \ti limiting
kernel $K(z,w)=G(z,w)\Phi(z+\bar{w})$. An application of Theorem \ref{morn}
now shows that $\Phi=F$ is the free boundary plasma function. The proof of Theorem
\ref{TT4} is finished. q.e.d.

\subsection{T.i. solutions to the mass-one equation.} \label{transcorproof}
We now prove Theorem \ref{TT5}.

Let $\Phi$ be an entire function of the form $\Phi=\gamma *f$ where $f$ is some bounded function.

Using Lemma \ref{alem1} and
the assumption that $\Psi(z,w)=\Phi(z+\bar{w})$, we can rewrite
the mass-one equation
(equality in \eqref{mocio}) in terms of the function $\rest=\Phi|_\R$, as follows
\begin{align*}\rest(x)&=\int e^{\,-\,|\,w\,|^{\,2}}\left|\,\Psi\left(x/2,x/2+w\right)\,\right|^{\,2}\, dA(w)
=\int e^{\,-\,|\,w\,|^{\,2}}\babs{\,\Phi(x+w)\,}^{\,2}\, dA(w)\\
&=\frac 1 {(2\pi)^2}\int_\C e^{\,-\,|\,w\,|^{\,2}}\, dA(w)\iint_{\R^2}e^{\,i(x+w)t}
e^{\,i(x+\bar{w})s}\hat{\rest}(t)\hat{\rest}(s)\, dsdt\\
&=\frac 1 {(2\pi)^2}\iint_{\R^2}
e^{\,ix(s+t)}e^{\,-\,st}\hat{\rest}(t)\hat{\rest}(s)\, dsdt.
\end{align*}

(I) If $\rest=\barg *\1_E$ where $E\subset\R$ is a Borel set of positive measure, then $\hat{\rest}(\xi)= e^{\,-\,\xi^{\,2}/2}\cdot \hat{\1}_E(\xi)$ in the sense of distributions. Passing to Fourier transforms we find that the mass-one equation is equivalent to that (with $\delta$ the Dirac delta function)
\begin{align*}e^{\,-\,\xi^{\,2}/2}\,\hat{\1}_E(\xi)&=\frac 1 {2\pi}
\iint_{\R^2}e^{\,-\,st}\delta\left(s+t-\xi\right)e^{\,-\,s^{\,2}/2}\,\hat{\1}_E(s)\,e^{\,-\,t^{\,2}/2}\,\hat{\1}_E(t)
\, dsdt\\
&=\frac {1}{2\pi}
\int e^{\,-\,\xi^{\,2}/2}\,\hat{\1}_E(s)\,\hat{\1}_E\left(\xi-s\right)\, ds\\
&=e^{\,-\,\xi^{\,2}/2}\frac 1 {2\pi}\left[\,\hat{\1}_E*\hat{\1}_E\,\right](\xi).
\end{align*}
By Fourier inversion, this is equivalent to $\1_E=\1_E^{\,2}$, which is true. We have shown that the function $\rest=\barg*\1_E$ satisfies the
mass-one equation.

(II) If $\Psi(z,w)=\Phi(z+\bar{w})$ satisfies the mass-one equation and $\Phi=\gamma *f$, then
the same calculations
as above with "$\1_E$'' replaced by "$f$'' lead to the equation
$$e^{\,-\,\xi^{\,2}/2}\hat{f}(\xi)=e^{\,-\,\xi^{\,2}/2}\frac 1 {2\pi}\left[\,\hat{f}*\hat{f}\,\right](\xi).$$
Taking inverse Fourier transforms we see that this is
equivalent to that $f(x)=f(x)^{\,2}$ almost everywhere.
Hence $f=\1_E$ almost everywhere where $E$ is some measurable set
of positive measure, and $\Phi=\1_E*\barg$.
The proof of Theorem \ref{TT5} is finished. q.e.d.


\section{Concluding remarks} \label{concrem}

In this section, we comment on the nature of Ward's equation and the mass-one equation in the general
(non translation invariant) case, relating those equations to harmonic analysis on the Heisenberg group.
We also explain how the technique of rescaling in Ward identities can be applied in several settings
different from the one we have studied hitherto. Namely, we will derive Ward's equation for the random normal matrix model with a hard edge spectrum, for certain types of bulk singularities, and for so called
$\beta$-ensembles. Finally, we will mention some connections to the theory of Hilbert spaces of entire functions, and to the theories of certain special functions.

\subsection{Twisted convolutions} \label{twistcon} For two functions $f,g$ defined on $\C$, the \textit{twisted convolution}
$f\star g$ is defined by
\begin{equation*}\label{twist}\left(f\star g\right)(z):=\int_\C f(z-w)g(w)e^{\,i\im\left(\bar{z}w\right)}\, dA(w).\end{equation*}
See the book \cite{F}.
We will show that Ward's equation and the mass-one equation have precisely the form of twisted convolution
equations. In the translation invariant case, the equations reduce to ordinary convolution equations, which is how
we were able to solve them. However, the general twisted case is certainly more interesting.

Consider the following transform
\begin{equation*}\label{tran}\hat{f}(t):=\int_\C e^{\,-i\left(z,t\right)}f(z)\, dA(z),\quad t\in\C,\quad
\left(z,t\right):=2\re\left(z\bar{t}\right).
\end{equation*}
Letting $\calF$ be two-dimensional Fourier transform with normalization
$$\calF[f](u+iv)=\frac 1 {2\pi}\int_\C e^{\,-i(xu+yv)}f(x+iy)\, dxdy,$$
we then have
$\hat{f}(t)=2\calF[f](2t)$, and
the inverse Fourier transform takes the form
$$f(z)=\int_\C e^{\,i(z,t)}\hat{f}(t)\, dA(t).$$
Let $K=G\Psi$ denote a limiting kernel in Theorem \ref{TT1} and write $R(z)=\Psi(z,z)$.
Using the transform in a proper generalized sense, we expect
that
there is a function $f$, such that
\begin{equation}\label{organ}\hat{R}=\hat{f}\cdot \Gamma,\quad \text{where}\quad \Gamma(z):=e^{\,-\,|\,z\,|^{\,2}/2}.\end{equation}
Here $\hat{f}$ is understood in the sense of tempered distributions.

Under these conditions, we can assert the following analogue of the identity \eqref{stes}.

\begin{lem} \label{fjuckby} For all $z,w\in\C$ we have
$$\Psi(z,w)=\int_\C e^{\,i\left(z\bar{t}+\bar{w}t\right)}\hat{R}(t)\, dA(t).$$
\end{lem}

Now define $R_0=1-R$, and assume that we can represent $R_0$ in a similar way to \eqref{organ}
\begin{equation}\label{organ2}\hat{R}_0=\hat{g}\cdot \Gamma,\end{equation}
where $g$ is a suitable function.

Lemma \ref{fjuckby} then allows us to rewrite the mass-one and Ward equations as follows.

\subsubsection*{Mass-one equation}
$$\iint_{\C^2}e^{\,-\bar{s}t}e^{\,i(z,s+t)}\hat{R}(s)\hat{R}_0(t)\, dA(s)dA(t)=0,\qquad (z\in\C).$$
Compare with Section \ref{transcorproof} for the translation invariant analogue.
\subsubsection*{Ward's equation} There exists a smooth function $P_0$ such that $\dbar P_0=R_0$ and
$$\iint_{\C^2}e^{\,-\bar{s}t}e^{\,i(z,s+t)}\hat{R}(s)\hat{P}_0(t)\, dA(s)dA(t)=0,\quad (z\in\C).$$
Compare with Lemma \ref{warn2}.

\medskip

Note that both equations take the form
\begin{equation}\label{both}\iint_{\C^2} e^{\,-\bar{s}t}e^{\,i(z,s+t)}\hat{F}(s)\hat{G}(t)\, dA(s)dA(t)=0.
\end{equation}
If we here represent
$$\hat{F}=\hat{f}\cdot\Gamma,\quad \hat{G}=\hat{g}\cdot\Gamma,\quad \Gamma(z)=e^{-|\,z\,|^{\,2/2}},$$
then \eqref{both} takes the form
$$\hat{f}\star \hat{g}=0.$$

\subsection{Ward's equation at the hard edge of the spectrum} \label{whe} For simplicity, we shall restrict our discussion
to the hard edge Ginibre ensemble; we refer to \cite{AKM2} for a discussion of more general hard edge ensembles.

Let $\{\zeta_j\}_1^n$ be the hard-edge Ginibre process and rescale about the boundary
point $p=1$ to obtain the boundary process
$\config_n=\{z_j\}_1^n$, where
$z_j=\sqrt{n}\left(\zeta_j-1\right).$

As before, we let $R_n(z)=K_n(z,z)$ denote the $1$-point function of the rescaled process. The hard edge
Berezin kernel and Cauchy transform are defined, respectively, by
$$B_n(z,w)=\frac {\babs{\,K_n(z,w)\,}^{\,2}} {K_{n}(z,z)},\quad
C_n(z)=\int_\C\frac {B_n(z,w)}{z-w}\, dA(w),$$
with the understanding that $B_n(z,w)=0$ when the point $\zeta=1+z/\sqrt{n}$ satisfies $\babs{\zeta}>1$.

We recall that the hard edge kernel is defined by
\begin{equation}\label{popo}K(z,w)=G(z,w)H(z+\bar{w})\1_\L\otimes\1_\L(z,w),\end{equation}
where $H$ is the hard edge plasma function, $H=\barg *\left(\frac {\1_{(-\infty,0)}} F\right)$.
In terms of this kernel, we put
$$R(z):=K(z,z),\quad B(z,w):=\frac {\babs{\,K(z,w)\,}^{\,2}}
{K(z,z)},\quad C(z):=\int_\L\frac {B(z,w)}{z-w}\, dA(w).$$
Observe that $R(z)=B(z,w)=0$ when $\re z>0$.

\begin{thm} \label{THLp} (Hard edge Ward equation.) The kernel \eqref{popo} gives rise to a solution to the equation
$$\dbar C(z)=R(z)-1-\Lap\log R(z),\qquad z\in\L.$$
\end{thm}

\begin{proof} We claim first that we have the asymptotic relation
\begin{equation}\label{fhae}\dbar C_n(z)=R_n(z)-1-\Lap \log R_n(z)+o(1),\qquad z\in\L,\end{equation}
where the error term $o(1)$ converges to zero uniformly on compact subsets of the left half plane $\L$.

In order to prove this, it is convenient to consider the Ginibre potential $Q(\zeta)=\babs{\,\zeta+1\,}^{\,2}$ which has the droplet $S=\{\,\babs{\,\zeta+1\,}\le 1\,\}$.
We rescale about the boundary
point $p=0$ according to
$$z=\sqrt{n}\,\zeta,$$
fix a number
$\eps>0$. Write
$U:=\drop\cap D(0;\eps)$
and consider
test functions $\psi$ supported in the dilated set $\sqrt{n}\cdot U$. As in
the free case we
define $\psi_n\left(\zeta\right):=\psi\left(z\right)$.
Since $Q^S=Q$ in the set $U$ where $\psi_n$ is supported,
the same arguments used in the free boundary case remain valid (cf. Section \ref{funo}). The only difference is that the dilated domains $\sqrt{n}\cdot U$ will,
in our present case,
increase to the open left half plane $\L$. Hence we deduce the Ward's equation \eqref{fhae} for $z\in \L$ precisely as before.

By Theorem \ref{hardg}, we have convergence $R_n\to R$ and $C_n\to C$ locally uniformly in $\L$ and boundedly
almost everywhere in $\C$. It follows that we can pass to the limit in \eqref{fhae}. The proof is complete.
\end{proof}

\begin{cor} \label{ogood2} $H$ satisfies the following "hard edge mass-one equation'',
\begin{equation}\label{hmass}H\left(2\re z\right)=\int_\L\babs{\,H\left(z+\bar{w}\right)\,}^{\,2}e^{\,-\,\babs{\,z-w\,}^{\,2}}\, dA(w),\qquad (z\in\L).\end{equation}
\end{cor}

\begin{proof}
The approximate Berezin kernels $B_n$ satisfy $\int B_n(z,w)\, dA(w)=1$ for $z\in\L$. The identity \eqref{hmass}
now follows from the convergence $B_n\to B$
in Theorem \ref{hardg} and the argument used in the foregoing proof.
\end{proof}

\subsection{Ward's equation at bulk singularities and Mittag-Leffler fields} \label{frej} Let us weaken our standing assumptions on the potential $Q$. We still require real-analyticity in a neighbourhood of $S$, but
now allow that $\Lap Q=0$ at isolated points
in the bulk of $S$.
A point $p\in \bulk S$ such that $\Lap Q(p)=0$ will be called a \textit{bulk singularity}.

Assume that
$p=0$ is a bulk singularity and let $\{\zeta_j\}_1^n$ be the point process corresponding to $Q$.
The effect of the bulk singularity is to repel the particles away from it.

There are various
types of bulk singularities depending on the local behaviour of $\Lap Q$ near $p$. For instance, if
$$\Lap Q(\zeta)=ax^{\,2}+by^{\,2}+\Ordo\left(\babs{\,\zeta\,}^{\,3}\right),\qquad (\zeta=x+iy\to 0),$$
where $a$ and $b$ are positive constants, then the local behaviour of the system $\{\zeta_j\}$ near $0$
will depend on $a$ as well as $b$. We will here mainly consider the simplest case when $a=b$, and more
generally, that there is a number $\lambda\ge 1$ such that
$$\Lap Q(\zeta)=\babs{\,\zeta\,}^{\,2\left(\lambda-1\right)}+\ldots,\qquad (\zeta\to 0),$$
where the dots represent negligible terms.
If we wish $Q$ to be real-analytic, we should of course assume that $\lambda$ be an integer. However, the condition of real-analyticity is important only in a neighbourhood of the boundary, e.g. in connection
with Sakai's theory. In the bulk it suffices to assume $C^2$-smoothness.
Thus we can in fact we can choose $\lambda$ as an arbitrary real
constant $\ge 1$. Note that $\lambda=1$ is the well-known case of an ordinary "regular'' bulk point,
in which case we know that the usual Ginibre point field arises. We may thus assume that $\lambda>1$.

It turns out that the proper scaling in the case at hand is
\begin{equation}\label{bulcsc}z_j=n^{\,1/(2\lambda)}\,\zeta_j.\end{equation}
We write $\config_n=\{z_j\}_1^n$ for the rescaled system, equipped with the law which is the
image of the Boltzmann-Gibbs measure $\Prob_n$ under the map \eqref{bulcsc}.

\ex Consider the "power potential''
$Q_\lambda(\zeta)=\babs{\,\zeta\,}^{\,2\lambda}$ where $\lambda>1$.
If $\mathbf{K}_n$ denotes a correlation kernel of the process $\{\zeta_j\}_1^n$, then $\config_n$ has the correlation kernel
\begin{equation}\label{lamsc}K_n(z,w)=n^{\,-1/\lambda}
\mathbf{K}_n\lpar \zeta,\eta\rpar,\quad \text{where}\quad z=n^{\,1/(2\lambda)}\,\zeta\quad ,\quad w=n^{\,1/(2\lambda)}\,\eta.\end{equation}

A straightforward calculation shows that the polynomial $\zeta^j$ has norm
$$\left\|\,\zeta^j\,\right\|_{nQ}^{\,2}=\int_\C\babs{\,\zeta\,}^{\, 2j}e^{\,-\,n\,\babs{\,\zeta\,}^{\,2\lambda}}\, dA(\zeta)=
\frac 1 \lambda n^{\,-\frac {j+1}\lambda}\Gamma\left(\frac {j+1} \lambda\right).$$
Inserting the result in formula \eqref{E1.8} for a correlation kernel, we get
$$\mathbf{K}_n(\zeta,\eta)=\lambda n^{\,1/\lambda}\sum_{j=0}^{n-1}\frac {\lpar n^{1/\lambda}\zeta\bar{\eta}\rpar^j} {\Gamma\lpar \frac {j+1} \lambda\rpar}
e^{\,-n\left(\,|\,\zeta\,|^{\,2\lambda}+|\,\eta\,|^{\,2\lambda}\,\right)/2}.$$
Rescaling according to $z=n^{1/2\lambda}\zeta$, $w=n^{1/2\lambda}\eta$, we obtain
$$K_n(z,w)=n^{\,-1/\lambda}\mathbf{K}_n(\zeta,\eta)=\lambda\sum_{j=0}^{n-1}
\frac {(z\bar{w})^j} {\Gamma\lpar \frac {j+1} \lambda\rpar}e^{\,-\left(\,
\babs{\,z\,}^{\,2\lambda}+\babs{\,w\,}^{\,2\lambda}\,\right)/2}.$$
It is now evident that $$K_n(z,w)\to M_\lambda\left(z\bar{w}\right)
e^{-\left(\,\babs{\,z\,}^{\,2\lambda}+\babs{\,w\,}^{\,2\lambda}\,\right)/2},\quad (n\to\infty),$$ locally uniformly in $\C^2$,
where $M_\lambda$ is the function
 \begin{equation*}\label{mlf}M_\lambda(z)=\lambda\sum_0^\infty \frac {z^{\,j}}{\Gamma\left(\frac {j+1}\lambda\right)}.\end{equation*}
We recognize $M_\lambda$ as the
generalized Mittag-Leffler function $\lambda E_{1/\lambda,1/\lambda}$. See \cite{Er0}, Vol. 3.

\begin{thm} \label{PP3} The point process $\config_n$ converges as $n\to\infty$ to the unique point-field $ML_\lambda$ in $\C$ with kernel
\begin{equation*}\label{pfi3}K(z,w)=M_\lambda\left(z\bar{w}\right)e^{\,-
\frac 1 2 \lpar\, \babs{\,z\,}^{\,2\lambda}
+\babs{\,w\,}^{\,2\lambda}\,\rpar}.\end{equation*}
The convergence holds in the sense of locally uniform convergence of intensity functions.
\end{thm}

\begin{proof} It is easy to see that $M_\lambda$ is of exponential type $\lambda$. This implies that
the kernel $K$ is uniformly bounded.
Existence and uniqueness of a point field $ML_\lambda$ with the given properties now follows,
via Lenard's theory, from the convergence of intensities in the preceding example (cf. the appendix).
\end{proof}

We next consider Ward's equation at $p=0$ for the
the potential $Q_\lambda(\zeta)=\babs{\,\zeta\,}^{\,2\lambda}$. To this end, we introduce the
Berezin kernel rescaled about $0$ on the scale \eqref{bulcsc},
i.e.
$$B_{n}(z,w):=\frac{\babs{\,K_n(z,w)\,}^{\,2}}
{K_n(z,z)}.$$
Ward's equation takes the following form.

\begin{thm} In the above situation, $B_n\to B$ uniformly on compact subsets of $\C^2$ where
$B$ is a solution to the "Ward equation of type $\lambda$''
 \begin{equation*}\label{modw}\dbar_z\int\frac {B(z,w)}
 {z-w}dA(w)=
 B(z,z)-\lambda^2\babs{z}^{\,2\left(\lambda-1\right)}-\Delta_z\log B(z,z).\end{equation*}
\end{thm}

\begin{proof} We shall first establish the asymptotic relation
\begin{equation}\label{RS}\dbar_z\int\frac {B_{n}(z,w)}
 {z-w}dA(w)=
 B_{n}(z,z)-\lambda^2\babs{\,z\,}^{\,2(\lambda-1)}-\Delta_z\log B_{n}(z,z)+o(1),\end{equation}
 where $o(1)\to 0$ as $n\to\infty$, uniformly on compact subsets of $\C$.

To this end, fix a
test-function $\psi$ and let
$\psi_n\left(\zeta\right)=\psi\left(z\right)$ where $z=n^{\,1/2\lambda}\,\zeta$.

We shall use Ward's identity; we therefore
recalculate the expectations of the terms
$I_n[\psi_n]$, $II_n[\psi_n]$ and $III_n[\psi_n]$ used in the free boundary
case, in Section \ref{funo}. As customary, we use the symbol
$\bfR_{n,k}$
to denote the $k$-point function of the system $\{\zeta_j\}_1^n$. The rescaling $z_j=n^{1/2\lambda}\zeta_j$
then implies that the $k$-point function of the rescaled system $\{z_j\}_1^n$ is
$$R_{n,k}\left(z_1,\ldots,z_k\right)=n^{\,-k/\lambda}\,\bfR_{n,k}\left(\zeta_1,\ldots,\zeta_k\right).$$

For $I_n[\psi_n]$, the change of variables in \eqref{lamsc} gives that
\begin{equation*}\label{ya}\begin{split}
\Expe_nI_n\left[\psi_n\right]&=n^{\,1/2\lambda}\frac 1 2
\iint \frac {\psi(z)-\psi(w)}
{z-w}R_{n,2}(z,w)\\
&=n^{\,1/2\lambda}\int\psi(z)\,dA(z)\int
\frac {R_{n,2}(z,w)}
{z-w}\,dA(w).\\
\end{split}
\end{equation*}

Turning to $II_n[\psi_n]$, we first observe that
$$\d Q_\lambda\left(\zeta\right)=\lambda\bar{\zeta}\cdot \babs{\,\zeta\,}^{\,2(\lambda-1)}.$$
Using this, we see that
$$n\d Q_\lambda\lpar \frac z {n^{\,1/2\lambda}}\rpar=n^{\,1/2\lambda}
\d Q_\lambda(z),$$
which gives
$$\Expe_nII_n\left[\psi_n\right]=n^{\,1/2\lambda}\int
\lambda\bar{z}\babs{\,z\,}^{\,2(\lambda-1)}\cdot\psi\cdot
R_{n,1}.$$
We also compute
$$\Expe_{n}III_n\left[\psi_n\right]
=n^{\,1/2\lambda}\int\d\psi\cdot R_{n,1}=
-n^{\,1/2\lambda}\int\psi\cdot \d R_{n,1}.$$

In view of Ward's identity (Section \ref{theidentity}) we now infer that, in the
sense of distributions,
$$\int\frac {R_{n,2}(z,w)}
{z-w}dA(w)=\lambda\bar{z}\babs{\,z\,}^{\,2(\lambda-1)}
\cdot R_{n,1}(z)+
\d R_{n,1}(z).$$
Dividing by $R_{n,1}$ and applying $\dbar$,
we conclude the proof of the formula \eqref{RS}.
To pass to the limit as $n\to \infty$, we now use the convergence in the example preceding Theorem \ref{PP3}
and the argument in Section \ref{vixen}.
\end{proof}

\begin{rem} Consider a potential $Q(x+iy)=p(x,y)$ where $p$ is a polynomial in $x$ and $y$, positive definite
and homogeneous of degree $2k$ where $k$ is a positive integer.
 Write $\zeta=x+iy$ and rescale by $z=n^{\,1/2k}\,\zeta$. As in the above proof, one deduces
without difficulty the asymptotic relation
\begin{equation}\label{akk}\dbar_z\int\frac {B_{n}(z,w)}
 {z-w}dA(w)=
 B_{n}(z,z)-\Lap Q(z)-\Delta_z\log B_{n}(z,z)+o(1).\end{equation}
Another equation of the type \eqref{akk}
was studied in the paper \cite{AK}.
\end{rem}

We now consider "kernels''  of the form
\begin{equation}\label{beform}B_E(z,w)=\frac {\babs{\,E\left(z\bar{w}\right)\,}^{\,2}} {E\left(\,\babs{\,z\,}^{\,2}\,\right)}
e^{\,-\,\babs{\,w\,}^{\,2\lambda}},\end{equation}
where $E$ is an entire function, real and positive on $(0,\infty)$.
We refer to $B_E$ as
a \textit{(generalized) Berezin kernel of type $\lambda$} (or of the "second kind'' as in \cite{AK}). We say that the entire function $E$ satisfies
the \textit{mass-one equation of type $\lambda$} if
$$\int B_E(z,w)\, dA(w)=1,\qquad z\in\C.$$

The following theorem appears in the paper \cite{AK}.

\begin{thm}\label{muniq} Let $E=M_\lambda$.
Then $B_E$ satisfies the mass-one equation of type $\lambda$.
Furthermore a kernel $B_E$ of the form \eqref{beform}, where $E$ is an entire function which is positive on $\R_+$
satisfies type-$\lambda$ mass-one and Ward equations if and only if $E=M_\lambda$.
\end{thm}

Finally, let $Q$ be a potential of the form
$Q(\zeta)=\babs{\,\zeta-p\,}^{\,2\lambda}+\ldots$ where $p\in\bulk S$ and the dots represent negligible terms near
$\zeta=p$.
Let $\config_n$ be the corresponding process rescaled about $p$ by a factor
$n^{\,1/(2\lambda)}$ about $p$. We conjecture that $\lim_{n\to\infty}\config_n= ML_\lambda$ in the sense of point fields.

\subsection{The mass-one equation and Hilbert spaces of entire functions} \label{good}
In this section, we shall interpret the mass-one equation as the reproducing property
in a suitable space of entire functions. As a consequence, we shall find non-trivial relations for the functions
$F$, $H$, and $M_\lambda$.

\smallskip

It has been observed (e.g. \cite{Bl,L0,L} and the references there) that universality
laws in the theory of random \textit{Hermitian} matrices are related to certain specific \textit{de Branges}
spaces $\calB(E)$ of entire functions. See \cite{DB} for the definition of these spaces. In particular, the
sine-kernel describing the spacing of eigenvalues in the bulk is the restriction to $\R^2\subset\C^2$
of the reproducing kernel of the Paley-Wiener space, i.e., the space $\calB(E)$ where $E(z)=e^{-i\pi z}$.
Moreover, the Airy kernel (see \eqref{airker}) which describes the spacing at the edge of the spectrum is the restriction
to $\R^2$ of the reproducing kernel of $\calB(E)$ where $E=\Ai'-i\Ai$, and the Bessel kernel (hard edge)
is the restriction to $(-\infty,0)\times(-\infty,0)$ of the reproducing kernel of the de Branges space corresponding to the function
$E(z)=\sqrt{z}J_0'(\sqrt{z})-iJ_0(\sqrt{z})$.

The appearance of de Branges spaces in the context of Hermitian random matrices is quite natural given the
fact that orthogonal polynomials on the real line can be related to a second order one-dimensional
self-adjoint spectral problem.

The Hilbert spaces $\calH$ of entire functions arising in the random normal matrix theory are not
of de Branges type, and we are not sure about their spectral interpretation. Nevertheless, we will
use the term "spectral measure'': $\mu$ is a spectral measure for $\calH$ if $\calH$ sits isometrically
in $L^2(\mu)$.

\begin{lem} \label{satur} Let $\Psi$ be a Hermitian entire function and $G=G(z,w)$ the Ginibre kernel. The following conditions are equivalent.
\begin{enumerate}[label=(\roman*)]
\item \label{popp}The kernel $K=G\Psi$ satisfies the mass-one equation, i.e.,
\begin{equation}\label{mqq}\int e^{\,-\,|\,w\,|^{\,2}}\babs{\,\Psi(z,z+w)\,}^{\,2}\, dA(w)=\Psi(z,z),\quad z\in\C.\end{equation}
\item \label{top}The holomorphic kernel
$L(z,w)=e^{\,z\bar{w}}\Psi(z,w)$
is the reproducing kernel of some Hilbert space $\calH$ with spectral measure $d\mu(z):=e^{\,-\,|\,z\,|^{\,2}}\, dA(z)$.
\end{enumerate}
If this is the case, then there is a unique point field with correlation kernel $K$.
\end{lem}

\begin{proof} Write $L_w(z)=L(z,w)=e^{\,z\bar{w}}\Psi(z,w)$ (cf. Section \ref{vixen}.)
The function $L$ is the reproducing kernel for a Hilbert space with spectral measure
$d\mu(z)=e^{\,-\,|\,z\,|^{\,2}}\, dA(z)$
if and only if
\begin{equation}\label{starre}\left\langle\, L_w\,,\,L_z\,\right\rangle_{L^2(\mu)}=L(z,w),\qquad z,w\in \C.
\end{equation}

For $z=w$, the identity \eqref{starre} means that
$$\int_\C\babs{\,L_z(\zeta)\,}^{\,2}e^{\,-\,|\,\zeta\,|^{\,2}}\, dA(\zeta)=e^{\,|\,z\,|^{\,2}}\Psi(z,z),$$
or
$$\int_\C\babs{\,\Psi(\zeta,z)\,}^{\,2}e^{\,\zeta\bar{z}\,+\,\bar{\zeta}z\,-\,|\,\zeta\,|^{\,2}\,-\,
|\,z\,|^{\,2}}\, dA(\zeta)=\Psi(z,z),$$
which is precisely the mass-one equation \eqref{mqq}. On the other hand, if the last equation holds, then
\eqref{starre} follows for $z\ne w$ by analytic continuation. This proves the equivalence of
\ref{popp} and \ref{top}.

Next note that the kernel $K=G\Psi$ can be written
$$K(z,w)=e^{\,-\,|\,z\,|^{\,2}/2\,-\,|\,w\,|^{\,2}/2}L(z,w).$$
From this we conclude that if $L$ gives rise to a reproducing kernel as in \ref{top}, then
$K$ is the reproducing kernel of the subspace $\calW=\left\{\,f;\, f(z)=g(z)e^{\,-\,|\,z\,|^{\,2}/2}\,\right\}$
of $L^2$.

Consider the integral operator $T$ on $L^2$ with kernel $K$. It is easy to check that this
operator satisfies the following conditions: $T$
is a Hermitian operator which satisfies $0\le T\le 1$, and is locally trace class.
(That an operator $T$ on $L^2$ is "locally trace class'' means that the operator
$T_B$ on $L^2$ defined by $T_B(f)=T(\1_B f)$ is trace class for every compact set $B\subset\C$.)

By a theorem of Soshnikov (Theorem 3 in \cite{S}), the conditions above guarantee that $K$ is the
correlation kernel of a unique random point field in $\C$.
\end{proof}

In the following, we write
$A^2(\mu)$ for the space of all \textit{entire} functions of class $L^2(\mu)$.
It follows from general facts for reproducing kernels that the Hilbert space $\calH$ in \ref{top}
is the closed linear span
$$\calH=\spann_{L^2(\mu)}\left\{\,L_w;\, w\in\C\,\right\},\quad (d\mu(z)=e^{\,-\,|\,z\,|^{\,2}}\, dA(z)).$$
Let us look at some examples.

For the (bulk) Ginibre process we have $L(z,w)=e^{\,z\bar{w}}$, and hence $\calH=L^2_a(\mu)$ is the Fock space.
The free boundary Ginibre process corresponds to the kernel
$L(z,w)=e^{\,z\bar{w}}F(z+\bar{w})$, and hence
$$\calH=\spann_{L^2(\mu)}\left\{\,e^{\,\bar{w}z}F(z+\bar{w})\,;\, w\in\C\,\right\},$$
where $F$ is the free boundary plasma function.

One can similarly interpret the hard edge mass-one equation \eqref{hmass} as a reproducing
property in a suitable space of entire functions. In fact, this space is
$$\calH=\spann_{L^2(\mu_h)}\left\{\,e^{\,\bar{w}z}H(z+\bar{w})\,;\, w\in\L\,\right\},\quad (d\mu_h(z)=e^{\,-\,|\,z\,|^{\,2}}\cdot\1_\L(z)\, dA(z)),$$
where $H$ is the hard edge plasma function. The fact that the last span consists of entire functions
requires a compactness property in the hard edge situation, which will be established in the
paper \cite{AKM2}.

It would be interesting to describe the above spaces in more constructive terms (e.g. similar to
de Branges theory).
It would also be interesting to know the meaning of Ward's equation for the spaces $\calH$.
(By Lemma \ref{satur}, the mass-one equation is a statement about spectral measures.)

\smallskip

We finally describe the Hilbert spaces corresponding to the Mittag-Leffler processes.
To this end, recall (Theorem \ref{muniq}) that the mass-one equation for the function $M_\lambda$
says that
$$\int_\C\babs{\,M_\lambda(z\bar{w})\,}^{\,2}e^{\,-\,|\,w\,|^{\,2\lambda}}\, dA(w)=M_\lambda
\left(\,|\,z\,|^{\,2}\,\right).$$

This gives, by polarization
\begin{equation*}\label{mp}\int_\C M_\lambda\left(z_1\bar{w}\right)M_\lambda\left(\bar{z}_2w\right)e^{\,-\,|\,w\,|^{\,2\lambda}}\, dA(w)=M_\lambda\left(z_1\bar{z}_2\right).\end{equation*}
(This formula has an alternative, elementary proof: insert
$M_\lambda(\zeta)=\lambda\sum {\zeta^{\,j}}/ {\Gamma\left(\frac{j+1}\lambda\right)}$ in the left hand side
and integrate termwise.)

Let $d\mu_\lambda(z)=e^{\,-\,|\,z\,|^{\,2\lambda}}\, dA(z)$.
The Hilbert space pertaining to the process $ML_\lambda$ is thus
$$\calH:=\spann_{L^2(\mu_\lambda)}\left\{\,z\mapsto M_\lambda(z\bar{w})\,;\, w\in\C\,\right\}.$$
It is not hard to show that polynomials are dense in $\calH$, and consequently
$\calH=A^2(\mu_\lambda)$.

\begin{rem} For a Borel set $e\subset\R$ of positive measure, let
$$F_e(z):=\barg*\1_e(z)=\frac 1 {\sqrt{2\pi}}\int_e e^{\,-\,(z-t)^{\,2}/2}\, dt.$$
Thus $F_{(-\infty,0)}$ is the free boundary plasma function $F$.
By Theorem \ref{TT5} we know that the kernel
$$K_e:=G\Psi_e ,\quad \Psi_e(z,w)=F_e\left(z+\bar{w}\right)$$
satisfies the mass-one equation \eqref{mqq}. The corresponding Hilbert space
is
$$\calH=\spann_{L^2(\mu)}\left\{\,z\mapsto e^{\,z\bar{w}}F_e(z+\bar{w});\, w\in\C\,\right\},\quad (d\mu(z)=
e^{\,-\,|\,z\,|^{\,2}}\, dA(z)),$$
and the weighted version
$\calW$, is the closed linear span in $L^2$ of the kernels
$K_w(z)=G(z,w)F_e\left(z+\bar{w}\right)$.
Then $K_w$ is the reproducing kernel in $\calW$, i.e.,
\begin{equation}\label{polar}\int_\C G(t,z)F_e(t+\bar{z})G(w,t)F_e(w+\bar{t})\, dA(t)=G(w,z)F_e(w+\bar{z}),\qquad z,\,w\in\C.\end{equation}
This can be regarded as a polarized version of the mass-one equation. At the same time, \eqref{polar}
gives a quite non-trivial relation for the function $F_e$.

The positivity property of the kernel $K(z,w)=G(z,w)F_e(z+\bar{w})$ also implies
non-trivial inequalities.
We give a few examples in the case $e=(-\infty,0)$.

The inequality $\babs{\,K(z,w)\,}^{\,2}\le R(z)R(w)$,
holds for the kernel of any determinantal process, and in particular for $K$. It implies that
\begin{equation}\label{ecu}\babs{\,F(z+\bar{w})\,}^{\,2}\le e^{\,|\,z-w\,|^{\,2}}F(z+\bar{z})
F(w+\bar{w}).\end{equation} When $w=0$, this gives
$$\babs{\,F(z)\,}^{\,2}\le \frac 1 2 e^{\,|\,z\,|^{\,2}}F(z+\bar{z}).$$
Setting $w=-z=\frac 1 2(x+iy)$ in \eqref{ecu} and using that $F(-x)=1-F(x)$ we get
$$\babs{\,F(iy)\,}^{\,2}\le e^{\,x^{\,2}+y^{\,2}}\,F(x)(1-F(x)).$$
Letting $y=0$ this gives the inequality
$F(x)-F(x)^{\,2}\ge e^{\,-\,x^{\,2}}/4$, ($x\in\R$).

We also mention the following polarized form of the hard edge mass-one equation
\begin{equation*}\label{polar2}\int_\L G(t,z)H\left(t+\bar{z}\right)G(w,t)H\left(w+\bar{t}\right)\, dA(t)=G(w,z)H(w+\bar{z}),\qquad z,\,w\in\L.\end{equation*}
In a similar way as for $F$, one can explore some consequences of the inequality $\babs{\,H(z+\bar{w})\,}^{\,2}\le e^{\,|z-w\,|^{\,2}}H(z+\bar{z})H(w+\bar{w})$ for $z,w\in\L$, which follows from the positivity of
the hard edge correlation kernel. We note first that
\begin{equation*}\label{log2}\hfun(0)=-\int_{-\infty}^{\,0}\frac d {dt}\log F(t)\, dt=-\log F(0)=\log 2,\end{equation*}
Setting $w=0$ in the preceding inequality now gives that
$$\babs{\,H(z)\,}^{\,2}\le e^{\,\babs{\,z\,}^{\,2}}H\left(z+\bar{z}\right)\log 2,\quad z\in\L.$$
Letting $z=0$ shows that the estimate is sharp.

As a final example, we note that the positivity of the Mittag-Leffler kernel
implies that
$$\babs{\,M_\lambda\left(z\bar{w}\right)\,}^{\,2}\le M_\lambda\left(\,|\,z\,|^{\,2}\,\right)M_\lambda\left(\,|\,w\,|^{\,2}\,\right).$$
\end{rem}

\subsection{Sections of power series} \label{sect} It seems that the type of asymptotics one encounters for the free boundary was
first observed in connection with sections of power series of the exponential function.
By a \textit{section} of an entire function
$$f(\zeta)=\sum_{j=0}^\infty a_j\zeta^{\,j}$$
we here simply mean a partial sum
$$s_n(\zeta)=\sum_{j=0}^{n-1} a_j\zeta^{\,j}.$$
Szeg\H{o}'s original study in \cite{Sz} concerns the distribution of zeros of the blow-up sections $s_n^\sharp(w):= s_n(nw)$
pertaining to the exponential function $f(\zeta)=e^{\,\zeta}$. In the course of the investigation,
 Szeg\H{o} proves asymptotic results for the
function $s_n^\sharp(w)$ valid for all $w$ \textit{except} for $w$ in a fixed neighbourhood of $1$.
This gap was later closed, and the following result ensued.
Consider the rescaled section
\begin{equation*}\label{szego1}\tilde{s}_n(z):=s_n\left(n\,\zeta\right),\qquad z=\sqrt{n}\left(\zeta-1\right).\end{equation*}
One then  has the following (locally uniform) convergence
\begin{equation}\label{wopp}\tilde{s}_n(z)e^{\,-n-\sqrt{n}z}\to e^{\,z^{\,2}/4}F(z)\qquad (n\to\infty),\end{equation}
where $F$ is the free boundary plasma function.
We are unsure concerning whom should be credited for the convergence in \eqref{wopp} when $f(\zeta)=e^\zeta$.
However, the book \cite{ESV} (Theorem 1) contains a statement valid for more general $f$, and the appendix in \cite{BM} contains
a very detailed convergence result for the case at hand.

To see the connection with the scaling limits of the present paper, we remind of the expression for the
correlation kernel for the Ginibre ensemble from Section \ref{pome},
$$\bfK_n(\zeta,\eta)=n\sum_{j=0}^{n-1}\frac {(n\zeta\bar{\eta})^j}{j!}e^{\,-n\frac {\babs{\,\zeta\,}^{\,2}+\babs{\,\eta\,}^{\,2}} 2}.$$
Noting that
$$\bfK_n(\zeta,\eta)=ns_n(n\zeta\bar{\eta})e^{\,-n\frac {\babs{\,\zeta\,}^{\,2}+\babs{\,\eta\,}^{\,2}} 2},$$
and rescaling via
$$z=\sqrt{n}(\zeta-1),\quad w=\sqrt{n}(\eta-1),\quad K_n(z,w)=n^{-1}\bfK_n(\zeta,\eta),$$
we now recognize that
$$K_n(z,w)=\tilde{s}_n(z+\bar{w}+o(1))\,e^{\,-n}\,e^{\,-\sqrt{n}\,\re(z+w)}\cdot e^{\,-\,|\,z\,|^{\,2}/2\,-\,|\,w\,|^{\,2}/2}.$$
Letting $c_n$ be the cocycle $c_n(z)=e^{-i\sqrt{n}\im z}$ we now form the kernels
$$\tilde{K}_n(z,w)=c_n(z)\bar{c}_n(w)K_n(z,w)=\tilde{s}_n(z+\bar{w}+o(1))e^{\,-n-\sqrt{n}(z+\bar{w})}\cdot e^{\,-\,|\,z\,|^{\,2}/2\,-\,|\,w\,|^{\,2}/2},$$
which now closely resembles the left hand side in \eqref{wopp}.
By Theorem \ref{TT4} (or Theorem \ref{ginibref}) we know that
$$\tilde{K}_n(z,z)\to G(z,z)F(z+\bar{z})=F(2\re z),\quad (n\to\infty).$$
As a consequence,
$$\tilde{s}_n(x)e^{\,-n-\sqrt{n}\,x}\to F(x)e^{\,x^{\,2}/4},\quad (n\to\infty),\qquad x\in\R.$$
By analytic continuation, one can now recover the limit in \eqref{wopp}.

The convergence in \eqref{wopp} has been proved for the sections corresponding to more general entire functions.
In the monograph \cite{ESV}, the authors consider the Mittag-Leffler function $E_{1/\lambda}$ as well as a class denoted
"$\mathcal{L}$-functions'', while Edrei has supplied even further examples, see e.g. \cite{E}. In each case the authors
prove a suitably rescaled version of the limit in \eqref{wopp}.

To interpret the above results in terms of our Theorem \ref{TT4}, one chooses a suitable radially symmetric potential $Q$.
For example, one chooses $Q(\zeta)=E_{1/\lambda}(|\,\zeta\,|^{\,2})$ in case of the Mittag-Leffler function alluded to above.
Expressing the
kernel $\bfK_n$ in terms of the orthogonal polynomials (as in \eqref{E1.8}) and
rescaling about a boundary point of the droplet, one can apply Theorem \ref{TT4} and recover the asymptotic behaviour of the sections.

\subsection{Moving points in the bulk} \label{thor}
We have hitherto been occupied with scaling limits at the boundary, and at bulk singularities.
Our methods do however also apply equally well to the more familiar case of a "regular'' point $p$ in the bulk, i.e., a point
where $\Lap Q(p)>0$.

For certain applications (e.g. in \cite{AOC,AKMW}) it is advantageous to allow the point $p$ to
vary with $n$. We shall thus in general work with sequences $p=(p_n)_1^\infty$ rather than points. Let us write
\begin{equation*}\label{distb}\delta_n=\delta_n(p)=\dist\left(p_n,\d S\right).\end{equation*}
We say that
$p$ belongs to the \textit{bulk regime} if all $p_n$ are in $\drop$ and
 $\liminf_{n\to\infty}\sqrt{n}\delta_n=\infty$.

As customary, let $\{\zeta_j\}_1^n$ denote a random sample from the ensemble associated with $Q$
and let $\config_{n}=
\{z_j\}_1^n$, where (for any fixed $\theta\in\R$)
\begin{equation*}\label{OOO}z_j=e^{-i\theta}\sqrt{n\Delta Q(p_n)}\lpar \zeta_j-p_n\rpar,\quad j=1,\ldots,n.
\end{equation*}
We write $K_n$ for the correlation kernel of the rescaled process.

\begin{thm} \label{RBB} Let $p=(p_n)$ be a sequence in the bulk regime and
assume that $Q$ is real analytic and strictly subharmonic
in a neighbourhood of $S$.
If one rescales about $p$ according to \eqref{E1.3}, one has for all $k$
$$R_{n,k}(z_1,\ldots,z_k)\to \det\lpar G(z_i,z_j)\rpar_{i,j=1}^k\qquad \text{as}\quad n\to\infty,$$
with locally uniform convergence on $\C^k$. (This holds both in the free boundary and hard edge cases.)
\end{thm}

\begin{proof}\label{regbulk}
 Let $K=\Psi G$ denote a limiting kernel in Theorem \ref{cpthm}.
If $p=(p_n)$ is in the bulk regime, then the estimate in Theorem \ref{T2} shows that there is a positive
constant $c$ such that
\begin{equation}\label{decayy}R_n(z):=K_n(z,z)=1+\Ordo\left(e^{-cn\delta_n^2}\right)=1+o(1),\quad (n\to\infty).\end{equation}
Hence $R_n\to 1$ uniformly on compact sets, which means that
$\Psi\equiv 1$. We have shown that every limiting kernel equals to $G$, as desired.
\end{proof}

\begin{rem} The theorem holds also when $Q$ is replaced by a smooth perturbation
$\tilde{Q}_n=Q+h/n$ where $h$ is a fixed smooth function. Indeed, it suffices
to prove that the estimate \eqref{decayy} holds in this more general case.
To see this, one can use
the method in the appendix of \cite{AHM3}.
\end{rem}

\subsection{Ward's equation and the mass-one equation for $\beta$-ensembles} \label{OCPs}
Consider a potential $Q$ satisfying the standing assumptions in Section \ref{subsec12} and fix a number
$\beta>0$. Let us consider the probability measure on $\C^n$ defined by
\begin{equation*}d\Prob_n^{\,\beta}(\zeta):=\frac 1 {Z_n^{\,\beta}}\, e^{\,-\,\beta\,\Ham_n(\zeta)},\quad
Z_n^{\,\beta}:=\int_{\C^n}  e^{\,-\,\beta\,\Ham_n(\zeta)}\, dV_n(\zeta),\end{equation*}
where $\Ham_n(\zeta)=-\sum_{j\ne k}\log\babs{\,\zeta_j-\zeta_k\,}+n\sum Q(\zeta_j)$ is the Hamiltonian in the external electric field $Q$. (The parameter $\beta$ is sometimes interpreted as an inverse temperature; the case $\beta=1$ studied
hitherto can be interpreted as the statement that the external field is magnetic.)

Denote by $(\zeta_j)_1^n$ a point picked randomly with respect to $\C^n$ and write
$\{\zeta_j\}_1^n$ for the corresponding unordered configuration. We denote the $k$-point function
of this process by the superscript "$\beta$''. Thus
$$\bfR_{n,k}^{\,\beta}\left(\zeta_1,\ldots,\zeta_k\right)=\lim_{\eps\downarrow 0}
\frac {\Prob_n^{\,\beta}\left(\,\bigcap_{j=1}^k\left\{\,N_{D(\zeta_j;\eps)}\ge 1\,\right\}\,\right)}
{\eps^{2k}},\qquad N_D:=\#\left\{\,j\,;\,\zeta_j\in D\,\right\}.$$
The function $\bfR_{n,k}^{\,\beta}$ can not be written as a determinant when $\beta\ne 1$.
However, we have the following version of Ward's identity. The proof in Section \ref{theidentity} works equally well for the present, more general situation.

\begin{thm}\label{beta1} Put, for a test-function $\psi\in \coity(\C)$,
$$W_n[\psi]=\beta\left(I_n[\psi]-II_n[\psi]\right)+III_n[\psi],$$
where $I_n$, $II_n$, and $III_n$ are as in Section \ref{theidentity}. If $Q$ is $C^2$-smooth near
$\supp\psi$, then $\Expe_n^{\,\beta}W_n^+[\psi]=0$.
\end{thm}

Now fix a point $p\in S$ and a real parameter $\theta$ and rescale about $p$ according to
$$z_j=e^{\,-i\theta}\sqrt{n\Lap Q(p)}\,\left(\zeta_j-p\right),\qquad j=1,\ldots,n.$$
We denote by $\config_n^{\,\beta}:=\{z_j\}_1^n$ the rescaled process and write
$R_{n,k}^{\,\beta}(z_1,\ldots,z_k):=\bfR_{n,k}^{\,\beta}(\zeta_1,\ldots,\zeta_k)$ for the
joint intensities.
We also define the \textit{Berezin kernel}
of the process $\config_n^{\,\beta}$ by
$$B_n^{\,\beta}(z,w):=\frac {R_{n,1}^{\,\beta}(z)R_{n,1}^{\,\beta}(w)-R_{n,2}^{\,\beta}(z,w)}
{R_{n,1}^{\,\beta}(z)}.$$
Notice that $B_n^{\,\beta}(z,z)=R_{n,1}^{\,\beta}(z)$ and $\int_\C B_n^{\,\beta}(z,w)\, dA(w)=1$.
The kernel $B_n^{\,\beta}$ can be interpreted in terms of insertion of a point charge, as explained
in Section \ref{bkamo} in the case $\beta=1$.

We can also rescale in the Ward identity in Theorem \ref{beta1}.
The proof of the following theorem is merely a repetition of the argument in Section \ref{funo}.

\begin{thm}\label{beta2} If $p$ belongs to some neighbourhood of $S$ in which $Q$ is strictly
subharmonic and $C^2$-smooth, then
\begin{equation}\label{ptl}\dbar C_n^{\,\beta}(z)=R_n^{\,\beta}(z)-1-\frac 1 \beta\, \Lap_z
\log R_n^{\,\beta}(z)+o(1),\end{equation}
where
$$C_n^{\,\beta}(z):=\int_\C\frac {B_n^{\,\beta}(z,w)}{z-w}\, dA(w),\quad R_n^{\,\beta}:=R_{n,1}^{\,\beta}.$$
\end{thm}

We do not know whether it is possible to pass to the limit as $n\to\infty$ in \eqref{ptl}, but for the
sake of argument, let us temporarily assume that we can define a limiting Berezin kernel $B^{\,\beta}$.
Letting $n\to\infty$ in \eqref{ptl}, one then formally obtains the following generalization of Ward's equation
\begin{equation}\label{betawa}\dbar_z\int_\C\frac {B^{\,\beta}(z,w)}{z-w}\, dA(w)=
B^{\,\beta}(z,z)-1-\frac 1 \beta\Lap_z\log B^{\,\beta}(z,z),\quad z\in\C.\end{equation}
This more general equation can easily be transformed to the case $\beta=1$ by the linear scaling
\begin{equation}\label{we}B(u,v)=B^{\,\beta}(z,w),\quad u=\sqrt{\beta}\,z,\, v=\sqrt{\beta}\,w.\end{equation}
The result is the following.
\begin{prop} Suppose that $B^{\,\beta}$ solves \eqref{betawa}. Then the kernel $B$ in \eqref{we}
solves Ward's equation (with $\beta=1$).
\end{prop}

We do not know whether the presumptive kernels $B^{\,\beta}$ would be non-negative, so speaking about
"mass-one'' could possibly be misleading. However, if we assume that
$\int_\C B^{\,\beta}(z,w)\, dA(w)=1$,
then the corresponding kernel $B$ in \eqref{we} satisfies
the "mass-$\beta$ equation'': $\int_\C B(u,v)\, dA(v)=\beta$.
Note also that our solution of Ward's equation in the case $\beta=1$ depends
on analyticity of solutions.

The case of large $\beta$ is associated to several conjectures of physical relevance, notably the "Hall effect'' and "Abrikosov's lattice'', see e.g \cite{CFTW,Se} and the references there.


\begin{rem} Consider the potential $Q_\lambda(\zeta)=|\zeta|^{2\lambda}$ where $\lambda>1$ and rescale about the bulk
singularity $p=0$. The argument in Section \ref{frej} shows that Ward's equation for the system $\config_n^{\,\beta}$ takes the form
$$\dbar_z C_n^{\,\beta}(z)=R_n^{\,\beta}(z)-\lambda^2\babs{z}^{2(\lambda-1)}-\frac 1 \beta\, \Lap_z
\log R_n^{\,\beta}(z)+o(1).$$
\end{rem}

\appendix
\section{Convergence of point processes} \label{ARPF}

A \textit{configuration} $\{z_i\}$ is a finite or countably infinite "set'' of points
$\zeta_i\in\C$ with repetitions allowed. The configuration is said to be \textit{locally finite} if
for any compact set $K\subset \C$ there are at most finitely many $i$ such that $z_i\in K$.

Let $\mathcal{X}$ denote the set of all locally finite configurations $\{z_j\}$ in
$\C$. Let $\calB$ be the $\sigma$-algebra generated by all "cylinder sets'' $C_n^{\,B}=\left\{\,z\in \mathcal{X};\, N_B(z)=n\,\right\}$ where $B\subset\C$ is a Borel set
and $$N_B(z)=\#\left\{\,j;\, z_j\in B\,\right\}.$$
By a \textit{random point field}, we mean a probability measure $\Prob$ on the measure space $\left(\mathcal{X},\mathcal{B}\right)$.

Now suppose that $\config=\left(\mathcal{X},\mathcal{B},\Prob\right)$ is a random point field. We write
$\config=\{z_j\}$ for a random sample from this distribution.

We say that a locally
integrable function
$\rho_k:\C^k\to[0,+\infty)$ is a \textit{$k$-point intensity function} for $\config$ if
$$\Expe\left(\prod_{j=1}^m\frac {N_{B_j}!}{(N_{B_j}-k_j)!}\right)=\int_{B_1^{k_1}\times\cdots\times B_m^{k_m}}\rho_k\, dV_k$$
for all bounded Borel sets $B_1,\ldots,B_m\subset\C$ and all positive integers $k_1,\ldots,k_m$ with $k_1+\cdots+k_m=k$.

Note that a point field does not necessarily possess intensity functions.

Now let $\config_n=\{z_j\}_1^n$ be a sequence of point processes in $\C$ with intensity functions
$R_{n,k}$ (see \eqref{E1.2}). We shall say that $\config_n$ \textit{converges} to a point field $\config$ as
$n\to \infty$ if $\config$ has intensity functions $R_k$ of all orders and $R_{n,k}\to R_k$ as
$n\to\infty$ locally uniformly as $n\to\infty$. (Exception: In the case of a sequence of hard edge processes
$\config_n$, we say that they converge to $\config$ if the corresponding intensities converge with bounded almost everywhere convergence.)

\subsection{Existence of limiting point fields} \label{exi} We will use the following result due to Lenard \cite{L2}. See also \cite{S}, Theorem 1, p. 926.

\begin{thm} \label{lenard} A sequence of locally integrable functions $\rho_k:\C^k\to [0,+\infty)$ can be represented as the intensity functions
of some random point field $\config$ in $\C$ if and only if
\begin{enumerate}[label=(\roman*)]
\item \label{lena1} $\rho_k\lpar z_{\pi(1)}\,,\,\ldots\,,\,z_{\pi(k)}\rpar=\rho_k\lpar z_1\,,\,\ldots\,,\,z_k\rpar$ for all permutations $\pi$ of $\left\{\,1\,,\,\ldots\,,\,k\,\right\}$,
\item \label{lena2} For any sequence $\fii_k:\C^k\to \R$ of compactly supported bounded Borel functions and any $N\ge 0$ the following implication holds:
    \begin{equation}\label{b2}\fii_0+\sum_{k=1}^N
    \sum_{i_1\ne\ldots\ne i_k}\fii_k \lpar z_{i_1}\,,\,\ldots\,,\,z_{i_k}\rpar\,\ge\, 0\quad \text{for all }\left\{\,z_j\,\right\}\in \mathcal{X}\end{equation}
    \begin{center}implies\end{center}
    \begin{equation}\label{positivity}
    \fii_0+\sum_{k=1}^N\int_{\C^k}
    \fii_k\,\cdot\,\rho_k\,d V_k\ge 0.
    \end{equation}
\end{enumerate}
\end{thm}

A couple of remarks are in order. The second sum in \eqref{b2} are over all sequences of distinct indices appearing in the configuration
$\{z_j\}$.
Since the supports of the $\fii_k$ are compact, only finitely many terms in the sum are nonzero.

Consider the modified condition
\begin{equation}\label{b2p}\fii_0+\sum_{k=1}^N\sum_{\pi\in S_{k,N}}\fii_k \lpar z_{\pi(1)}\,,\,\ldots\,,\,z_{\pi(k)}\rpar\ge 0\quad \text{for all }z\in\C^N,
\end{equation}
where $S_{k,N}$ is the set of injective functions $\{1,\ldots,k\}\to\{1,\ldots,N\}$.

Let us say that the sequence $\rho_k$ has \textit{property $P\lpar N\rpar$} if \eqref{b2p} implies \eqref{positivity}.
It is easy to see that condition (ii) of Lenard's theorem holds if $\left\{ \rho_k \right\}$ has property $P\lpar N\rpar$ for all $N\ge 0$ (and all $\{\fii_k\}_0^N$). Indeed, if the assumption \eqref{b2} holds for all locally finite configurations $\left\{z_j\right\}$, it will in particular hold for configurations of length $N$. I.e. \eqref{b2p} holds. Thus by $P\lpar N\rpar$ we can conclude \eqref{positivity}, which shows that (ii) holds.

The $n$-point process $\{\zeta_j\}_1^n$ associated with a potential $Q$
is not a
random point field, since the sample space
is not $\mathcal{X}$. However, for fixed $n$
the functions $\varrho_{n,k}:=\bfR_{n,k}$ (or even $\bfR_{n,k}^{\,\beta}$) satisfy condition \ref{lena1} of Lenard's existence theorem (which is immediate) and also that $\{\zeta_j\}$ has property $P\lpar\,N\,\rpar$ whenever $n\ge N$ (take $\Expe_N$ of the random variable in the left hand side of
 \eqref{b2p} and use the non-negativity of the $\varrho_{n,k}$).

Rescaling about some point $z_0$, we find that the functions $\rho_{n,k}:=R_{n,k}$ satisfy the condition \ref{lena1} in Theorem \ref{lenard}, and also has property $P\lpar N\rpar$ whenever $n\ge N$. From this we conclude that if a locally integrable limit
$\rho_k=\lim_{n\to\infty}\rho_{n,k}$ exists, with bounded almost everywhere convergence on compact subsets of $\C$,
then the limiting
functions $\rho_k$ must satisfy conditions \ref{lena1} and \ref{lena2} of Theorem \ref{lenard}. Consequently they will be intensity functions
of some random point field.

\subsection{Uniqueness of limiting point fields} \label{uni} The following theorem, also due to Lenard, gives a convenient condition for checking the uniqueness of
a point field in $\C$.

\begin{thm} Let $\rho_k:\C^k\to[0,+\infty)$ be a sequence of intensity functions with respect to some random point field. Suppose that there is a number $c$ such that
\begin{equation}\label{len2}\sup_{E^k}\rho_k\le \lpar ck^{2}\rpar^{k}.\end{equation}
We then have uniqueness of the random point field in $\C$ having $\rho_k$ for correlation functions.
\end{thm}

For a proof, see \cite{L1}, p. 42.

The condition \eqref{len2} is certainly satisfied for a determinantal point field if we can
prove that its correlation kernel $K$ is uniformly bounded. For if $\babs{K}\le C$ on $\C^2$, then
the corresponding intensity functions satisfy
$R_{k}\le C^{\,k}k^{\,k/2}$ by the Hadamard inequality for determinants. Since
$\babs{K(z,w)}^2\le K(z,z)K(w,w)$ this gives
the following simple uniqueness criterion.

\begin{cor}\label{easytouse} Let $K_n$ be correlation kernels of the processes $\config_n$. If $K_n\to K$
locally uniformly on $\C^2$ and if $R(z):=K(z,z)$ is bounded on $\C$, then $K$ is the correlation kernel
of a unique point field $\config$ in $\C$, and $\config_n$ converges to $\config$ as point fields.
The same conclusion holds if $K_n\to K$ almost everywhere with bounded convergence on the diagonal in $\C^2$.
\end{cor}

\subsection*{Acknowledgements}
We would like to thank Seong-Mi Seo and Hee-Joon Tak for careful reading and much appreciated help with improving this manuscript. We thank Aron Wennman for many useful discussions. In particular, he communicated the proof of Lemma \ref{ln}.


\end{document}